\documentclass[reqno]{amsart} 
\usepackage[dvipsnames,usenames]{color} 
\usepackage{enumitem,accents}
\usepackage{pdfsync}
\usepackage{hyperref}
\usepackage{ifpdf} 
\usepackage[mathscr]{euscript} 
\usepackage{pgf}
\usepackage{tikz}

\usepackage{cite} 
\usepackage{graphicx}  
\usepackage[all]{xy} \xyoption{arc} \xyoption{color}
\usepackage{amsmath} 
\usepackage{amsthm} 
\usepackage{amsfonts}  
\usepackage{amssymb} 
\usepackage{verbatim}  
 
\numberwithin{equation}{section} 

\newtheorem{defn}{Definition} \numberwithin{defn}{section} 
\newtheorem{thm}{Theorem} \numberwithin{thm}{section} 
\newtheorem{lem}{Lemma} \numberwithin{lem}{section} 
\newtheorem{prop}{Proposition} \numberwithin{prop}{section} 
\numberwithin{equation}{section} 
\theoremstyle{remark} 
\newtheorem{rem}{Remark}

\newcommand{\notation}[3]{ $#1$ &\parbox{.75\textwidth}{#2\dotfill}&\pageref{#3}}
\newcommand{\set}[1]{ \{#1\} } 
 
\newcommand{\bset}[1]{ [#1] } 
\newcommand{\bbset}[1]{\big[#1\big]} 
\newcommand{\Bbset}[1]{\Big[#1\Big]} 
\newcommand{\Bset}[1]{\left[#1\right]} 
\newcommand{\pset}[1]{ (#1) } 
\newcommand{\bpset}[1]{\big(#1\big)} 
\newcommand{\Bpset}[1]{\Big(#1\Big)} 
\newcommand{\Pset}[1]{ \left(#1\right) } 
\newcommand{\norm}[1]{ \|#1\| } 
 
\newcommand{\Bnorm}[1]{\Big \|#1\Big\| } 
 
\newcommand{\abs}[1]{ |#1| } 
\newcommand{\Abs}[1]{\left |#1\right| }

\newcommand{\cp}[1]{,_{#1}}

\def\dist{\operatorname{dist}} 
\def\supp{\operatorname{spt}} 
\def\Sup{\sup_{t\in[0,T]}}

\def\p{\partial}  
\def\hd{\bar{\partial}} 
\def\Div{\operatorname{div}} 
\def\curl{\operatorname{curl}}

\def\err{\check{b}}
\def\berr{\bar{b}}
\def\rerr{\varrho \err}

\def\vc{\check{v}} 
\def\ec{\check{\zeta}_{\epsilon}} 
\def\nc{\check{n}} 
\def\Ac{\check{A}} 
\def\ac{\check{a}} 
\def\gc{\check{g}} 
\def\Jc{\check{J}}
\def\hc{\check{h}}
\def\Kc{\check{\mathcal{K}}}
\def\JTc{\check{\mathcal{J}}_{t}}
\def\Htc{\check{H}_\epsilon}
\def\jtt{\check{\jmath}}
\def\htt{\check{\mathscr{V}}_t}

\def\NB{\ddddot{\mathcal{N}}_0} 
\def\PB{\ddddot{\mathcal{P}}}
\def\RB{\ddddot{\mathcal{R}}} 

\def\EE{E^{\mu}(t)}
\def\ME{\ddot{\mathcal{N}}_0}
\def\PE{\ddot{\mathcal{P}}}
\def\RE{\ddot{\mathcal{R}}} 

\def\EK{E^\kappa (t)} 
\def\NK{\mathcal{N}_0}
\def\PK{{\mathcal{P}}}
\def\RK{\mathcal{R}} 

\def\EZ{\mathscr{E}(t)} 
\def\MZ{\mathscr{M}_0}

\def\ES{E^\sigma\!(t)}
\def\MS{\dddot{\mathcal{N}}_0}
\def\PS{\dddot{\mathcal{P}}}
\def\RS{\dddot{\mathcal{R}}}

\def\vz{\breve{v}} 
\def\ez{\breve{\eta}} 
\def\nz{\breve{n}} 
\def\Az{\breve{A}} 
\def\az{\breve{a}} 
\def\gz{\breve{g}} 
\def\Jz{\breve{J}}
\def\fz{\breve{f}}
 \def\ellz{\breve{\ell}} 
\def\elz{\breve{l}} 
\def\k{\mathbf{k}} 

\def\Fz{\breve{F}}

\def\vt{\tilde{v}} 
\def\et{\tilde{\eta}} 
\def\nt{\tilde{n}} 
\def\At{\tilde{A}} 
\def\at{\tilde{a}} 
\def\gt{\tilde{g}} 
\def\Jt{\tilde{J}}

\def\LM{\Lambda_{\mu}}
\def\LE{\Lambda_{\epsilon}}

\def\EM{E^\epsilon (t)} 
\def\MM{\dot{\mathcal{N}}_0}
\def\PM{\dot{\mathcal{P}}}
\def\RM{\dot{\mathcal{R}}}

\def\vr{\mathring{v}} 
\def\rro{\mathring{\varrho}}
\def\er{\mathring{\zeta}_{\epsilon}} 
\def\Ar{\mathring{A}_{\epsilon}} 
\def\ar{\mathring{a}_{\epsilon}} 
\def\Jr{\mathring{J}_{\epsilon}}
\def\gr{\mathring{g}_{\epsilon}} 
\def\nr{\mathring{n}_{\epsilon}} 
\def\gfrk{\mathring{\mathfrak{g}}_{\epsilon,\local}}
\def\Kr{\mathring{\mathcal{K}}}
\def\hrm{\mathring{h}^{\mu}}
\def\hr{\mathring{h}}
\def\br{\mathring{b}}
\def\crr{c^{\mu}(t)}

\def\Aru{\bset{\Ar}}

\def\Am{\check{A}_{\epsilon}}
\def\am{\check{a}_{\epsilon}} 
\def\Jm{\check{J}_{\epsilon}}
\def\gm{\check{g}_{\epsilon}} 
\def\nm{\check{n}_{\epsilon}}

\def\Amu{\bset{\Am}} 
\def\amu{\bset{\am}}

\def\Hm{\check{H}_{\epsilon}}

\def\gfk{\check{\mathfrak{g}}_{\epsilon}}
\def\gbfk{\bar{\mathfrak{g}}_{\epsilon,\local}}
\def\Hk{\tilde{H}_{\kappa}}
\def\Htk{\tilde{H}_{\kappa}}

\def\JT{\check{\mathcal{J}}_t}

\def\ebm{\bar{\zeta}_{\epsilon}} 
\def\Abm{\bar{A}_{\epsilon}} 
\def\Abmu{\bset{\bar{A}_{\epsilon}}} 
\def\abm{\bar{a}_{\epsilon}} 
 
\def\Jbm{\bar{J}_{\epsilon}}

\def\gbm{\bar{g}_{\epsilon}} 
\def\gbmu{\bset{\bar{g}_{\epsilon}}} 
\def\nbm{\bar{n}_{\epsilon}}

\def\CT{\mathcal{C}_T(\MB)} 
\def\XT{\mathbf{X}_T}
 
\def\MB{\mathcal{M}} 

\def\vb{ \bar{v}}

\def\Jb{\bar{J}}

\def\rb{\bar{\varrho}}
\def\Kb{\bar{\mathcal{K}}}
\def\Gb{\bar{\mathcal{G}}}

\def\hb{\bar{h}}
\def\jb{\bar{\jmath}}

\def\Gre{\mathbf{\Gamma}}
\def\gre{\mathbf{g}}
\def\epr{{\boldsymbol{\epsilon}}}

\def\mur{{\boldsymbol{\mu}}}
\def\tre{\boldsymbol{\theta}}
\def\Tre{\boldsymbol{\tau}_\alpha}
\def\Hre{\boldsymbol{H}}
\def\Kre{\boldsymbol{K}}
\def\kre{\boldsymbol{k}}
\def\jre{\boldsymbol{j}}
\def\kre{\boldsymbol{k}}
\def\Are{\boldsymbol{\mathscr{A}}}

\def\phrep{\boldsymbol{\varphi}}

\def\JRE{\boldsymbol{\mathfrak{J}}}

\def\dre{\boldsymbol{d}}
\def\fre{\boldsymbol{f}}

\def\Jre{\boldsymbol{J}}
\def\vre{\boldsymbol{\rho}_0}
\def\bv{\boldsymbol{v}}
\def\bV{\boldsymbol{V}}
\def\rre{\boldsymbol{\varrho}_0}
\def\Ore{\mathbf{\Omega}}
\def\Gre{\mathbf{\Gamma}}
\def\gre{\mathbf{g}}
\def\epr{{\boldsymbol{\epsilon}}}
\def\lmr{{\boldsymbol{\lambda}}}
\def\mur{{\boldsymbol{\mu}}}
\def\ere{\boldsymbol{\eta}}
\def\tre{\boldsymbol{\theta}_\local}
\def\Nre{\boldsymbol{N}}
\def\Tre{\boldsymbol{\tau}_\alpha}
\def\Hre{\boldsymbol{H}}
\def\jre{\boldsymbol{j}}
\def\Kre{\boldsymbol{K}}

\def\re{\check{\varrho}}

\def\hbe{\bar{h}^{\mu}}

\def\j{\mathfrak{j}} 
\def\I{\mathcal{I}} 
\def\II{\mathfrak{i}} 
 
\def\i{\mathfrak{r}}
\def\ellt{\tilde{\ell}} 
\def\elt{\tilde{l}}

\def\counter{l} 
\def\local{{l}} 
\def\UL{U_\local} 
\def\VL{\mathcal{B}_\local} 
\def\VLP{\VL^+} 
\def\DL{\mathcal{D}_\local} 
\def\thetal{\theta_\local}

\def\xl{\xi_\local}

\title[Compressible Euler with   surface tension and the zero surface tension limit]{Well-posedness of the
  free-boundary compressible 3-D Euler equations with surface tension and the zero surface tension limit}

\author[D. Coutand]{Daniel Coutand} \address{CANPDE, Maxwell Institute for Mathematical Sciences and department of Mathematics, Heriot-Watt University, Edinburgh, EH14 4AS, UK} \email{D.Coutand@ma.hw.ac.uk}

\author[J. Hole]{Jason Hole} 
\author[S. Shkoller]{Steve Shkoller} \address{Department of Mathematics, University of California, Davis, CA 95616, USA}

\email{jhole@math.ucdavis.edu} \email{shkoller@math.ucdavis.edu}


\ifpdf  \fi
 
\begin{document}
\begin{abstract}
We prove that the 3-D compressible Euler equations with surface tension along the  moving free-boundary 
are well-posed.   Specifically, we consider isentropic dynamics and consider an equation of state, modeling a liquid, 
given by Courant and Friedrichs \cite{CF48} as 
$p(\rho) =  \alpha  \rho^ \gamma - \beta$ for consants $\gamma >1$ and
$ \alpha , \beta > 0$.   The analysis is made difficult by two competing nonlinearities associated with the potential energy: {\it compression} in the bulk, 
and {\it surface area dynamics} on the free-boundary.   Unlike the analysis of the incompressible Euler equations, wherein boundary regularity controls
regularity in the interior, the compressible Euler equation require the additional analysis of  nonlinear wave equations generating sound waves.   An
existence theory is developed by a specially chosen parabolic regularization together with the vanishing viscosity method.   The artificial parabolic term
is chosen so as to be asymptotically consistent with the Euler equations in the limit of zero viscosity.    Having solutions for the positive surface tension
problem, we proceed to obtain a priori estimates which are independent of the surface tension parameter.  This requires choosing initial data which
satisfy the Taylor sign condition.  By passing to the limit of zero surface tension, we  prove the well-posedness of the compressible Euler system
without surface on the free-boundary, and without derivative loss.
\end{abstract}

\maketitle {\small  
  \tableofcontents}

\section{Introduction} \label{sec:introduction}

\subsection{The compressible Euler equations in Eulerian variables } 

The  compressible Euler equations with moving free-boundary are given by the following system:
\begin{subequations}
	\label{prob: Euler} 
	\begin{alignat}
		{4} \label{momentum}
		\partial_t (\rho u)+ \operatorname{div} ( \rho u \otimes u + p \operatorname{Id} ) &=0&\quad&\text{in }\Omega(t),\\
		\label{mass} 
		\partial_t\rho +\Div(\rho u)&=0&\quad&\text{in }\Omega(t),\\
		\label{Laplace-Young} p&= \sigma H(t)&&\text{on }\Gamma(t),\\
		\label{boundary-moves-with-velocity} \mathcal{V}(\Gamma(t))&=u\cdot n(t) &&\\
		\label{initial velocity} (u,\rho)&=(u_0,\rho_0)&&\text{on } \Omega(0),\\
		\label{fix domain} \Omega(0)&=\Omega \,,
	\end{alignat}
\end{subequations}
where  $ \Omega (t)$ denotes  an open and  bounded subset  of $\mathbb{R}^3$,   $\Gamma(t)=
\partial\Omega(t)$ is the moving free-boundary, and  $t \in [0,T]$ denotes time. 
We use the notation $\mathcal{V}(\Gamma(t))$ for  the normal velocity of boundary $\Gamma(t)$, which is  equal to the normal component of the
fluid velocity $u \cdot n$,  where
$n(t)$ is the outward-pointing unit normal  to $\Gamma(t)$,
$u=\pset{u_1,u_2,u_3}\label{n:u}$ denotes the 
velocity field, $p$ denotes the  pressure, and
$\rho\label{n:rho}$ denotes the density.

The first two equation are conservation laws for momentum and mass.  The boundary condition 
 \eqref{Laplace-Young} is often referred to as the  Laplace-Young
 condition,  stating that the fluid stress is proportional to the  mean curvature $H(t)$ of the moving surface,
 the proportionality constant defining the surface tension parameter $  \sigma $.  The last two equations
 provide the initial conditions for the dynamics.
 
In order to model the motion of a compressible liquid, we use the equation-of-state given by
Courant and Friedrichs \cite{CF48} as
\begin{align}
	\label{isentropic} p(x,t)=\alpha \rho(x,t)^\gamma-\beta\ \ \ \text{for}\ \gamma>1, 
\end{align} 
where $\alpha>0$ and $\label{n:beta}\beta>0$. For convenience, we set
$\alpha=1$. \footnote{Using \eqref{isentropic}, liquid water is modeled 
  using the values $\gamma=7$, $\alpha=3001$ and $
  \beta=3000$.}

Using the equation of state   (\ref{isentropic}), the momentum equations (\ref{momentum}) and
Laplace-Young boundary condition (\ref{Laplace-Young}) are equivalently written as 
\begin{subequations}
  \begin{alignat}
    {2} \rho[\p_tu +(u\cdot D)u]+D\rho^{\gamma}&=0&\quad&\text{in}\ \Omega(t), \\
    \rho^\gamma &=\beta+\sigma H&&\text{on}\ \Gamma(t).
  \end{alignat}
\end{subequations}

We assume  that the initial  density function is strictly positive and that
\begin{align*}
	\rho_0 \ge \lambda>0\ \ \text{in}\ \overline{\Omega}	.
 \end{align*}
In the absence of surface tension,   we further require the initial pressure function $p_0$ to
satisfy the {\it Taylor sign condition} (see, for example,  \cite{Taylor1950} and \cite{Ra1878}), given by 
\begin{align*}
0<\nu\le -\frac{\p p_0}{\p N}\ \ \text{on } \Gamma \,,
\end{align*}
where $N$ denotes the outward unit normal to $\Gamma$.    This is equivalent to
\begin{align}
  \label{TSC}
  0<\nu\le -\frac{\p \rho_0^\gamma}{\p n}\ \ \text{on } \Gamma \,.
\end{align}

\subsection{Prior results on the Euler equations with moving free-boundary} \label{sec:backgr-probl-prior}  

\subsubsection{The incompressible setting} There has been a recent
explosion of interest in the analysis of the free-boundary
incompressible Euler equations, particularly in irrotational form,
that has produced a number of different methodologies for obtain- ing
a priori estimates. The accompanying existence theories have relied
mostly on the Nash–Moser iteration to deal with derivative loss in
linearized equations when arbitrary domains are considered, or on
complex analysis tools for the irrotational problem with infinite
depth. We refer the reader to \cite{CS07 ,AM09, CS12,D.L05,H.L05,SZ08,
  W97,W99, ZZ08} for a partial list of papers on this topic.

\subsubsection{The compressible setting} The mathematical analysis of
moving hypersurfaces in the multidimensional compressible Euler
equations began with the existence and stability of the shock-front solution
initiated in \cite{M84} and extensively studied by \cite{FM,GM91,GMWZ,M01} 
(see the references in these articles for a thorough review of the literature in this area.)
More delicate than the non-characteristic case of the shock-front solution, the characteristic
boundary case is encountered in the study of vortex sheets or current
vortex sheets. This class of problems has been studied by \cite{CW08,CSP08,CSP09,T05,T9} (and see the references therein);
the linearization of the vortex-sheet problem produces 
derivative loss, similar to that experienced by
many authors in the incompressible flow setting (both irrotational
flows and flows with vorticity).

The  problem of the expansion of a compressible gas with the so-called physical vacuum singularity has been
studied in \cite{CLS09,CS11, JM10, CS10, JM11},  and is degenerate because of the vanishing of the density
function on the moving free-boundary.

For the model of a compressible liquid,  considered in this paper, the Euler equations are  uniformly hyperbolic thanks to the equation-of-state (\ref{isentropic}).  In the
absence of surface tension, an existence theory was  given in 
\cite{H.L05.1} using Lagrangian coordinates and  a Nash-Moser construction, but the estimates had derivative loss.   Using 
the theory of symmetric hyperbolic systems, the paper
\cite{T09} gave a different proof for the existence of solutions, also with derivative loss.     We prove well-posedness for the motion of
a compressible liquid with and without surface tension, and with no derivative loss.   We also establish asymptotic limit of zero surface tension.

\subsection{Fixing the domain and Lagrangian
  variables} \label{sec:fixing-doma-lagr} 
To transform the system (\ref{prob: Euler}) into Lagrangian variables, we let $\label{n:eta}\eta(x,t)$ denote the flow of a fluid particle $x$ at time $t$. Thus, 
\begin{align*}
	\p_t \eta=u\circ \eta\ \text{for}\ t>0\ \ \text{and}\ \ \eta(x,0)=x 
\end{align*}
where $\circ$ denotes composition, so that $[u\circ
\eta](x,t)=u(\eta(x,t),t)$. The \textit{flow map} $\eta$ induces
the following Lagrangian variables on $\Omega$:
\begin{alignat*}
	{2} v&=u\circ\eta\ &\label{n:Lagrangian variables}& \text{(\textit{Lagrangian velocity})},\\
	f&= \rho\circ \eta\ && \text{(\textit{Lagrangian density})},\\
	A&= [D\eta]^{-1}\ &&\text{(\textit{inverse of the deformation tensor})},\\
	J&= \det D\eta\ &\quad&\text{(\textit{Jacobian determinant})},\\
	a&= JA\ && \text{(\textit{cofactor matrix of the deformation tensor})}. 
\end{alignat*}
Using the notation defined below in Section~\ref{sec:notat-part-diff}, the Lagrangian version of equations (\ref{prob: Euler}) on the fixed domain $\Omega$ is given by 
\begin{subequations}
	\label{prob: pre-Lagrangian} 
	\begin{alignat}
		{4} f v^i_t+A_i^k f^\gamma\cp{k} &=0&\qquad&\text{in } \Omega\times(0,T],\\
		f_t+fA_i^j v^i\cp{j}&=0&&\text{in}\ \Omega\times(0,T],\label{L mass}\\
		f^\gamma &=\beta+\sigma H(\eta)&&\text{on}\ \Gamma\times(0,T],\\
		(\eta,v,f)|_{t=0}&=(e,u_0,\rho_0)&&\text{on}\ \Omega, 
	\end{alignat}
\end{subequations}
where $\Gamma=\p\Omega$ and $\label{n:e} e(x)=x$ denotes the identity
map on $\Omega$.

Since $J_t=a_i^jv^i\cp{j}$ by (\ref{formula: J_t}), we multiply \eqref{L mass} by $J$ and integrate in time for the identity 
\begin{align}
	\label{Lagrangian density} f=\rho_0 J^{-1}. 
\end{align} 
For $\sigma\ge0$ and $\gamma=2$, we use the identity (\ref{Lagrangian
  density}) for the Lagrangian density $f$ and
equivalently write the compressible Euler equations
\eqref{prob: pre-Lagrangian} 
as \begin{subequations}
	\label{Euler} 
	\begin{alignat}
		{4} \rho_0v^i_t+a_i^k(\rho_0^2J^{-2})\cp{k}&=0&\qquad&\text{in}\ \Omega\times(0,T],\label{Euler.m}\\
		\rho_0^2 J^{-2}&=\beta +\sigma H(\eta)&&\text{on}\ \Gamma\times(0,T],\label{Euler.bc}\\
		(\eta,v)|_{t=0}&=(e,u_0)&&\text{on}\ \Omega.
	\end{alignat}
\end{subequations}
For $\sigma>0$, we shall refer to the equations \eqref{Euler} 
  as the \textit{surface tension problem}.
 
For reference, we record that the momentum equations
\eqref{Euler.m} are equivalent to
\begin{align}
  \label{coupled equation}
v_t^i +2 A_i^k (\rho_0J ^{-1})\cp{k}=0\ \  \ \ \ \text{in } \Omega\times(0,T].
\end{align}

\begin{rem}
  The identity (\ref{Lagrangian density}) establishes that in
  Lagrangian variables the initial density $\rho_0$ is a parameter in
  the system of compressible Euler equations.
\end{rem}

\subsection{The higher-order energy functions
$E(t)$ and $\EZ$} \label{sec:higher-order-energy} 
While the physical energy  $\int_{\Omega} \bset{\rho_0 \frac{1}{2}\abs{v}^2 +
\rho_0^2 J ^{-1}+ \beta  J} +\sigma \mathcal{A} (t)$, $ \mathcal{A} (t)$ denoting the surface area of $\Gamma(t)$,
is a conserved quantity, it is far too weak for the  energy estimates methodology that we
employ.
  We instead define the higher-order energy functions $E(t)$ and $\EZ$ to
respectively correspond with the surface tension problem 
and the zero surface tension limit. Although neither $E(t)$ or $\EZ$ is
conserved, we will establish that each of $\Sup E(t)$ and  $\Sup\EZ$ is bounded on a
sufficiently small time-interval of existence $[0,T]$.

\subsubsection{The higher-order energy function for $\sigma>0$}
\label{sec:energy-funct-surf}

We define  \label{n:E}
the energy function $E(t)$ as
\begin{align}
	\label{s.energy function} E(t) &=1+\sum_{a=0}^5\norm{\p^a_t\eta(t)}_{5-a}^2+\abs{v_{ttt}\cdot n(t)}_1^2+\sum_{a=0}^2\abs{ \hd^2\!\p^a_t v\cdot n(t)}_{2.5-a}^2. 
\end{align}
We let $\label{n:M0}M_0\ge0$ denote a generic constant given by a polynomial function $P$ of $E(0)$:
\begin{align} 
	\label{s.con.M0} M_0=P\Pset{E(0)}. 
\end{align}
\subsubsection{The higher-order    energy function for $\sigma=0$}
We define the energy function $\EZ$ as
\begin{align}
	\label{z.energy function} 
        \begin{aligned}[b]
          \EZ =1 +
          \sum_{a=0}^7\norm{\p^a_t\eta(t)}_{4.5-\frac{1}{2}a}^2 
          + \sum_{a=0}^5\norm{\p^a_tJ(t)}_{4.5-\frac{1}{2}a}^2 
&+
\norm{\p^{6}_tJ(t)}_{1}^2.
        \end{aligned}
\end{align}
We let $\label{n:zM0} \MZ\ge0$ denote a generic constant given by a polynomial function $P$ of $\mathscr{E}(0)$:
\begin{align} 
	\label{z.con.M0} \MZ=P\Pset{\mathscr{E}(0)}. 
\end{align}

Section~\ref{sec:notation} explains the notation in \eqref{s.energy
  function} and \eqref{z.energy function}.
\begin{rem}\label{scaling}Corresponding with $\sigma=0$, time-derivatives in the higher-order energy function $\EZ$ 
  scale like one-half a space-derivative.
Since the scalar product of the momentum equations
\eqref{Euler.m} and the tangential vector $\eta \cp{\alpha}$
yields the identity $\rho_0v_t\cdot\eta
  \cp{\alpha} =-J(\rho_0^2 J ^{-2}) \cp{\alpha}$ in $\Omega$, the
  Laplace-Young boundary condition \eqref{Euler.bc} with $\sigma=0$
 provides that 
  \begin{align*}
    v_t\cdot\eta \cp{\alpha}=0\ \  \ \text{on } \Gamma.
  \end{align*}
Differentiating this
  identity with respect to time shows that 
  $\p_t$ scales like $(\p_x)^{1/2}$ in the zero
  surface tension limit of \eqref{Euler}.
\end{rem}
\subsection{The initial data
  $(\rho_0,u_0,\Omega)$} \label{sec:s.comp-cond} 
We assume that $\rho_0$
  and $u_0$ are given and sufficiently smooth. 
We require that the initial density $\rho_0$ satisfy \begin{align}
  \label{l.boundedness.r0}
    \rho_0\ge 2\lambda &>0\ \ \ \text{in } \overline{\Omega}. 
\end{align}
 Using the identities $\eta(0)=e$ and $J_t=a_r^sv^r\cp{s}$, we let
  $J_1$ be defined by
  \begin{align*}
    J_1&=\p_t\pset{J^{-2}}(0)=-2\Div u_0.
  \end{align*}
  We let $v_1$ and $v_2$ be the vectors
respectively given by
  \begin{align*}
    v_1^i=-2\p_i\rho_0\ \ \text{and}\ \
    v_2^i=-\rho_0^{-1}\bset{\p_i\pset{\rho_0^2J_1} + \p_ta_i^k(0)
      \pset{\rho_0^2}\cp{k}},
  \end{align*}
  where $\p_ta(0)$ is a smooth function of $Du_0$. 
We define, as a function of $\rho_0$ and $u_0$,
  \begin{align*}
    J_a=\p^a_t(J^{-2})|_{t=0} \ \ \text{ for $a=0,1,2,3.$}
  \end{align*}

\subsubsection{The case of positive  surface tension}
\label{sec:surf-tens-probl-1y} 
For $\sigma>0$, we assume that $\Omega$ is an $H^5$-class domain and
that $\rho_0$ and $u_0$ are in $H^4(\Omega)$. We define, as a function of $\rho_0$ and $u_0$,
  \begin{align*}
H_a=\p^a_t
    H(\eta)|_{t=0} \ \ \text{ for $a=0,1,2,3.$}
  \end{align*}
We require that the initial data satisfy the
following compatibility conditions:
	\begin{alignat}
		{2}\label{s.compatibility conditions}  \rho_0^2J_a&=\p^a_t\beta + \sigma H_a&\ \ &\text{on $\Gamma$ for $a=0,1,2,3.$} 
	\end{alignat}
For  $\beta$ as in \eqref{isentropic} and $\lambda$ as in
\eqref{l.boundedness.r0}, we  note that $a=0$ in
\eqref{s.compatibility conditions} yields
  \begin{align*}
\sigma H_0\ge 4 \lambda^2- \beta
.
  \end{align*}

\subsubsection{The case of zero surface tension}   
For $\sigma=0$, we assume that $\Omega$ is an $H^{4.5}$-class domain
and that    $\rho_0$ is in $H^{4.5}(\Omega)$ and $u_0$ is in
$H^4(\Omega)$.

We require that $\rho_0$ and $\Omega$ satisfy the
Taylor sign condition:
\begin{align}
  \label{Taylor}
  0< 2\nu \le -\frac{1}{\sqrt{g}} N^j a_i^j a_i^k \pset{\rho_0^2J ^{-2}}\cp{k}\big|_{t=0}\ \ \ \text{on } \Gamma.
\end{align}
We require that the initial data satisfy the
following compatibility conditions:
 	\begin{align}
		\label{z.compatibility conditions} 
\rho_0^2J_a=\p^a_t \beta \ \ \ \text{on } \Gamma\ \text{for }
a=0,\dots,6.
	\end{align}

        \begin{rem} 
          The twice-mean-curvature function $H(\eta)$ is not present in
          the zero surface tension limit of \eqref{Euler}.  Accordingly, there
          is no restriction on the curvature of the 
          initial surface $\Gamma$.
        \end{rem}

\subsection{Main Results} \label{sec:main-result} The main results of
this paper are the existence and uniqueness of solutions to the surface tension
problem  and its zero surface tension limit.
\begin{thm}
	[Existence and uniqueness for $\sigma>0$]\label{thm.s.main} Suppose that the initial data $(\rho_0,u_0, \Omega)$ verify
	\begin{enumerate}
		\item $M_0= P(E(0)) < \infty $,
		\item the lower-bound condition \eqref{l.boundedness.r0}, and 
		\item the compatibility conditions \eqref{s.compatibility conditions}. 
	\end{enumerate}
	Then for some $T>0$, there exists a solution to  $\eqref{Euler}$ on the
        time-interval $[0,T]$ such that  $\rho(t) \ge
 \lambda $ in  $\overline{\Omega}(t)$, $ \sigma H(t)>  -\beta$ on 
        $\Gamma(t)$, and 
	\begin{align*}  
		\Sup E(t)\le 2M_0. 
	\end{align*}
	
	Furthermore, the solution is unique if the initial data
is such that
	\begin{align*}
	\sum_{a=0}^6\norm{\p^a_t\eta(0)}_{6-a}^2 + \abs{v_{tttt}\cdot n(0)}_1^2 + \sum_{a=0}^3\abs{ \hd^2\!\p^a_t v\cdot n(0)}_{3.5-a}^2 < \infty.
	\end{align*}
\end{thm}
\begin{thm}
	[The zero surface tension limit]\label{thm.z.main}
        Suppose that the initial data $( \rho_0, u_0, \Omega)$ verify
	\begin{enumerate}
		\item $\MZ= P(\mathscr{E}(0)) < \infty $, 
		\item the lower-bound condition
\eqref{l.boundedness.r0}, 
                \item the Taylor sign condition \eqref{Taylor}, and 
		\item the compatibility conditions \eqref{z.compatibility conditions}. 
	\end{enumerate}
	Then for some $T>0$, there exists a solution to
        $\eqref{Euler}$ with $\sigma=0$
        on  the
        time-interval $[0,T]$ such that  $\rho(t) \ge  \lambda$ in $\overline{\Omega}(t)$,
        $-\frac{\p \rho^2(t)}{\p n(t)} \ge \nu$ on $\Gamma(t)$, and  
	\begin{align*}  
		\Sup \EZ\le 2\MZ. 
	\end{align*}
	
	Furthermore, the solution is unique if the initial data is such that 
	\begin{align*}
 \sum_{a=0}^9\norm{\p^a_t\eta(0)}_{5.5-\frac{1}{2}a}^2 
+
\sum_{a=0}^7\norm{\p^a_tJ(0)}_{5.5-\frac{1}{2}a}^2  
&+
\norm{\p^{8}_tJ(0)}_{1}^2  
  < \infty.
	\end{align*} 
\end{thm}

\begin{rem}
The proofs of Theorems~\ref{thm.s.main} and~\ref{thm.z.main}
        do not rely on our choice of $\gamma=2$ and as such are valid
        for general $\gamma>1$, since  by introducing new enthalpy-type
        variables for $\rho_0^{\gamma-1}$ and $J^{\gamma-1}$ in \eqref{prob: pre-Lagrangian}, 
we can reduce the case of general $\gamma > 1$ to the case that
$\gamma = 2$.  
\end{rem}

\subsection{Structure of the proofs of Theorems~\ref{thm.s.main} and
  \ref{thm.z.main}}
\label{sec:struct}

\subsubsection{Existence for the surface tension problem
  \eqref{Euler}}  
 An existence theory is obtained via the vanishing viscosity method, but with a very special choice of artificial viscosity
 which does not alter the transport-type structure of vorticity.   This is introduced in Section \ref{sec:parab-kappa-appr}.
The Euler equations \eqref{Euler.m}  yield
\begin{align*}
  \curl_{\eta}v_t =0,
\end{align*}
where the Lagrangian curl operator $\curl_\eta$ is defined below in
Section~\ref{sec:curl-diverg-oper}. By defining a
parabolic approximation of the Euler equations \eqref{Euler.m} which
preserves the homogeneous vorticity equation, we  ensure that the
vorticity estimates for the approximate problem are independent of the parabolic
approximation parameter $ \kappa $.
The parabolic approximate $\kappa $-problem defined in
Section~\ref{sec:parab-kappa-appr} maintains a
homogeneous vorticity equation.

The  structure of the Euler equation \eqref{Euler.m} shows that
an a priori estimate for $v_t$ provides an estimate for the gradient of $J$, which then
yields an estimate for  the gradient of  $ \operatorname{div} \eta$.   
As such,  yet another constraint on the choice of parabolic regularization operator is that
its presence must still permit estimates for $D \operatorname{div} \eta$ whenever estimates for $ v_t$ are known.
We choose our $\kappa$-parabolic operator so
that the parabolically-regularized momentum equations  have the form
\begin{align*}
  f + \kappa f_t =g,
\end{align*}
for $f$ equalling the gradient of $J$.    Such a structure allows us to obtain $ \kappa $-independent estimates for $f$, given estimates for $g$.
We must further add parabolic regularization to the surface tension boundary conditions  \eqref{Euler.bc} which also has the same structure,
in order to infer the maximal regularity of the free-boundary.

Our strategy is to obtain energy estimates for the highest number of time-derivatives in the problem, and then use the
structure of the momentum equations to infer estimates for the divergence.  In conjunction with curl-estimates and boundary
regularity, we boot-strap the full regularity of the velocity, its time-derivatives, and the flow map $\eta$.

In Section~\ref{sec:kappa independent estimates}, we  perform energy estimates in the fourth (highest)
time-differentiated approximate $\kappa $-problem.  From these energy estimates, we are able
to close estimates for $v_{tttt}(t)$ in $L^2(\Omega)$ and $v _{ttt}(t)$  in
$H^1(\Omega)$.    We then infer the optimal regularity of $v_{tt}$ and $v_{t}$.   This, in turn, enables us to get the
optimal regularity for $v$ and $\eta$.  All estimates are found to be independent of the artificial viscosity parameter $\kappa $.

In Section~\ref{sec:appr-kappa-probl}, 
we construct smooth solutions the approximate $\kappa$-problem.
We
utilize (the Lagrangian variables version of) the basic vector identity
$
  -\Delta=\curl\curl- D\Div
$
in order to replace 
gradient of the divergence  of $v$ with 
  the Laplacian of $v$, modulo lower-order terms. Additionally, we use that a sufficiently regular vector $\xi\in
\mathbb{R}^3$ satisfies 
\begin{align*}
N ^j A_r^j A_r^k \xi \cp{k}=
\sqrt{g} J^{-1} {(\curl_{\eta}
 \xi)\times n}
+\sqrt{g} J^{-1}  (\Div_{\eta} \xi )n+   b(\xi,\eta),
\end{align*}
where $b(\xi,\eta)$ is a first-order differential boundary operator
with respect to $\xi$ and a second-order differential boundary operator
with respect to $\eta$.   Since the regularity of the flow map $\eta$ is
dictated by the regularity of $v$ in our construction scheme, we employ a horizontal convolution-by-layers operator
(first introduced in \cite{CS07}) in order to view $b(v,\eta)$ as a
lower-order term; with the horizontal convolution-by-layers approximation in place, solutions can be found via the existence
theory for uniformly parabolic second-order equations.   We must then find estimates for such solutions which are
indeed independent of the convolution parameter and pass to the limit.

 In Section~\ref{sec:uniqueness-solutions}, we use the $\kappa$-independent a priori estimates established in
Section~\ref{sec:kappa independent estimates} and 
the construction of  solutions to our approximate $ \kappa $-problem  in 
Section~\ref{sec:appr-kappa-probl} to establish
the existence and uniqueness of solutions  to the surface tension problem
  \eqref{Euler}.

\subsubsection{The zero surface tension limit}  
In Section~\ref{sec:uniqueness-solutions.z}, we establish our existence theory for the
zero surface tension limit of \eqref{Euler} 
via
$\sigma$-independent a priori estimates.  For initial data
satisfying  the Taylor sign condition
\eqref{Taylor}, we have that solutions to the surface tension problem
\eqref{Euler} satisfy 
\begin{align*}
0< \nu\int_{\Gamma}\abs{ \eta\cdot n}^2\le-   \int_{\Gamma}\frac{1}{\sqrt{g}}N^ja_\ell^ja_{\ell}^k (\rho_0^2 J
    ^{-2})\cp{k} \abs{ \eta\cdot n}^2,
\end{align*}
on a sufficiently small time-interval $[0,T]$.

\section{Preliminaries}
\label{sec:preliminaries}

\subsection{Notation} \label{sec:notation} 
\subsubsection{The    three-dimensional gradient vector} \label{sec:grad-horiz-deriv}

Throughout this paper the symbol $D \label{n:D}$ will be used to denote the three-dimensional gradient vector 
\begin{align*}
	D=\Pset{\frac{
	\partial}{
	\partial x_1},\frac{
	\partial}{
	\partial x_2},\frac{
	\partial}{
	\partial x_3} }.
\end{align*}

\subsubsection{Notation for partial differentiation and Einstein's summation convention} \label{sec:notat-part-diff}

The $k$th partial derivative of $F$ will be denoted by $F\cp{k}=\frac{
\partial F}{
\partial x_k}$. Repeated Latin indices $i,j,k$, etc., are summed from $1$ to $3$, and repeated Greek indices $\alpha, \beta, \gamma$, etc., are summed from $1$ to $2$. For example, $F\cp{ii}=\sum_{i=1}^3\frac{\p^2F}{\p x_i\p x_i}$, and $F^i\cp{\alpha} I^{\alpha\beta} G^i\cp{\beta}=\sum_{i=1}^3\sum_{\alpha=1}^2\sum_{\beta=1}^2\frac{\p F^i}{\p x_\alpha} I^{\alpha\beta} \frac{\p G^i}{\p x_\beta}$.

\subsubsection{The divergence and curl
  operators} \label{sec:curl-diverg-oper} We use the notation $\Div u$
to denote the divergence of a vector field $u$ on
$\Omega$: \label{n.div.curl} 
\begin{align*}
	\Div u=u^1\cp{1}+u^2\cp{2}+u^3\cp{3}, 
\end{align*}
and we use the notation $\curl u$ to denote the curl of a vector $u$ on $\Omega$: 
\begin{align*}
	\curl u=\Pset{u^3\cp{2}-u^2\cp{3},u^1\cp{3}-u^3\cp{1}, u^2\cp{1} -u^1\cp{2}}. 
\end{align*}
We recall that the permutation symbol $\varepsilon_{ijk}$ is defined by 
\begin{align*}
	\varepsilon_{ijk}= 
	\begin{cases} 
		\hfill 1,& \text{even permutation of $\set{1,2,3}$,}\\
		\hfill -1,& \text{odd permutation of $\set{1,2,3}$,}\\
		\hfill 0,&\text{otherwise}. 
	\end{cases}
\end{align*}
This allows for   the curl of a vector field $u$
to be expressed as  $
\curl u=\varepsilon_{\cdot jk} u^k\cp{j}$. 
Letting $v=u(\eta)$ for a given flow map $\eta$, we use the notation
$\Div_\eta v$ to denote the Lagrangian divergence of $v$ on $\Omega$: 
\begin{align*}
	\Div_\eta v=A_r^s v^r \cp{s}, 
\end{align*}
and we use the notation $\curl_\eta v$ to denote the Lagrangian curl
of $v$ on $\Omega$: 
\begin{align*}
	\curl_\eta v= \varepsilon_{\cdot jk }A_j ^s v^k \cp{s}.
\end{align*}
\subsubsection{The scalar- and cross-product of vectors in $\mathbb{R}^3$}
\label{sec:scalar-cross-product}
Let $u$ and $v$ be vectors in $\mathbb{R}^3$.
The scalar-product of $u$ and $v$, denoted $u\cdot v$, is defined as
\begin{align}
  \label{notation.scalar}
  u\cdot v = u^1v^1+ u^2v^2+ u^3v^3.
\end{align}
The cross-product of $u$ and $v$, denoted $u\times v$, is defined as
\begin{align}
  \label{notation.cross}
  u\times v = \varepsilon_{\cdot jk}u^jv^k.
\end{align}

\subsubsection{Local coordinates near $\Gamma$ } \label{sec:local-coord-near}

We let $\Omega\subset\mathbb{R}^3$ denote an open, bounded subset of
class $H^s$ for $s\ge4$, and we let $\set{\UL}_{\local=1}^K$ denote an open covering of $\Gamma=\p\Omega$, such that for each $\label{n:l}\local\in \set{1,2,\dots,K}$, with 
\begin{align*}
	\VL&=B(0,r_\local),\text{ denoting the open ball of radius $r_\local$ centered at the origin}, \\
	\VLP&=\VL\cap\set{x_3>0}, \\
	\DL&=\VL\cap\set{x_3=0}, 
\end{align*}
there exists an $H^s$-class chart $\thetal$  satisfying
\begin{align*} 
	\thetal\colon \VL\to\UL\ &\text{ is an $H^s$ diffeomorphism}, \\
	\thetal&(\VLP)=\UL\cap\Omega, \\
	\thetal&(\DL)=\UL\cap\Gamma. 
\end{align*}
\begin{figure}
	[here] \centering 
	\includegraphics{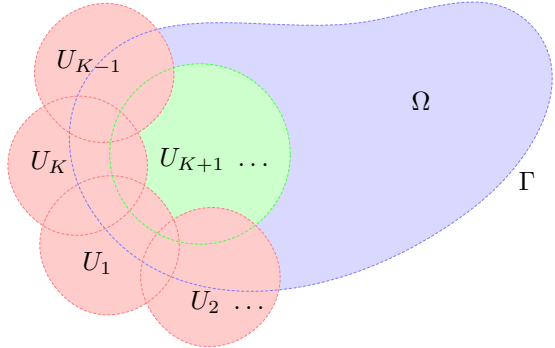} \caption{Indexing convention for the open cover $\set{U_\local}_{\local=1}^L$ of $\Omega$.} 
\end{figure}
For $L>K$, we let $\set{\UL}_{\local=K+1}^L$ denote a family of
open balls of radius $r_\local$ properly contained in $\Omega$ such that $\set{\UL}_{\local=1}^L$ is an open cover of $\Omega$. We let 
\begin{align*}
	\set{\xl}_{\local=1}^L \text{ denote a $C^\infty$
          partition-of-unity subordinate to the open covering of $\Omega$.} 
\end{align*}

\subsubsection{Tangential derivatives}\label{sec: tangential-derivative} On each $\UL\cap\Omega$, for $1\le\local\le L$, we let $\label{n:hd}\hd_\local$ denote the \textit{local tangential-derivative}. That is, for a differentiable function $f$ on $\Omega$, the $\alpha$th component of the local tangential-derivative of $f$ is defined in $\UL\cap\Omega$ by 
\begin{align*}
	\hd_{\local,\alpha} f=\Pset{\frac{\p}{\p x_\alpha}\bset{f\circ\thetal}}\circ\thetal^{-1}=\Pset{\pset{D f\circ\thetal}\frac{\p\thetal}{\p x_\alpha}}\circ\thetal^{-1}, 
\end{align*}
where for $K+1\le \local\le L$, we have set $\thetal$ to be the identity map $e$.

We let $\hd$ denote the \textit{tangential-derivative} whose $\alpha$th component is given by 
\begin{align*}
	\hd_\alpha=\sum_{l=1}^L\xl\hd_{\local,\alpha}. 
\end{align*}
We use $\hd_\alpha f$ or $f\cp{\alpha}$ to denote the components of the tangential-derivative of $f$.
\subsubsection{Geometry of the moving surface
  $\Gamma(t)$} \label{sec:geometry-moving}   The vectors $\eta\cp{\alpha}$ for $\alpha=1,2$, span the tangent space to the moving surface $\Gamma(t)=\eta(\Gamma)$ in $\mathbb{R}^3$.
The \textit{surface metric} $\label{n:g}g$ on $\Gamma(t)$ has components 
\begin{align*}
	g_{\alpha\beta}=\eta\cp{\alpha}\cdot\eta\cp{\beta}. 
\end{align*}
We let $\label{n:g0}g_0$ denote the surface metric of the initial surface $\Gamma$.
The components of the inverse metric $\label{n:sqrt g}\bset{g}^{-1}$
are denoted by $ g^{\alpha\beta}$. We use $\sqrt{g}$ to denote $\sqrt{\det g}$; we note that $\sqrt{g}=\abs{\eta\cp{1}\times\eta\cp{2}}$, so that \label{n:n}$n=\bset{\eta\cp{1}\times\eta\cp{2}}/\sqrt{g}$. Equivalently,
\begin{align}
  \label{e.cross.e}
  \sqrt{g} n =\eta \cp{1}\times \eta \cp{2}.
\end{align}
The
Laplace-Beltrami\label{n:BL} operator $\Delta_g$ is defined on $\Gamma$ as 
\begin{align*}
	\Delta_g=\sqrt{g}^{-1}
	\hd_\alpha[\sqrt{g}g^{\alpha\beta}
	\hd_\beta]. 
\end{align*} 

\subsubsection{Sobolev spaces on $\Omega$} \label{sec:diff-norms-open}

For integers $k\ge0$ and a smooth, open domain $\Omega$ of $\mathbb{R}^3$, we define the Sobolev space $H^k(\Omega)$ $\pset{H^k(\Omega;\mathbb{R}^3)}$ to be the closure of $C^\infty(\overline{\Omega})$ $\pset{C^\infty(\overline{\Omega};\mathbb{R}^3)}$ in the norm 
\begin{align*}
	\norm{u}_k^2=\sum_{\abs{a}\le k}\int_\Omega \Abs{D^a u(x) }^2 , 
\end{align*}
for a multi-index $a\in\mathbb{Z}^3_+$, with the convention that
$\abs{a}=a_1+a_2+a_3$. For real numbers $s\ge0$, the Sobolev spaces $H^s(\Omega)$ and the norms $\label{n:interior norm}\norm{\cdot}_s$ are defined by interpolation. 
We will write $H^s(\Omega)$ instead of $H^s(\Omega;\mathbb{R}^3)$ for
vector-valued functions. We use $H^s(\Omega)'$ to denote the dual
space of $H^s(\Omega)$. 

\subsubsection{Sobolev spaces on $\Gamma$} \label{sec:sobolev-spaces-gamma} For functions $u\in H^k(\Gamma)$, $k\ge0$, we set 
\begin{align*}
	\abs{u}_k^2=\sum_{\abs{a}\le k } \int_\Gamma \Abs{ \hd^a u(x)}^2, 
\end{align*}
for a multi-index $a\in \mathbb{Z}^2_+$. For real $s\ge0$, the Hilbert space $H^s(\Gamma)$ and the boundary norm $\label{n:boundary-norm}\abs{\cdot}_s$ is defined by interpolation. The negative-order Sobolev spaces $H^{-s}(\Gamma)$ are defined via duality. That is, for real $s\ge0$, 
\begin{align*}
	H^{-s}(\Gamma)=H^s(\Gamma)'. 
\end{align*}
\begin{rem}
  Throughout this paper, we suppress the Euclidean measure $dx$ by letting $\int_{\Omega} $ represent $\int_{\Omega} dx$. Similarly, the notation $\int_{\Gamma} $ represents $\int_{\Gamma} dS_0$ where  $dS_0=\sqrt{ g_0} dx_1 dx_2$ is the surface measure of the initial surface $\Gamma$. Equally, the time integral $\int_0^{t} $ should be read as $\int_0^{t} ds$. 
\end{rem}
\begin{rem}
We let $\abs{\cdot}_{s,\DL}$ denote the $H^s(\DL)$-norm.
\end{rem}
 
 \subsection{Differentiation and geometric identities and properties} 
\label{sec:prop-determ-j}

\subsubsection{An identity for the Jacobian determinant $J$} \label{sec:an-identity-jacobian} With $\dim\Omega=3$, we have that the Jacobian determinant $J$ is written as 
\begin{align}
	\label{identity: J} J=\frac{1}{\dim\Omega} \Bset{ a_r^s\eta^r\cp{s}}, 
\end{align}
which follows from the definition of the cofactor matrix $a$ in Section~\ref{sec:fixing-doma-lagr}.

\subsubsection{Time-differentiating the Jacobian determinant $J$ and
  the cofactor matrix $a$} \label{sec:diff-cofact-matr} We record the following basic differentiation formulas:
\begin{align}
	\label{formula: J_t} \p_t J&=a_r^sv^r\cp{s},\\
	\label{formula: a_t} \p_t a_i^k&=J^{-1}\bset{a_r^sa_i^k-a_i^sa_r^k}v^r\cp{s}. 
\end{align}
Using (\ref{formula: J_t}) and (\ref{formula: a_t}) and the fact that $a=JA$, we have that 
\begin{align}
	\label{formula: derivative: A } \p_t A_i^k=-A_i^sA_r^k v^r\cp{s}. 
\end{align}
We note that the time-differentiation formulas (\ref{formula:
  J_t})--(\ref{formula: derivative: A }) at once become formulas for
tangential-differentiation by replacing $v^r\cp{s}$ with
$\hd\eta^r\cp{s}$ in the right-hand sides of (\ref{formula: J_t})--(\ref{formula: derivative: A }).

The formulas (\ref{formula: J_t}) and (\ref{formula: a_t}) imply the following scaling relations:
\begin{alignat*}
	{3} \p_tJ&\sim aDv,&\qquad\p_ta&\sim a^2Dv,\\
	\p^2_tJ&\sim a^2\pset{Dv}^2+aDv_t,&\p^2_ta&\sim a^3\pset{Dv}^2+a^2Dv_t. 
\end{alignat*}
These scaling relations are particularly useful when estimating error-terms.

\subsubsection{The Piola identity} \label{sec:piola-identity} Columns of every cofactor matrix are divergence-free. Thus, 
\begin{align}
	\label{Piola identity} a_i^k\cp{k}=0. 
\end{align}

\subsubsection{Relating the normal vectors of $\Gamma$ and $\Gamma(t)$} \label{sec:relat-norm-vect} With $\label{n:N}N$ as the outward unit normal to the reference surface $\Gamma$, the outward-normal direction of the moving surface $\Gamma(t)$ is 
\begin{align*}
	a_i^k N^k=\abs{\eta\cp{1}\times\eta\cp{2} } n^i. 
\end{align*}
The identity $\sqrt{g}=\sqrt{\det g}=\abs{\eta\cp{1}\times
  \eta\cp{2}}$ implies that 
\begin{align}
	\label{formula: outward normal} a_i^kN^k=\sqrt{g}n^i. 
\end{align}

\subsubsection{Derivatives of the inverse metric $g^{\alpha\beta}$, Jacobian $\sqrt{g}$ and unit normal $n$} \label{sec:diff-metr-g-1}

A tangential-derivative of the inverse metric $g^{\alpha\beta}$,
Jacobian determinant $\sqrt{g}$ and moving outward normal unit vector
$n$ are given by the formulas 
\begin{subequations}
\label{hor.gJn}
  \begin{align}
    \label{formula: derivative: metric inverse} \hd g^{\alpha\beta}&=-g^{\alpha\mu}\hd g_{\mu\nu}g^{\nu\beta},\\
    \label{formula: derivative: Jacobian determinant} \hd \sqrt{
      g}&=\textstyle\frac{1}{2}\sqrt{g}g^{\mu\nu}\hd g_{\mu\nu},
    \\
    \label{formula: derivative: outward normal vector} \hd
    n&=-g^{\gamma\delta}\bset{\hd\eta\cp{\delta}\cdot n
    }\eta\cp{\gamma}.
  \end{align}
\end{subequations}
Also, \begin{subequations}
\label{t.gJn}
  \begin{align}
    \label{t derivative: metric inverse} \p_t g^{\alpha\beta}&=-g^{\alpha\mu}\p_t g_{\mu\nu}g^{\nu\beta},\\
    \label{t derivative: Jacobian determinant} \p_t \sqrt{
      g}&=\textstyle\frac{1}{2}\sqrt{g}g^{\mu\nu}\p_t g_{\mu\nu},
    \\ 
    \label{t derivative: outward normal vector} \p_t
    n&=-g^{\gamma\delta}\bset{v\cp{\delta}\cdot n
    }\eta\cp{\gamma}.
  \end{align}
\end{subequations}

\begin{rem}
  The right-hand side of (\ref{formula: derivative: outward normal
    vector}) and (\ref{t derivative: outward normal vector}) is a
vector that is  tangent to the embedded surface.
\end{rem}

\subsubsection{Relating the Laplace-Beltrami operator $\Delta_g$ to the unit normal $n$} \label{sec:relat-surf-tens-1}
With the formulas (\ref{formula: derivative: metric inverse}) and
(\ref{formula: derivative: Jacobian determinant}), we have that the
Laplace-Beltrami operator $\Delta_g=\sqrt{g} ^{-1}\hd_\alpha[\sqrt{g}
g ^{\alpha\beta} \hd_\beta]$ applied to the particle flow
$\eta$ decomposes into normal and tangential components  as  
\begin{align*}
	\sqrt{g}\Delta_g(\eta)=\sqrt{g} \underbrace{(g ^{\alpha\beta} \eta
        \cp{\beta\alpha}-  \eta \cp{\mu\alpha} \cdot \eta \cp{\nu}g
        ^{\alpha\mu} g ^{\nu\beta}\eta\cp{\beta})}_{g ^{\alpha\beta}
     [   \eta \cp{\alpha\beta} \cdot n] n}+\sqrt{g}\eta\cp{\mu\alpha}\cdot\eta\cp{\nu} (g^{\alpha\beta}g^{\mu\nu}-g^{\alpha\mu}g^{\nu\beta})\eta\cp{\beta}. 
\end{align*}
We therefore have the identity 
\begin{align}
	\label{formula:Laplace-Beltrami opearator and projection operator} \sqrt{g}\Delta_g(\eta)=\sqrt{g}g^{\alpha\beta}\bset{\eta\cp{\alpha\beta}\cdot n}n. 
\end{align}
For reference, we recall the identity $\Delta_g(\eta)=-H(\eta)n.$

\subsection{Two  identities for the Euler equations in Lagrangian variables} \label{sec:tangential-identity} 
 
\subsubsection{The Lagrangian vorticity
  equation} \label{sec:lagrangian-vorticity} With the
operator $\curl_\eta$  defined in Section~\ref{sec:curl-diverg-oper},
\begin{align}
	\label{E.Lagrangian vorticity} \curl_\eta v_t=0\ \ \ \text{in
        } \Omega.
\end{align}
 The identity \eqref{E.Lagrangian vorticity} is obtained by taking the
 Lagrangian curl of the Euler equations
(\ref{coupled equation}).

\subsubsection{A tangential identity for $v_t$ on the
  boundary $\Gamma$}\label{sec:tangential-identity-1} Setting $\sigma=1$,
\begin{align} 
	\label{identity: tangential} v_t\cdot \eta\cp{\alpha} = f^{-1}\hd_\alpha\pset{g^{\mu\nu} \bset{\eta\cp{\mu\nu}\cdot n}}\ \ \ \text{on } \Gamma.
\end{align}
The identity \eqref{identity: tangential} is established using the
Euler equations~(\ref{coupled equation}), the Laplace-Young boundary condition \eqref{Euler.bc}, and the formula (\ref{formula:Laplace-Beltrami opearator and projection operator}).

\subsection{General inequalities} \label{sec:trace-estim-techn} 
   
\subsubsection{Trace estimates} \label{sec:trace-estimate} For
$s>\frac{1}{2}$ 
and some constant $C$ independent of $w\in H^s(\Omega)$, the trace
theorem \cite{A75} states that the trace of $w$ is
defined in $H^{s- \frac{1}{2}}(\Gamma)$ with the estimate 
\begin{align*}
	\abs{w}_{s-\frac{1}{2}}\le C \norm{w}_s.
\end{align*}

\begin{lem}\label{lem.j}
Let $w\in L^2(0,T;H^1(\Omega ))\cap L^\infty(0,T;L^2(\Omega ))$. Then
\begin{align}
\label{in.lem}
 \int_{0}^{T} \abs{w}_{0.25 }^2 \le\delta \int_{0}^{T} \norm{ w}_1^2+ C_\delta \, T \Sup \norm{ w(t)}_0^2 ,
\end{align}
where the constant $C_\delta $  depends on $1/ \delta>0$.
\end{lem}

\begin{proof} 
  Using interpolation and Young's inequality with $\delta >0$, we have that
  \begin{align*}
  \int_{0}^{T}
    \norm{w}_{0.75}^2\le C \int_{0}^{T}
    \norm{w}_{0}^{\frac{1}{2}}\norm{w}_{1}^{\frac{3}{2}}\le\delta
    \int_{0}^{T} \norm{ w}_1^2+ C_\delta \, T \Sup \norm{ w(t)}_0^2 .
  \end{align*}
The proof is complete thanks to the trace theorem.
\end{proof}

\begin{lem}
	[Normal trace theorem]\label{lem: normal trace} Let $w$ be a vector field defined on $\Omega$ such that $\hd w\in L^2(\Omega)$ and $\Div w\in L^2(\Omega)$, and let $N$ denote the outward unit normal vector to $\Gamma$. Then the normal trace $\hd w\cdot N$ exists in $H^{-0.5}(\Gamma)$ with the estimate 
	\begin{align}
		\label{estimate: normal trace} \abs{\hd w\cdot N}_{-0.5}^2 \le C \Bset{\norm{\hd w}_{L^2(\Omega)}^2 +\norm{\Div w}_{L^2(\Omega)}^2}, 
	\end{align}
	for some constant $C$ independent of $w$. 
\end{lem}

See \cite{Temam} for the proof of Lemma~\ref{lem: normal
trace} in the case that $\hd w$ is replaced by $w$.  For $\phi \in H^1(\Omega)$, $\int_\Gamma \hd w\cdot n \phi dS = \int_\Omega  
\hd w \cdot D \phi dx - \int_\Omega \operatorname{div} w \hd \phi dx$ because we can integrate-by-parts with $\hd$ without any boundary contributions.  Thus, the
identical proof given in \cite{Temam} proves Lemma \ref{lem: normal trace}.   Similarly, we have 

\begin{lem}
	[Tangential trace theorem]\label{lem: tangential trace} Let $w$ be a vector field defined on $\Omega$ such that $\hd w\in L^2(\Omega)$ and $\curl w\in L^2(\Omega)$, and let $\tau_1, \tau_2$ denote the unit tangent vectors to $\Gamma$, so that any vector field $u$ on $\Gamma$ can be uniquely written as $u^\alpha \tau_\alpha$. Then the tangential trace $\hd w\cdot \tau_\alpha$ exists in $H^{-0.5}(\Gamma)$ with the estimate 
	\begin{align}
		\label{est: tangential trace} \abs{\hd w\cdot \tau_\alpha}_{-0.5}^2 \le C \Bset{\norm{\hd w}_{L^2(\Omega)}^2 +\norm{\curl w}_{L^2(\Omega)}^2}, 
	\end{align}
	for some constant $C$ independent of $w$. 
\end{lem}
 
See \cite{CCS07} for the proof of Lemma~\ref{lem: tangential trace}.  

\subsubsection{An elliptic estimate which is independent of
  $\kappa$} \label{sec:kappa-indep-estim} The following lemma is used to establish our
$\kappa$-independent a priori estimates. The proof is given in Section~6 of \cite{CS06}. 
\begin{lem}
	\label{kappa-indep: f+kappa f_t} Let $\kappa>0$, $s\ge0$ and
        $g\in L^\infty\pset{0,T; H^s(\Omega)}$ be given. Suppose that
        $f\in H^1\pset{0,T;H^s(\Omega)}$ satisfies
	\begin{align}
		\label{form: f + kappa f_t} f+\kappa f_t=g\ \ \ \text{in}\ \Omega\times (0,T). 
	\end{align}
	Then, 
	\begin{align*}
		\norm{f}_{L^\infty(0,T;H^s(\Omega))}\le C\max\set{\norm{f(0)}_s,\norm{g}_{L^\infty(0,T;H^s(\Omega))}}. 
	\end{align*}
\end{lem}

\subsubsection{A technical lemma}
\label{sec:technical-lemma}

The following technical lemma is established in \cite{CS07}.
\begin{lem}
  \label{lem.tech}
There exists a constant $C$ such that
\begin{align*}
  \norm{\hd w}_{H^{0.5}(\Omega)'}\le C    \norm{
    w}_{H^{0.5}(\Omega)}\ \ \ \forall \ w\in   H^{0.5}(\Omega).
\end{align*}
\end{lem}

\subsubsection{The Hodge decomposition elliptic estimates} \label{sec:hodge-decomp-ellipt}

The following Hodge-type elliptic estimate is well-known and follows from the identity $-\Delta w=\curl\curl w- D \Div w$, together with estimates divergence-form
elliptic operators with Sobolev-class coefficients:
\begin{prop}
	\label{prop:Hodge} For an $H^r$-class domain $\Omega$, $r\ge 3$, if $w\in L^2(\Omega;\mathbb{R}^3)$ with $\curl w\in H^{s-1}(\Omega;\mathbb{R}^3)$, $\Div w\in H^{s-1}(\Omega)$, and $w\cdot N|_\Gamma \in H^{s-\frac{1}{2}}(\Gamma)$ for $1\le s\le r$, then there exists a constant $C>0$ depending only on $\Omega$ such that 
	\begin{align*}
		\norm{w}_s\le C\Bset{\norm{w}_0 +\norm{\curl w}_{s-1}+\norm{\Div w}_{s-1}+\abs{\hd w\cdot N}_{s-\frac{3}{2}}}, 
	\end{align*}
	where $N$ denotes the outward unit-normal to $\Gamma$. 
\end{prop}
 In fact, the well-known version of this elliptic estimate replaces $\abs{\hd w\cdot N}_{s-\frac{3}{2}}$ with $\abs{ w\cdot N}_{s-\frac{1}{2}}$, but this requires too much 
 regularity for the unit normal $N$.   It is easy to verify that having estimates for the divergence and curl of a vector field $w$ only requires the slightly weaker norm
 $\abs{\hd w\cdot N}_{s-\frac{3}{2}}$ estimated in order to infer full regularity of $w$.   Also, this elliptic estimate is usually stated for smooth domains; the Sobolev-class
 regularity follows from the Sobolev embedding theorem and somewhat standard elliptic estimates for second-order elliptic operators with Sobolev-class coefficients.
 
\subsubsection{A polynomial-type inequality} \label{sec:polyn-type-ineq} For a constant $M\ge 0$, suppose $f\colon t\mapsto f(t)\ge0$ continuously and satisfies 
\begin{align}
	\label{poly-type ineq} f(t)\le M +t\, P(f(t)), \qquad t\ge0, 
\end{align}
where $P$ denotes a generic polynomial function. Then for $t$ taken sufficiently small, 
\begin{align*} 
	f(t)\le 2 M. 
\end{align*}

\section{A parabolic $\kappa$-approximation of the  surface tension
  problem \eqref{Euler}} 
\label{sec:parab-kappa-appr} 
 
In this section, we define a parabolic approximation of the
surface tension problem \eqref{Euler}, which we term the $\kappa$-problem.
The $\kappa$-problem is defined by adding artificial viscosity terms
to the Euler equations and the Laplace-Young boundary condition.   The
salient feature of the $\kappa$-problem is its compatibility with
our   energy-estimates methodology based on
Proposition~\ref{prop:Hodge} in that (1)
the
transport structure of the Euler equations is maintained and (2) the momentum
equations of the $\kappa$-problem are equivalently expressed in the
form $f+ \kappa f_t =g$ with $f$ equalling the gradient of $J$.  These
structural properties of the $\kappa$-problem respectively yield the
$\kappa$-independent curl- and
divergence-estimates.

\subsection{Assuming $C^\infty$-class initial data}
\label{sec:assum-cinfty-class-1}
In our construction of solutions to the surface tension problem
\eqref{Euler}, we assume that the
initial data $(\rho_0,u_0, \Omega)$ is of $C^\infty$-class and satisfy
the conditions
\eqref{l.boundedness.r0} and \eqref{s.compatibility conditions}, as in
Appendix~\ref{sec:appendix:init}.   Later, in    
Section~\ref{sec:optim-regul-init}, we will recover the optimal regularity
of the initial data stated in Theorem~\ref{thm.s.main}.

\subsection{The parabolic  approximation of the surface tension problem \eqref{Euler}} 
\label{sec:parab-appr-compr} 
 
We recall that an equivalent expression of the   surface tension problem \eqref{Euler} is 
\begin{alignat*}{2}
	v^i_t +2 A_i^k\pset{\rho_0 J^{-1}}\cp{k}&=0&\qquad&\text{in}\ \Omega\times(0,T], \\
	\rho_0^2J^{-2} &=\beta - \sigma g^{\alpha\beta}\eta
        \cp{\alpha\beta}\cdot n&&\text{on}\ \Gamma\times(0,T].
\end{alignat*}
We have used the identity $H(\eta)=-g^{\alpha\beta}\eta
\cp{\alpha\beta}\cdot n$ in writing the  boundary condition.

The variables  in the following problem a
priori 
depend on the parabolic parameter $\kappa$.
To indicate this dependence, we place the
symbol $\sim$ above each of the variables.
\begin{defn}
	[The $\kappa$-problem] \label{defn.akp} For
        $\kappa>0$, we define $\vt$ as the solution of
	\begin{subequations}
		\label{akp} 
		\begin{alignat}
			{4} \label{akp.momentum} \vt^i_t+2\At_i^k\pset{\rho_0\Jt^{-1}}\cp{k}-\kappa \At_i^k(\rho_0\Jt_t)\cp{k}&=0&&in\ \Omega\times(0,T_\kappa],\\
			\label{akp.bc}
                        \rho_0^2\Jt^{-2}-\kappa
                        \rho_0^2\Jt ^{-1}\Jt_t&=\Hk &\ \ &on\ \Gamma\times(0,T_\kappa],\\
			\label{akp: init cond} \pset{\et,\vt}|_{t=0}&=\pset{e,u_0}&&on\ \Omega.
		\end{alignat}
	\end{subequations}
The function $\Hk$ appearing in the right-hand side of \eqref{akp.bc}
is defined as
\begin{align}
\label{akp.H}
  \Hk = \beta(t) - \sigma \gt ^{\alpha\beta}\et \cp{\alpha\beta}\cdot \nt
  - \kappa   \gt ^{\alpha\beta} \vt \cp{\alpha\beta}\cdot \nt,
\end{align}
where
the  function $\beta(t)$ appearing in
        the right-hand side of \eqref{akp.H} is defined as
\label{n:beta zero}  
	\begin{align}
		\label{akp:beta zero}
                \begin{aligned}[b]
                  \beta(t)=\beta&+\sum_{a=0}^{3}
                  \frac{t^a}{a!}\p^a_t\Bbset{\rho _0^2\Jt^{-2} -
                    \beta+\sigma \gt ^{\alpha\beta}\et
                    \cp{\alpha\beta}\cdot \nt }|_{t=0}
                  \\
                  &+\kappa\sum_{a=0}^{3} \frac{t^a}{a!}\p^a_t\Bbset{-\rho
                    _0^2\Jt^{-1}\Jt_t+  \gt ^{\alpha\beta}
\vt\cp{\alpha\beta}\cdot \nt }|_{t=0}.
                \end{aligned}
	\end{align}
\end{defn}
\begin{rem}
	The particular artificial viscosity $-\kappa
        \At_i^k (\rho_0\Jt_t)\cp{k}$ appearing in
        (\ref{akp.momentum})  preserves the
        transport structure of the Euler equations.  In comparison, an
        artificial viscosity of the form $- \kappa \At_r^j[\At_r^k
        \vt^i \cp{k}]\cp{j}$ would not
        preserve the transport structure of the Euler equations.
\end{rem}

\begin{rem}\label{rem.beta}
The initial data satisfy the compatibility conditions \eqref{s.compatibility conditions} so that   definition \eqref{akp:beta zero} of $\beta (t)$ implies that $\beta(t)\to\beta$ in the limit as  $\kappa$ tends to
  zero. Formally, the $\kappa $-problem~\eqref{akp} is
  asymptotically consistent with the surface tension problem \eqref{Euler}.  
\end{rem}
\begin{rem}
For   $\phi\in L^2 (0,T;H^1(\Omega))$ such that $\phi\cdot \nt\in L^2(0,T;H^1(\Gamma))$, the variational equation  for the $\kappa$-problem \eqref{akp} is 
\begin{align}
\label{akp.var}
\begin{aligned}[b]
  \int_{0}^{T}\int_{\Omega} \rho_0 \vt_t \cdot\phi -
  \int_{0}^{T}\int_{\Omega}  \rho_0^2 \Jt ^{-2} \at_i^k \phi^i \cp{k} +
  \kappa\int_{0}^{T} \int_{\Omega} \rho_0\Jt_t \at_i^k (\rho_0 \Jt
  ^{-1} \phi^i ) \cp{k} &\\
+ \int_{0}^{T}\int_{\Gamma}\beta(t) \sqrt{\gt} \phi\cdot \nt+
\int_{0}^{T}\int_{\Gamma}[\sigma\et\cp{\beta} + \kappa \vt  \cp{\beta}]^i (\nt^i\sqrt{\gt}\gt
^{\alpha\beta} \phi\cdot \nt)\cp{\alpha} &=0. 
\end{aligned}
\end{align}
Since $\lim_{\kappa\to
  0}\beta(t)=\beta$, we have that   \eqref{akp.var} with $\kappa=0$ is  the variational equation for  the surface tension problem
\eqref{Euler}. Furthermore,   for $\kappa>0$ our choice of
artificial
        viscosity $-\kappa \At_i^k  (\rho_0 \Jt _t)\cp{k}$ has the
        nice property of not
        introducing any nontrivial boundary integrals in the variational equation
        \eqref{akp.var}, whereas the use of the artificial
        viscosity 
 $- \kappa \At_r^j[\At_r^k
        \vt^i \cp{k}]\cp{j}$, for example,  in \eqref{akp.var}
would       require additional boundary conditions to account for the tangential
        components of
        $- \kappa N^j \At_r^j\At_r^k
        \vt^i \cp{k}$.
\end{rem}

\subsection{The constant-in-time vectors $\mathrm{v}_a$ for $a=1,2,3,4$} \label{sec:comp-cond-kappa}  The vector field $\vt_t|_{t=0}$ is computed using the
momentum equations (\ref{akp.momentum}), as follows:
\begin{align*}
	\vt_t|_{t=0}=\Pset{\kappa \At_\cdot^k(\rho_0 \Jt_t) \cp{k} -2\At_\cdot^k\pset{\rho_0 \Jt^{-1}}\cp{k}}|_{t=0}=D\pset{\kappa\rho_0\Div u_0-2\rho_0}. 
\end{align*}
Similarly, for all $a\in \mathbb{N}$, 
\begin{align*}
	\p^a_t\vt|_{t=0}=\frac{\p^{a-1}}{\p t^{a-1}}\Pset{\kappa
          \At_\cdot^k(\rho_0 \Jt_t)\cp{k}
          -2\At_\cdot^k\pset{\rho_0\Jt^{-1}}\cp{k}}\big|_{t=0}\ \ \
        \text{on } \Omega. 
\end{align*}
This formula makes it clear that each $\p_t^a\vt|_{t=0}$ is a function
of space-derivatives of  the initial data $u_0$ and $\rho_0$.
We define the constant-in-time vectors $\mathrm{v}_a$   as\label{n.rmv}
\begin{align}
	\label{defn.va} {\rm v}_a=\frac{\p^{a-1}}{\p t^{a-1}}\Pset{\kappa \At_\cdot^k(\rho_0 \Jt_t)\cp{k} -2\At_\cdot^k\pset{\rho_0\Jt^{-1}}\cp{k}}\big|_{t=0}, \ \ \ \text{for}\ a=1,2,3,4. 
\end{align}
 Since $2\At_i^k\pset{\rho_0 \Jt^{-1}}\cp{k}=\rho_0^{-1}\at_i^k\pset{\rho_0^2 \Jt^{-2}}\cp{k}$, we have that ${\rm v}_a\to v_a$, $a=1,2$, as $\kappa\to0$ where $v_a$ are defined in Section~\ref{sec:s.comp-cond}. We use (\ref{akp.bc}) to compute the following identities: for $a=0,1,2,3,$
\begin{align}
	\label{akp: compat cond} \partial^a_t\bset{ \rho_0^2\Jt^{-2}
          -\kappa\rho_0^2\Jt^{-3}\Jt_t}\big|_{t=0} =
  \p^a_t\bset{\beta(t)-\sigma \gt^{\alpha\beta}\et
        \cp{\alpha\beta}\cdot \nt -   \kappa     \gt ^{\alpha\beta}
        \vt\cp{\alpha\beta}  \cdot \nt}\big|_{t=0}.
\end{align}
  \section{A priori estimates for the $\kappa$-problem \eqref{akp} } \label{sec:kappa independent estimates}

We establish our  $\kappa$-independent  a priori
estimates in this section; the precise estimate is stated below  in Lemma~\ref{lem.EK}.  Our existence of solutions to the
$\kappa$-problem \eqref{akp}   is established  in
Section~\ref{sec:appr-kappa-probl}.

For $\kappa>0$, we define the following higher-order energy function:
 \label{n:EK} 
\begin{align}
	\label{defn.kappa E(t)} 
	\begin{aligned} [b]
		\EK =1&+ \sum_{a=0}^5\norm{\p^a_t\et (t)}_{5-a}^2 +
                \abs{ \vt_{ttt}\cdot \nt(t)}_1^2 + \sum_{a=0}^2
                \abs{\hd^2\!\p^a_t\vt \cdot \nt(t)}_{2.5-a}^2\\
			+ \int_0^T \abs{
                          \sqrt{\kappa}\hd\vt_{tttt}\cdot \nt}_{0}^2		&
			+ \int_0^T \norm{ \sqrt{\kappa}\vt_{tttt}}_{1}^2		+
                \sum_{a=0}^3\norm{\kappa\p^{a}_t\vt(t)}_{5-a}^2
+ \sum_{a=0}^2
                \abs{\kappa\hd^2\!\p^a_t\vt_t \cdot \nt(t)}_{2.5-a}^2.
	\end{aligned}
\end{align}

\begin{rem}
  The inequality stated in Theorem~\ref{thm.akp}
  ensures  that $\EK$ is continuous in time.
\end{rem}

We make the following definition to allow for constants to depend on
$1/\delta$:
\begin{defn}
 	[Notational convention for constants depending on $1/ \delta>0$] We let $\PK\label{n:kP}$ denote a generic polynomial with constant and coefficients depending on $1/\delta>0$.
	
	We define the constant $\NK>0$ by 
	\begin{align}
		\NK=\PK\pset{\norm{u_0}_{100}, \norm{\rho_0}_{100}}. 
	\end{align}
	
	We let $\RK\label{n:kR}$ denote generic lower-order terms satisfying 
	\begin{align*}
		\int_0^T\RK\le \NK+\delta \Sup \EK+T\,\PK(\Sup \EK). 
	\end{align*}
\end{defn} 
The artificially high
$H^{100}(\Omega)$-norm in defining $\NK$ is acceptable as the initial
data $(\rho_0,u_0, \Omega)$ is of $C^\infty$-class.
We shall assume that       
\begin{align}
\label{assum.Jt}  \frac{1}{2}\le \Jt \le \frac{3}{2}\ \ \text{for all } t\in[0,T] \text{
    and } x\in {\Omega}.
\end{align}
The bounds \eqref{assum.Jt} are possible  by taking $T>0$ sufficiently small, since thanks to Theorem~\ref{thm.akp}
\begin{align}\label{universal constant}
	\norm{\Jt-1}_{L^\infty(\Omega)}\le C\norm{\int_0^t\p_t \Jt }_2\le C\sqrt{T}. 
\end{align}

\begin{lem}
	[A priori estimates for the $\kappa$-problem] We let
        $ \vt$ solve the $\kappa$-problem $\eqref{akp}$\label{lem.EK}
        on a time-interval $[0,T]$, for some $T=T _{\kappa} >0$. Then independent of $\kappa>0$,
	\begin{align}
		\label{est: EK} \Sup\EK\le \int_0^{T}\RK. 
	\end{align}
\end{lem}

We will establish Lemma~\ref{lem.EK} in the following six steps:

\subsection*{Step 1: The $\kappa$-independent curl-estimates} \label{sec:curl-estimates.k} 

We follow \cite{CS10} in establishing
\begin{lem}[The  $\kappa$-independent curl-estimates]
	\label{lem: curl estimates} 
	\begin{align*}
\Sup			\sum_{a=0}^4\norm{\curl \p^a_t\et (t)}_{4-a}^2
+ \int_{0}^{T}\norm{\sqrt{\kappa} \curl \vt_{tttt}}_0^2 
+
\Sup			\sum_{a=0}^3\norm{\kappa\curl \p^a_t\vt (t)}_{4-a}^2
			&\le\int_{0}^{T}\RK. 
	\end{align*}
\end{lem}
\begin{proof}
The Lagrangian curl of  \eqref{akp.momentum} yields $\curl_{\et}\vt_t=0$. Setting
\begin{align*}
  B(\At, D\vt)= \varepsilon_{\cdot ji}\At_t{}_j^s\vt^i\cp{s},
\end{align*}
we find that $\pset{\curl_{\et}
          \vt}_t = B(\At,D\vt)$. By the fundamental theorem of calculus, 
	\begin{align}
		\label{curl v} \curl_{\et} \vt(t)=\curl u_0 +\int_0^t B(\At,D\vt). 
	\end{align}
	Applying the gradient operator $D$ to (\ref{curl v}) and a second application of the fundamental theorem of calculus, we find that 
	\begin{align}
		\label{curl D eta} 
		\begin{aligned}[b]
			D \curl \et (t)=&\;tD\curl u_0-\varepsilon_{\cdot ji}D\et ^i\cp{s}\int_0^t \At_t{}_j^s\\
			&\qquad+\varepsilon_{\cdot ji}\int_0^t\bset{\At_t{}_j^sD\et ^i\cp{s}-D\At_j^s\vt^i\cp{s}} +\int_0^t\int_0^{t'} DB(\At,D\vt), 
		\end{aligned}
	\end{align}
	where we have used the identity $\curl_{\et} D\et =\curl D\et +\varepsilon_{\cdot ji}D\et ^i\cp{s}\int_0^t \At_t{}_j^s$. 
	
	The differentiation formula (\ref{formula: derivative: A }) equally holds for when the gradient $D$ replaces $\p_t$; hence, the first three terms on the right-hand side of (\ref{curl D eta}) are, with respect to the action of $D^3$, each bounded by $\int_{0}^{T}\RK$.
	
	We next analyze the highest-order term created in $\int_0^t \int_0^{t'}D^4B(\At,D\vt)$. With (\ref{formula: derivative: A }), 
	\begin{align*}
		DB(\At,D\vt)=-\varepsilon_{\cdot ji}\bset{\At_j^q D\vt^r\cp{q} \At_r^s \vt^i\cp{s} +\At_j^q \vt^r\cp{q} \At_r^s D\vt^i\cp{s} + \vt^r\cp{q} \vt^i\cp{s} D\pset{\At_j^q \At_r^s} } , 
	\end{align*}
	from which it follows that the highest-order term of $D^4B(\At,D\vt)$ is 
	\begin{align*}
		-\varepsilon_{\cdot ji}\int_0^t\int_0^{t'}\bset{ \At_j^q D^4\vt^r\cp{q}\At_r^s \vt^i\cp{s}+\At_j^q \vt^r\cp{q}\At_r^s D^4\vt^i\cp{s} }. 
	\end{align*}
	With a relaxation of the precise structure of the summands in the integrands of the highest-order terms of $D^4B(\At,D\vt)$, we highlight the derivative count that results from integration by parts in time by writing 
	\begin{align*}
		\int_0^t\int_0^{t'}D^4 B(\At,D\vt)=-\int_0^t\int_0^{t'}D^5\et (D\vt\,\At\,\At)_t+\int_0^t D^5\et D\vt \,\At\,\At. 
	\end{align*}
	
	With such a temporal-integration-by-parts computation, the action of $D^3$ in (\ref{curl D eta}) yields
	\begin{align}
		\label{estimate: curl eta.k} \Sup \norm{\curl\et (t)}_{4}^2\le \int_{0}^{T}\RK, 
	\end{align}
	and by the same arguments, the action of $D^3$ and $\kappa D^4$ in (\ref{curl v}) yield
	\begin{align}
		\label{estimate: curl v} \Sup \norm{\curl
                  \vt(t)}_{3}^2+
\Sup \norm{\kappa\curl \vt(t)}_{4}^2\le \int_{0}^{T}\RK. 
	\end{align}
	
	By returning to the Lagrangian vorticity equation $\curl_{\et}\vt_t=0$, we find that 
	\begin{align}
		\label{curl: vort eq} \curl \vt_t=\varepsilon_{j\cdot i}\vt_t^i\cp{s}\int_0^t \At_t{}_j^s, 
	\end{align}
	and by considering the action of $D^2$ and $\kappa D^3$  in
        the identity \eqref{curl: vort eq} we infer that 
	\begin{align}
		\label{estimate: curl v_t}
 \Sup\norm{\curl
                  \vt_{t}(t)}_{2}^2  + \Sup\norm{\kappa\curl
                  \vt_{t}(t)}_{3}^2  \le\int_{0}^{T}\RK. 
	\end{align}
	
	By considering the action of $D\p_t$, $\p^2_t$,  $\kappa D^2\p_t$,
        $\kappa D\p^2_t$,  and
        $\sqrt{\kappa} \p^3_t $ in
        (\ref{curl: vort eq}), and using the
        fundamental-theorem-of-calculus identity
        $\p^a_t\vt_t=\p^a_t\vt_t |_{t=0}+\int_0^t\p^{a+1}_t\vt_t$ in the lower-order terms, we establish that 
	\begin{align}
		\label{estimate: curl v_tt and curl vttt} \Sup
\sum_{a=0}^1\norm{\curl\p^a_t\vt_{tt}(t)}_{1-a}^2+\sum_{a=0}^1\norm{\kappa\curl\p^a_t\vt_{tt}(t)}_{2-a}^2+
\int_{0}^{T}\norm{\kappa\curl\vt_{tttt}(t)}_{0}^2 \le\int_{0}^{T}\RK. 
	\end{align}
	 The sum of the inequalities (\ref{estimate: curl eta.k})--(\ref{estimate:  curl v_tt and curl vttt}) completes the proof. 
\end{proof}

\subsection*{Step 2: The $\kappa$-independent estimates for $\vt_{tttt}$ and $\vt_{ttt}$} \label{sec: Step 2.k}

We equivalently write the momentum equations  (\ref{akp.momentum}) and
boundary condition (\ref{akp.bc}) of the $\kappa$-problem as
\begin{subequations}
	\label{twice rewritten parabolic equations} 
	\begin{alignat}
		{4}\label{test: twice rewritten interior equation}
                \rho_0 \vt_t^i + \at_i^k\pset{\rho_0^2 \Jt^{-2}}\cp{k}
                - \kappa \rho_0\Jt^{-1}\at_i^k(\rho_0 \Jt_t)\cp{k} &=0 &\quad&{\rm in}\ \Omega\times(0,T], \\
		\rho_0^2 \Jt^{-2}-\kappa \rho_0^2\Jt^{-1}\Jt_t &=
\Htk &&{\rm on}\ \Gamma\times(0,T],\label{test: twice rewritten boundary condition} 
	\end{alignat}
\end{subequations}
where $\Htk$ is given by \eqref{akp.H}. From (\ref{akp.momentum}), we have the identity 
\begin{align*}
	\vt_t\cdot\et_\gamma=\hd_\gamma[-2\rho_0\Jt^{-1}+\kappa \rho_0
        \Jt_t]. 
\end{align*}
Multiplying this identity by $\rho_0\Jt^{-1} $ yields
\begin{align*}
	\rho_0\Jt^{-1}
        \vt_t\cdot\et_\gamma=\hd_\gamma[-\rho_0^2\Jt^{-2}+\kappa\rho_0^2\Jt^{-1}
        \Jt_t]-\kappa(\rho_0\Jt^{-1})\cp{\gamma} \rho_0\Jt_t.  
\end{align*}
We use the boundary condition (\ref{test: twice rewritten boundary
  condition}) to obtain the following tangential identity for $v_t$:
\begin{align} 
	\label{indentity: smoothed tangential} 
		\rho_0 \Jt ^{-1} \vt_t\cdot\et\cp{\gamma}
                =\hd_\gamma\Big[
                {\sigma\gt^{\mu\nu}\et\cp{\mu\nu}\cdot\nt} +
                \kappa \gt ^{\mu\nu}\vt \cp{\mu\nu}\cdot \nt-\beta(t) \Big] - \kappa(\rho_0\Jt^{-1})\cp{\gamma} \rho_0\Jt_t. 
\end{align}
\begin{prop}
	[Energy estimates for the fourth time-differentiated
        problem]\label{prop: divergence: energy estimates}  
	\begin{align*}
			\Sup
                        \norm{\vt_{tttt}(t)}_0^2
+
\Sup\norm{\p^4_t\Jt(t)}_0^2
+
\Sup\abs{\vt_{ttt}\cdot \nt(t)}_1^2
\\
+
\int_{0}^{T}\abs{\sqrt{\kappa} \p^4_t\vt\cdot\nt(t)}_1^2			+ \int_0^T \norm{ \sqrt{\kappa}\p^5_t\Jt}_{0}^2\le \int_{0}^{T}\RK .
	\end{align*}
\end{prop}
\begin{proof}
	Testing four time-derivatives of (\ref{test: twice rewritten interior equation}) against $\p^4_t \vt$ in the $L^2(\Omega)$-inner product, and integrating by parts with respect to $\p_k$ in the interior integrals $\int_\Omega \at_i^k \p^4_t(\rho_0 ^2\Jt^{-2})\cp{k} \p^4_t\vt^i$ and $-\kappa\int_\Omega\at_i^k \p^4_t[ \rho_0\Jt^{-1}(\rho_0\Jt_t)\cp{k}] \p^4_t\vt^i $, we find that 
	\begin{align*}
			{\int_\Omega \p^4_t\bset{\rho_0\vt^i_t } \p^4_t\vt^i} &- \underbrace{\int_\Omega \p^4_t\bset{\rho_0^2\Jt^{-2} }\at_i^k\p^4_t\vt^i\cp{k}}_{\I_1} + \kappa \underbrace{\int_\Omega \p^4_t\bset{\rho_0^2\Jt^{-1}\Jt_t }\at_i^k\p^4_t\vt^i\cp{k}}_{\I_2} \\
&			+ \underbrace{\int_\Gamma
                          \p^4_t\bset{\rho_0^2\Jt^{-2}
-                          \kappa \rho_0^2\Jt^{-1} \Jt_t}\at_i^k\p^4_t\vt^iN^k}_{\II}=\RK. 
	\end{align*}
	It is convenient to rewrite this equation as 
	\begin{align}
		\label{test:divergence} 
			\frac{1}{2} \frac{d}{dt} \int_\Omega
                        \rho_0\abs{\p^4_t \vt}^2 + \frac{d}{dt}
                        \int_\Omega {\rho_0^2\Jt^{-3}}\abs{\p^4_t\Jt}^2
                        + \kappa\int_\Omega \rho_0^2
                        \Jt^{-1}\abs{\p^5_t\Jt}^2 + \II  = \RK, 
	\end{align}
	where the identity $\Jt_t=\at_r^s\vt^r\cp{s}$ implies that the error created in order to write $\at_i^k\p^4_t\vt^i\cp{k}$ as $\p^5_t\Jt$ in $\I_1$ and $\I_2$ is of lower-order and so is absorbed in $\RK$.

	\subsection*{Rewriting the boundary integral in
          (\ref{test:divergence})} Using the outward normal identity
        (\ref{formula: outward normal}) and the boundary condition
        (\ref{test: twice rewritten boundary condition}) with 
        $\sigma=1$, we find
        that 
	\begin{align}
		\label{tv.4.i} 
		\begin{aligned}[b]
			{\II}
			&=\frac{1}{2}\frac{d}{dt}\int_\Gamma
                        \sqrt{\gt} \gt^{\alpha\beta} \,
                        \vt_{ttt}\cp{\alpha}\cdot \nt\,
                        \vt_{ttt}\cp{\beta}\cdot \nt  
+
\int_{\Gamma}
\sqrt{\gt}\gt ^{\alpha\beta} \sqrt{\kappa} \vt_{tttt} \cp{\alpha}\cdot \nt \sqrt{\kappa} \vt_{tttt} \cp{\beta}\cdot \nt
\\
			&\quad \underbrace{ -\int_\Gamma \sqrt{\gt}
                          \gt^{\alpha\beta} \vt_{ttt}\cp{\alpha}\cdot
                          \p_t\nt \,\vt_{ttt}\cp{\beta}\cdot \nt  }_{ \i_{1}} -
                          \frac{1}{2}
\underbrace{\int_\Gamma \p_t\bset{\sqrt{\gt}\gt^{\alpha\beta}}\, \vt_{ttt}\cp{\alpha}\cdot \nt \,\vt_{ttt}\cp{\beta}\cdot \nt }_{ \RK} \\
			&\quad+ \underbrace{ \int_\Gamma
                          \vt_{ttt}\cp{\beta}\cdot\nt                            \sqrt{\gt}\gt^{\alpha\beta}
                          \vt_{tttt} \cdot \nt \cp{\alpha}
+
\int_\Gamma         \vt_{ttt}\cp{\beta}\cdot\nt \cp{\alpha}                            \sqrt{\gt}\gt^{\alpha\beta}
                          \vt_{tttt} \cdot \nt 
                       }_{ \i_{2} }\\
	&\quad+ \underbrace{ \int_\Gamma  \vt_{ttt}\cp{\beta}\cdot \nt(
          \sqrt{\gt}\gt^{\alpha\beta})\cp{\alpha} \vt_{tttt} \cdot \nt
-
 \sum_{\counter=0}^{3}c_\counter\int_\Gamma \sqrt{\gt}\gt^{\alpha\beta} \p^\counter_t \et ^j\cp{\alpha\beta} \p^{4-\counter}_t\nt^j\,\vt_{tttt}\cdot \nt }_{\i_{3}}\\
			&\quad- \underbrace{ \sum_{\counter=1}^4c_\counter\int_\Gamma\p^\counter_t{\gt^{\alpha\beta}}\p^{4-\counter}_t\bset{\et \cp{\alpha\beta}\cdot \nt }\sqrt{\gt}\vt_{tttt}\cdot\nt }_{\i_{4}}\\
			&\quad
+
\underbrace{
                          \kappa\int_\Gamma\vt_{tttt} ^i\cp{\beta} \bpset{\nt^i
    \gt ^{\alpha\beta} \nt^j}
                          \cp{\alpha}                         \sqrt{\gt}\vt_{tttt}^j 
-\kappa
\sum_{\counter=1}^4c_\counter\int_\Gamma\p^\counter_t(
\gt^{\alpha\beta}\nt^i)\p^{4-\counter}_t\et \cp{\alpha\beta}
\sqrt{\gt}\vt_{tttt}\cdot \nt}_{\i_{5}}. 
		\end{aligned}
	\end{align}
	
	In our analysis of $\int_0^T\i_i$, $i=1,2,3$, we adopt the convention of letting 
	\begin{align*}
		\text{ $\ellt$ denote a function of $L^\infty(\Gamma)$-class and $\elt$ a function of $H^{0.5}(\Gamma)$-class.} 
	\end{align*}
	We recycle the symbols $\j$, $\j_A$, $\j_B$, etc.\ in terming boundary integrals that require explanation. 
	
	\subsection*{Analysis of $\int_0^T \i_1$ in the time-integral
          of \eqref{tv.4.i}} The action of $\hd_ \alpha\p^2_t$ in the tangential
        identity (\ref{indentity: smoothed tangential}) provides that 
	\begin{align}
		\label{tagential: hd v_ttt} \vt_{ttt}\cp{\alpha}\cdot
                \et \cp{\gamma}=\rho_0^{-1}\Jt
                \bset{\hd_{\alpha\gamma}\big(\gt^{\mu\nu}\et\cp{\mu\nu}\cdot\nt+
                  \kappa \gt ^{\mu\nu} \vt \cp{\mu\nu}\cdot \nt-\beta(t) \big)_{tt} - \elt_{\gamma\alpha}}, 
	\end{align}
	where the lower-order $\elt_{\gamma\alpha}\in H^{0.5}(\Gamma)$ is given by 
	\begin{align*}
		\elt_{\gamma\alpha}= \kappa\hd_\alpha\big[(\rho_0\Jt^{-1})\cp{\gamma} \rho_0\Jt_t\big]_{tt} + \vt_{ttt}\cdot\pset{\et\cp{\gamma}\rho_0 \Jt^{-1}}\cp{\alpha} +\sum_{a=1}^2c_a\bset{\p^a_t\vt\cdot\p^{3-a}_t\pset{\et\cp{\gamma}\rho_0 \Jt^{-1}} }\cp{\alpha}. 
	\end{align*}
	Setting
        $\ellt_{\gamma}^{\alpha\beta}=\rho_0^{-1}\Jt\sqrt{\gt}\gt^{\alpha\beta}\gt^{\gamma\delta}\vt\cp{\delta}\cdot\nt$
        we use the tangential identity (\ref{tagential: hd v_ttt}),
        together with the outward normal differentiation formula  \eqref{t derivative: outward normal vector}, to find that 
	\begin{align}
\label{r1.j.t}
          \begin{aligned}[b]
            \i_1& = \underbrace{ \int_\Gamma \ellt_{\gamma}^{\alpha\beta}
              \hd_{\gamma\alpha}\bset{\gt^{\mu\nu}\et\cp{\mu\nu}\cdot\nt}_{tt}\,
              \vt_{ttt}\cp{\beta}\cdot\nt}_{\i_1{}'} +
            \underbrace{\kappa \int_\Gamma
              \ellt_{\gamma}^{\alpha\beta}
              \hd_{\gamma\alpha}\bset{\gt^{\mu\nu}\vt\cp{\mu\nu}\cdot\nt}_{tt}\,
              \vt_{ttt}\cp{\beta}\cdot\nt}_{\j} +\RK.
          \end{aligned}
	\end{align}
	
	Concerning $\i_1{}'$, we integrate by parts with respect to a time-derivative to write 
	\begin{align*}
			\int_0^T&\i_1{}'=-\underbrace{\int_\Gamma \hd_{\gamma}\bset{\gt^{\mu\nu}\et\cp{\mu\nu}\cdot\nt}_{tt}\, \hd_\alpha\bset{\ellt_{\gamma}^{\alpha\beta} \vt_{tt}\cp{\beta}\cdot\nt }}_{\j_A}\Big|_0^T + \int_0^T\underbrace{\int_\Gamma \hd_{\gamma}\bset{\gt^{\mu\nu}\et\cp{\mu\nu}\cdot\nt}_{tt}\, \hd_\alpha\bset{\ellt_{\gamma}^{\alpha\beta} \vt_{tt}\cp{\beta}\cdot\nt_t}}_{\j_B}\\
			& + \int_0^T \underbrace{\int_\Gamma \hd_{\gamma}\bset{\gt^{\mu\nu}\et\cp{\mu\nu}\cdot\nt}_{tt}\, \hd_\alpha\bset{\p_t\ellt_{\gamma}^{\alpha\beta} \vt_{tt}\cp{\beta}\cdot\nt}}_{\j_C} + \int_0^T \underbrace{ \int_\Gamma \hd_{\gamma}\bset{\gt^{\mu\nu}\et\cp{\mu\nu}\cdot\nt}_{ttt}\, \hd_\alpha\bset{\ellt_{\gamma}^{\alpha\beta} \vt_{tt}\cp{\beta}\cdot\nt}}_{\j_D}, 
	\end{align*}
	where we have further integrated by parts with respect to $\hd_\alpha$ on the right-hand side. 
	
Interpolation and two applications of Young's inequality provide that
	\begin{align*}
		\abs{\vt_{tt}\cp{\beta}\cdot\nt(T)}_1^2 \le C\abs{\vt_{tt}\cdot\nt(T)}_{1.5} \abs{\hd\vt_{tt}\cdot\nt(T)}_{1.5} \le C \abs{\vt_{tt}\cdot\nt(T)}_{0.5}^{2}+2\delta \abs{\hd\vt_{tt}\cdot\nt(T)}_{1.5}^2.
	\end{align*}
Hence,
	\begin{align*}
		\j_A(t)\big|_0^T\le \int_{0}^{T}\RK. 
	\end{align*}
We record that via  Lemma~\ref{lem: tangential trace},
	\begin{align*}
		\abs{\vt_{tt}\cp{\alpha\beta}\cdot\et\cp{\sigma} }_{H^{-0.5}(\Gamma)}^2\le C \Bset{\norm{\hd^2 \vt_{tt}}_0^2+\norm{\curl \hd \vt_{tt}}_0^2 }\le C\norm{\vt_{tt}}_2^2.
	\end{align*}
Since $\nt_t=-\et\cp{\sigma}\gt^{\sigma\rho}\vt\cp{\rho}\cdot\nt$ and
$\hd^2\vt_t\cdot\nt$ is in $H^{1.5}(\Gamma)$, we
use an $H^{-0.5}(\Gamma)$-duality pairing in the highest-order term to write 
	\begin{align*}
		\int_0^T\j_B \le \int_0^T\RK +
                C\int_{0}^{T}\abs{\hd^2\vt_t \cdot \nt(t)}_{1.5}
                \abs{\vt_{tt}\cp{\alpha\beta}\cdot\et\cp{\sigma}
                }_{H^{-0.5}(\Gamma)} \le \int_{0}^{T}\RK.
	\end{align*}
	We use the Cauchy-Schwarz inequality for the estimate 
	\begin{align*}
		\int_0^T\j_C \le \int_{0}^{T}\RK. 
	\end{align*}
	
	Regarding $\int_0^T\j_D$, for the integral
        $\int_0^T\int_\Gamma
        \hd_{\gamma}\bset{\gt^{\mu\nu}\et\cp{\mu\nu}\cdot\nt}_{ttt}\,
        \hd_\alpha\ellt_{\gamma}^{\alpha\beta}
        \vt_{tt}\cp{\beta}\cdot\nt$, we integrate by parts with
        respect to $\hd_\gamma$ and estimate using the Cauchy-Schwarz
        inequality, and for the integral $\int_0^T\int_\Gamma
        \hd_{\gamma}\bset{\gt^{\mu\nu}\et\cp{\mu\nu}\cdot\nt}_{ttt}\,
        \ellt_{\gamma}^{\alpha\beta}
        \vt_{tt}\cp{\beta}\cdot\nt\cp{\alpha}$, we also integrate by
        parts with respect to $\hd_\gamma$ and then estimate using an
        $H^{-0.5}(\Gamma)$-duality pairing. These methods work equally
        well in all terms of the spacetime-integral
        $\int_0^T\int_\Gamma
        \hd_{\gamma}\bset{\gt^{\mu\nu}\et\cp{\mu\nu}\cdot\nt}_{ttt}\,
        \ellt_{\gamma}^{\alpha\beta}
        \vt_{tt}\cp{\alpha\beta}\cdot\nt$, as well as in
        $\int_{0}^{T}\j$, where $\j$ is defined in \eqref{r1.j.t}.
	 Thus, we have established that 
	\begin{align}
		\label{boundary integral: error i_1} \int_0^T \i_{1}\, \le \int_{0}^{T}\RK. 
	\end{align}
	
	\subsection*{Analysis of $\int_0^T\i_2$ in the time-integral
          of \eqref{tv.4.i}} We have that
	\begin{align}
		\label{boundary: rewrite i_2} 
		\begin{aligned}[b]
			\i_2&= \underbrace{ \int_\Gamma
                          \vt_{ttt}\cp{\beta}\cdot\nt                            \sqrt{\gt}\gt^{\alpha\beta}
                          \vt_{tttt} \cdot \nt \cp{\alpha}
}_{\i_{2a}} 
+
\underbrace{
\int_\Gamma         \vt_{ttt}\cp{\beta}\cdot\nt \cp{\alpha}                            \sqrt{\gt}\gt^{\alpha\beta}
                          \vt_{tttt} \cdot \nt 
                       }_{ \i_{2b} }.
		\end{aligned}
	\end{align}
	
We write the action of $\p^3_t$ in  the tangential identity
\eqref{indentity: smoothed tangential} as
\begin{subequations}
\label{tid.v31}
  \begin{align}
    \label{tangential identity: v_ttt} \vt_{tttt}\cdot\et \cp{\gamma}
    = \rho_0^{-1}\Jt\Bset{\hd_{\gamma}\big(
      \gt^{\mu\nu}\et\cp{\mu\nu}\cdot\nt
      +\kappa\gt^{\mu\nu}\vt\cp{\mu\nu}\cdot\nt-\beta(t) \big)_{ttt}
      -\elt_\gamma},
  \end{align}
  where the lower-order $\elt_\gamma\in H^{0.5}(\Gamma)$ is given by
  \begin{align}
    \label{notation: l_gamma} \elt_\gamma =
    \kappa\big[(\rho_0\Jt^{-1})\cp{\gamma} \rho_0\Jt_t\big]_{ttt} +
    \sum_{\counter=1}^3c_a\p^a_t\vt\cdot\p^{4-\counter}_t\pset{\et\cp{\gamma}\rho_0
      \Jt^{-1}} .
  \end{align}
\end{subequations}
	Letting
        $\ellt_\gamma^\beta=-\rho_0^{-1}\Jt\sqrt{\gt}\gt^{\alpha\beta}\gt^{\gamma\delta}\et
        \cp{\delta\alpha}\cdot\nt$,	we use \eqref{tid.v31} to write
	\begin{align*}
		\i_{2a} &
		= - \underbrace{ \int_\Gamma \ellt_\gamma^\beta \, \bset{\gt^{\mu\nu}\et\cp{\mu\nu}\cdot\nt}_{ttt} \vt_{ttt}\cp{\beta\gamma}\cdot\nt }_{ \i_{2a}{}'} + \kappa \underbrace{ \int_\Gamma \ellt_\gamma^\beta\,\hd_{\gamma}[ \gt^{\mu\nu}\vt_{ttt}\cp{\mu\nu}\cdot\nt] \,\vt_{ttt}\cp{\beta}\cdot\nt }_{ \i_{2a}{}''}+\RK, 
	\end{align*}
	where we have used integration by parts with respect to $\hd_\gamma$ to determine $\i_{2a}{}'$. We integrate by parts with respect to time in order to write 
	\begin{align*}
		- \int_0^T \i_{2a}{}' =&-\int_\Gamma \ellt_\gamma^\beta\,\vt_{tt}\cp{\beta\gamma}\cdot\nt\, \bset{\gt^{\mu\nu}\et\cp{\mu\nu}\cdot\nt}_{ttt} \Big|_0^T + \int_0^T\int_\Gamma \ellt_\gamma^\beta\, \vt_{tt}\cp{\beta\gamma}\cdot\nt_t\, \bset{\gt^{\mu\nu}\et\cp{\mu\nu}\cdot\nt}_{ttt} \\
		&+ \int_0^T\int_\Gamma \p_t\ellt_\gamma^\beta\,\vt_{tt}\cp{\beta\gamma}\cdot\nt\, \bset{\gt^{\mu\nu}\et\cp{\mu\nu}\cdot\nt}_{ttt} + \int_0^T\int_\Gamma \ellt_\gamma^\beta\,\vt_{tt}\cp{\beta\gamma}\cdot\nt \,\bset{\gt^{\mu\nu}\et\cp{\mu\nu}\cdot\nt}_{tttt} \\
		=& \int_0^T\underbrace{ \int_\Gamma \ellt_\gamma^\beta\,\vt_{tt}\cp{\beta\gamma}\cdot\nt \,\gt^{\mu\nu}\vt_{ttt}\cp{\mu\nu}\cdot\nt }_{\j} + \int_0^T\RK. 
	\end{align*}
	Letting $\ellt=\rho_0^{-1}\Jt\sqrt{\gt}\,\gt^{\gamma\delta}\et\cp{\gamma\delta}\cdot\nt$ we utilize the symmetry of $\ellt_\gamma^\beta$ to exchange $\hd_\alpha$ and $\hd_\gamma$ via integration by parts for 
	\begin{align}
		\label{boundary: symmetry} 
		\begin{aligned}[b]
			\int_0^T\j = &\frac{1}{2}\int_\Gamma \ellt\,\abs{\gt^{\mu\nu}\vt_{tt}\cp{\mu\nu}\cdot\nt}^2\Big|_0^T \\
			& - \int_0^T \int_\Gamma \ellt\,\vt_{tt}\cp{\alpha\beta}\cdot\nt\, \vt_{tt}\cp{\mu\nu}\cdot\pset{ n\gt^{\mu\nu}}_t - \frac{1}{2}\int_0^T\int_\Gamma \ellt_t\,\abs{\gt^{\mu\nu}\vt_{tt}\cp{\mu\nu}\cdot\nt}^2=\int_0^T\RK. 
		\end{aligned}
	\end{align}	Next, integrating by parts in time, we find that 
	\begin{align*}
		\kappa\int_0^T \i_{2a}{}'' &= \int_\Gamma \ellt_\gamma^\beta\,\hd_{\gamma}[ \gt^{\mu\nu}\vt_{tt}\cp{\mu\nu}\cdot\nt] \,\kappa(\vt_{ttt}\cp{\beta}\cdot\nt) \Big|_{0}^T \\
		+ \int_0^T &\int_\Gamma [\gt^{\mu\nu}\vt_{tt}\cp{\mu\nu}\cdot\p_t\nt] \hd_\gamma[ \ellt_\gamma^\beta\, \,\kappa(\vt_{ttt}\cp{\beta}\cdot\nt)] - \int_0^T \int_\Gamma \ellt_\gamma^\beta\,\hd_{\gamma}[ \gt^{\mu\nu}\kappa(\vt_{tt}\cp{\mu\nu}\cdot\nt)] \,\vt_{ttt}\cp{\beta}\cdot\p_t\nt \\
		& - \int_0^T \underbrace{ \int_\Gamma \ellt_\gamma^\beta\,\hd_{\gamma}[ \gt^{\mu\nu}\sqrt{\kappa}\vt_{tt}\cp{\mu\nu}\cdot\nt] \,\sqrt\kappa\vt_{tttt}\cp{\beta}\cdot\nt }_{\j} =\int_0^T\RK+\int_0^T\j. 
	\end{align*}
Interpolation and	Young's inequality provide for the estimate 
	\begin{align*}
		\int_0^T \j &\le \delta\int_0^T\abs{\sqrt{\kappa}\hd \vt_{tttt}\cdot\nt}_0^2 + C_\delta \int_0^T\abs{\sqrt{\kappa}\hd\vt_{tt}\cdot\nt}_2^2 \\
		&\le \delta\int_0^T\abs{\sqrt{\kappa}\hd \vt_{tttt}\cdot\nt}_0^2 + C_\delta \,T\,\abs{\hd\vt_{tt}\cdot\nt}_{1.5}\abs{\kappa\hd\vt_{tt}\cdot\nt}_{2.5}\le \int_0^T\RK. 
	\end{align*}	
	We have thus established that 
	\begin{align*}
		\int_0^T \i_{2a}=\int_0^T\RK. 
	\end{align*}
	
	The analysis of $\int_0^T\i_{2b}$ is similar. We set $\ellt^{\gamma\beta}=\sqrt{\gt}\gt^{\alpha\beta}\gt^{\gamma\delta}\et \cp{\delta\alpha}\cdot\nt$ and write 
	\begin{align*}
		\int_0^T\i_{2b} = \int_\Gamma \vt_{ttt}\cp{\beta}\cdot\nt\cp{\alpha}\,\sqrt{\gt}\gt^{\alpha\beta\,}\vt_{ttt}\cdot\nt\Big|_0^T &- \int_0^T \int_\Gamma\vt_{ttt}^j\cp{\beta}\bset{\nt^j\cp{\alpha}\sqrt{\gt}\gt^{\alpha\beta} \nt^i }_t\, \vt^i_{ttt} \\
		+ \int_0^T \underbrace{ \int_\Gamma \ellt^{\gamma\beta}\,\vt_{tttt}\cp{\beta}\cdot\et \cp{\gamma}\,\vt_{ttt}\cdot\nt}_{\i_{2b} {}'} &= \int_0^T\RK+\int_0^T \i_{2b} {}'. 
	\end{align*}
	Regarding $\int_0^T \i_{2b} {}'$, we use the identity
	\begin{align}
		\label{tangential identity: hd v_ttt} 
			\vt_{tttt}\cp{\beta}\cdot\et \cp{\gamma} = \rho_0^{-1}\Jt\bset{\hd_\beta \hd_\gamma \big( \gt^{\mu\nu}\et\cp{\mu\nu}\cdot\nt+\kappa\gt^{\mu\nu}\vt\cp{\mu\nu}\cdot\nt -\beta(t) \big)_{ttt}-\elt'_{\beta\gamma}}, 
	\end{align}
	where
        $\elt'_{\beta\gamma}=\vt_{tttt}\cdot\et\cp{\gamma\beta}\rho_0
        \Jt^{-1}+ \vt_{tttt}\cdot\et\cp{\gamma}\pset{\rho_0
          \Jt^{-1}}\cp{\beta}+\hd_\beta \elt_\gamma$, with
        $\elt_\gamma$ given by (\ref{notation: l_gamma}). We integrate by parts with respect to both $\hd_\beta$ and
        $\hd_\gamma$ in $\int_0^T \int_\Gamma\rho_0^{-1}\Jt
        \hd_{\beta\gamma} [
        \gt^{\mu\nu}\et\cp{\mu\nu}\cdot\nt]_{ttt}\,
        \ellt^{\gamma\beta}\,\vt_{ttt}\cdot\nt$, where the
        highest-order term produced by $\hd_{\beta\gamma}$-integration
        by parts is (\ref{boundary: symmetry}). To estimate the
        integral where $\hd_\beta \elt_\gamma$ appears, we have the
        choice of integration by parts with respect to $\hd_\beta$ or
        an $H^{-0.5}(\Gamma)$-duality pairing. In the integral where
        $\vt_{tttt}\cdot\et\cp{\gamma}\pset{\rho_0
          \Jt^{-1}}\cp{\beta}$ appears, we use the identity
        (\ref{tangential identity: v_ttt}). Hence, we conclude that
	\begin{align*}
		\int_0^T\i_{2b} {}'&= \int_0^T\RK - \int_0^T\underbrace{
                \int_\Gamma \ellt^{\gamma\beta}\, \hd_\beta
                \elt_\gamma\,\vt_{ttt}\cdot\nt  }_{\RK}.
	\end{align*}
Thus, we have established that 
	\begin{align}
		\label{boundary integral: error i_2} \int_0^T \i_{2}\, \le \int_{0}^{T}\RK. 
	\end{align}

        \subsection*{Analysis of $\int_{0}^{T}\i_{3}$  in the time-integral
          of \eqref{tv.4.i}}
We have that
        \begin{align*}
          \i_3= \underbrace{ \int_\Gamma  \vt_{ttt}\cp{\beta}\cdot \nt(
          \sqrt{\gt}\gt^{\alpha\beta})\cp{\alpha} \vt_{tttt} \cdot \nt }_{\i_{3a}}- \sum_{\counter=0}^{3}c_\counter\underbrace{\int_\Gamma \sqrt{\gt}\gt^{\alpha\beta} \p^\counter_t \et ^j\cp{\alpha\beta} \p^{4-\counter}_t\nt^j\,\vt_{tttt}\cdot \nt }_{\i_{3b,\counter}}.
        \end{align*}
	We will first establish that $\int_0^T\i_{3a}$ cancels with a
        term arising in $\int_0^T\i_{3b,0}$. Recalling that $\nt{}_t=-\gt^{\gamma\delta}\bset{\vt\cp{\delta}\cdot\nt}\et \cp{\gamma}$, the highest-order term of $\nt{}_{tttt}$ is $-\gt^{\gamma\delta}\bset{\vt{}_{ttt}\cp{\delta}\cdot\nt}\et \cp{\gamma}$. It follows that 
	\begin{align}
		\label{h.o. counter=0 term} 
			-\int_0^T\i_{3b,0} &= \int_0^T \underbrace{ \int_\Gamma \sqrt{\gt}\,\gt^{\alpha\beta} \gt^{\gamma\delta}\et \cp{\alpha\beta}\cdot \et \cp{\gamma} \vt{}_{ttt}\cp{\delta}\cdot\nt\,\vt_{tttt}\cdot\nt }_{\j} + \int_0^T\RK,
	\end{align}
	where the estimation of the lower-order terms of
        $\int_0^T\i_{3b,0}$ requires integration by parts with respect
        to a time-derivative of $\vt_{tttt}$. Using the differentiation formulas (\ref{formula: derivative: metric inverse}) and (\ref{formula: derivative: Jacobian determinant}), we write 
	\begin{align*}
		\j= -\underbrace{ \int_\Gamma \vt_{ttt}\cp{\beta}\cdot\nt \pset{\sqrt{\gt}\gt^{\alpha\beta}}\cp{\alpha}\vt_{tttt}\cdot\nt }_{\i_{3a}}. 
	\end{align*}
Thus,
	\begin{align*}
		\int_0^T\i_{3a} - \int_0^T\i_{3b,0}=\int_0^T\RK. 
	\end{align*}
	
	The terms $\i_{3b,1}$ and $\i_{3b,2}$ are analyzed by
        integrating by parts with respect to a time-derivative of
        $\vt_{tttt}$, followed by elementary estimates. Thus, 
	\begin{align*}
		\int_0^T\i_{3b,1}+ \int_0^T\i_{3b,2}=\int_0^T\RK. 
	\end{align*}
	
	We next examine $\i_{3b,3}=\int_\Gamma \sqrt{\gt}\gt^{\alpha\beta} \vt_{tt}\cp{\alpha\beta}\cdot\nt_t\,\vt_{tttt}\cdot\nt$. After integration by parts with respect to a time-derivative of $\vt_{tttt}$, we find that 
	\begin{align*}
\i_{3b,3} &=- \int_\Gamma \sqrt{\gt}\gt^{\alpha\beta} \vt_{ttt}\cp{\alpha\beta}\cdot\nt_t\,\vt_{ttt}\cdot\nt + \RK = \underbrace{\int_\Gamma \sqrt{\gt}\gt^{\alpha\beta} \vt_{ttt}\cp{\beta}\cdot\nt_t\,\vt_{ttt}\cp{\alpha}\cdot\nt }_{\i_{3b,3}{}'} + \RK.
	\end{align*}
 Letting
 $\ellt^\gamma_{\alpha\beta}=\sqrt{\gt}\gt^{\alpha\beta}\gt^{\gamma\delta}\vt\cp{\delta}\cdot\nt$,
 we have that
	\begin{align*}
\int_{0}^{T}			\i_{3b,3}{}' = \int_\Gamma \ellt^\gamma_{\alpha\beta} \, \vt_{ttt}\cp{\beta}\cdot \et \cp{\gamma}\vt_{tt}\cp{\alpha}\cdot\nt\Big|_0^T &- \int_0^T \int_\Gamma \vt_{ttt}\cp{\beta}^j\vt^i_{tt}\cp{\alpha}\cdot \pset{\nt^i\, \et ^j\cp{\beta}\, \ellt^\gamma_{\alpha\beta} }_t \\
			&- \int_0^T \underbrace{ \int_\Gamma \ellt^\gamma_{\alpha\beta} \,\vt_{tttt}\cp{\beta}\cdot \et \cp{\gamma}\vt_{tt}\cp{\alpha}\cdot\nt }_{\j} = - \int_0^T\j + \int_0^T\RK. 
	\end{align*}
	By using the identity (\ref{tangential identity: hd v_ttt}) in
        $\int_0^T\j$, we note that the   term corresponding with $\tilde l_\gamma$ has an elementary
        estimate after integration by parts with respect to
        $\hd_\beta$. The remaining terms are similarly analyzed,
        except for $ \int_0^T \int_\Gamma
        \ellt^\gamma_{\alpha\beta}
        \rho_0^{-1}\Jt\hd_{\beta\gamma}\bset{\gt^{\mu\nu}\vt_{tt}\cp{\mu\nu}\cdot\nt}\,
        \vt_{tt}\cp{\alpha}\cdot\nt$ in which we integrate by parts with respect to both
        $\hd_\beta$ and $\hd_\gamma$. For a certain highest-order term created by $\hd_{\beta\gamma}$-integration by parts, we form an exact derivative: setting 
        $\ellt_\gamma=\rho_0^{-1}\Jt\sqrt{\gt}\gt^{\gamma\delta}\vt\cp{\delta}\cdot\nt$,
	\begin{align*}
\int_0^T\int_\Gamma \ellt_\gamma\,\bset{\gt^{\alpha\beta} \vt_{tt}\cp{\alpha\beta}\cdot\nt}\cp{\gamma} \gt^{\mu\nu}\vt_{tt}\cp{\mu\nu}\cdot\nt = -\frac{1}{2}\int_0^T \int_\Gamma \hd_\gamma{\ellt_\gamma}\abs{ \gt^{\mu\nu}\vt_{tt}\cp{\mu\nu}\cdot\nt }^2.
	\end{align*}
This
        establishes that $\int_0^T\j=\int_0^T\RK$. We thus conclude that
	\begin{align*}
		\int_0^T\i_{3b,3}=\int_0^T\RK. 
	\end{align*}
		Hence, 
	\begin{align}
		\label{boundary integral: error i_3} \int_0^T \i_{3}\, \le \int_{0}^{T}\RK. 
	\end{align}
	
	\subsection*{Analysis of $\int_0^T\i_4$ in the time-integral
          of \eqref{tv.4.i}} We integrate by parts with respect to a
        time derivative of $\vt_{ttt}$ in   
        $\int_0^T\i_4$ and, if need be, spatially integrate by parts. For example, letting $\i_{4}=\sum_{\counter=1}^4 \i_{4,\counter}$, we find that after integration by parts with respect to time, 
	\begin{align*}
		\int_0^T\i_{4,1}&= \int_0^T\int_\Gamma \pset{\sqrt{\gt} \gt^{\alpha\beta}}_t \pset{\et \cp{\alpha\beta}\cdot\nt }_{tttt} \,\vt_{ttt}\cdot\nt+\int_0^T\RK \\
		&= - \int_0^T\int_\Gamma \vt_{ttt} \cp{\alpha}\cdot\hd_\beta[\nt \,\pset{\sqrt{\gt} \gt^{\alpha\beta}}_t \vt_{ttt}\cdot\nt]+\int_0^T\RK =\int_0^T\RK. 
	\end{align*}
	
	Similarly, integration by parts with respect to time provides for the expression 
	\begin{align*}
		\int_0^T\i_{4,2}+\int_0^T\i_{4,3}=\int_0^T\RK. 
	\end{align*}
	
	Finally, using the differentiation formulas (\ref{formula: derivative: metric inverse}) and (\ref{formula: derivative: Jacobian determinant}), 
	\begin{align*}
		\int_0^T \i_{4,4}= \int_{0}^{T}\int_\Gamma \sqrt{\gt} (2 \gt^{\alpha\mu} \gt^{\nu\beta}-\gt^{\alpha\beta} \gt^{\mu\nu})\,\vt{}_{ttt}\cp{\mu}\cdot \et \cp{\nu}\,\et \cp{\alpha\beta}\cdot\nt\, \vt_{tttt}\cdot\nt + \int_0^T\RK = \int_0^T\RK, 
	\end{align*}
	where the second equality follows from our above analysis of $\int_0^T\i_{2b}$.
	
	Hence, 
	\begin{align}
		\label{boundary integral: error i_4} \int_0^T \i_4\, \le \int_{0}^{T}\RK. 
	\end{align}
	
	\subsection*{Analysis of $\int_0^T\i_5$  in the time-integral
          of \eqref{tv.4.i}} In the
        first term defining $\i_5$, we integrate by parts
        with respect to $\hd_ \alpha$ when $\hd_\beta$ acts on
        $\nt^j$. This term is bounded by $\int_{0}^{T}\RK$, since
        thanks to Lemma~\ref{lem.j}, 
	\begin{align*}
		\int_0^T\abs{\sqrt{\kappa}\vt_{tttt}}_0^2\le C 		\int_0^T\abs{\sqrt{\kappa}\vt_{tttt}}_{0.25}^2\le \delta
                \int_0^T \norm{\sqrt{\kappa}\vt_{tttt}}_1^2 +C_ \delta\,T\Sup\norm{\vt_{tttt}(t)}_0^2.
	\end{align*}
The remaining terms are similarly estimated. 
	Hence, 
	\begin{align}
		\label{boundary integral: error i_5} \int_0^T \i_5\, \le \int_{0}^{T}\RK. 
	\end{align}

	\subsection*{Rewriting the equation (\ref{test:divergence})} Summing the inequalities (\ref{boundary integral:
          error i_1}), (\ref{boundary integral: error i_2}),
        (\ref{boundary integral: error i_3}), (\ref{boundary integral:
          error i_4}),  (\ref{boundary integral: error i_5}) yields
	\begin{align}
		\label{est: boundary-error} \sum_{i=1}^5\int_0^T\i_i\le \int_{0}^{T}\RK.
	\end{align}
Thanks to the inequality \eqref{est: boundary-error} and the
identity (\ref{tv.4.i}), we equivalently write (\ref{test:divergence}) as
	\begin{align}
		\label{energy equation} 
		\begin{aligned}[b]
			\frac{1}{2}\frac{d}{dt} \int_\Omega\rho_0
                        \abs{\p^4_t \vt}^2 + \frac{d}{dt} \int_\Omega
                        {\rho_0^2\Jt^{-3}}\abs{\p^4_t\Jt}^2 +
                        \frac{1}{2}\frac{d}{dt}\int_\Gamma \sqrt{\gt}
                        \gt^{\alpha\beta} \,
                        {\vt_{ttt}}\cp{\alpha}\cdot \nt&
                        \,{\vt_{ttt}}\cp{\beta}\cdot \nt \\
+ 
\int_{\Gamma} \sqrt{\kappa}\vt_{tttt}\cp{\beta}\cdot \nt\sqrt{\gt}\gt
^{\alpha\beta} \sqrt{\kappa}\vt_{tttt}\cp{\alpha}\cdot \nt
			+ \int_\Omega \rho_0^2\Jt^{-1} \abs{ \sqrt{\kappa}\p^5_t\Jt}^2 & = \RK.
		\end{aligned}
	\end{align}
The time-integral of (\ref{energy equation}) completes the proof. 
\end{proof}

Via Proposition~\ref{prop:Hodge}, the estimates of Lemma~\ref{lem:
  curl estimates} and  Proposition~\ref{prop: divergence: energy estimates} imply
the following 
\begin{prop}
	[The $\kappa$-independent estimates for $\vt_{ttt}$ and $\sqrt{\kappa} \vt_{tttt}$] \label{prop: estimate for v_ttt
          and sqrt kp v_tttt}  
	\begin{align*}
		\Sup\norm{\vt_{ttt}(t)}_1^2 + \int_{0}^{T}\norm{\sqrt{\kappa}\vt_{tttt}}_1^2\le\int_{0}^{T}\RK.
	\end{align*}
\end{prop}
\begin{proof}
	
	The fundamental theorem of calculus provides for 
	\begin{align}
		\label{identity: div and Lagrangian} \Div \p^a_t\et& =\at_r^s\p^a_t\et ^r\cp{s}-\p^a_t\et ^r\cp{s}\int_0^t\at_t{}_r^s,\ \ a=0,1,2,3,4. 
	\end{align}
The identity $\Jt_t=\at_r^s\vt^r\cp{s}$ implies that for
        $a=1,2,3$, $\p^a_t\Jt_t$ is equal to
        $\at_r^s\p^a_t\vt^r\cp{s}+\tilde\jmath$, where $\tilde \jmath$
        scales like $D\p^{a-1}_t\vt$. The  estimate for
        $\p^4_t\Jt$ stated in  Proposition~\ref{prop: divergence:
          energy estimates} therefore implies that
	\begin{align}
\label{ftc.div}
		\Sup \norm{\Div\vt_{ttt}(t)}_0^2\le \int_{0}^{T}\RK. 
	\end{align}
	 
 	Using the identity $N= \nt-\int_0^t\p_t\nt$, and the estimate 
	\begin{align*}
	\abs{\hd\p^a_t\vt\cdot \int_0^t\p_t \nt}_{2.5-a}^2\le T\,C  \norm{ \p^a_t\vt}_{4-a}^2\ \ \ \text{for $a=0,1,2,3$,}
	\end{align*}
	we infer from the trace-estimate stated in Proposition~\ref{prop: divergence: energy estimates} that 
	\begin{align*}
			\Sup\abs{\hd\vt_{ttt} (t)\cdot N}^2_{-0.5}\le \int_{0}^{T}\RK. 
	\end{align*}
	Thanks to the curl-estimates of Lemma~\ref{lem: curl
          estimates}, Proposition~\ref{prop:Hodge} therefore
        establishes that
        \begin{align*}
          \Sup \norm{\vt_{ttt}(t)}_1^2\le \int_{0}^{T}\RK.
        \end{align*}
A similar analysis establishes the $L^2(0,T;H^1(\Omega))$-estimate for $\sqrt{\kappa}\vt_{tttt}$.
\end{proof}

We now improve the normal trace-estimates of Proposition~\ref{prop: divergence: energy estimates}.
\begin{prop}[An improved $\kappa$-independent normal trace-estimate for $\vt_{ttt}$]
	\label{prop: lowest-order normal trace} 
	\begin{align*}
		\Sup\abs{\vt_{ttt}\cdot\nt (t)}_{1}^2 \le\int_{0}^{T}\RK. 
	\end{align*}
\end{prop}
\begin{proof}
	The estimate of Proposition~\ref{prop: estimate for v_ttt and sqrt kp v_tttt} together with the trace theorem implies that 
	\begin{align*}
		\Sup\abs{\vt_{ttt}}_{0.5}^2 \le\int_{0}^{T}\RK. 
	\end{align*}
	Combining this estimate with the normal trace-estimate of Proposition~\ref{prop: divergence: energy estimates} completes the proof. 
\end{proof}
\subsection*{Step 3: The $\kappa$-independent divergence- and normal trace-estimates} \label{sec:step-3:-divergence-1.k}
	We equivalently write the approximate surface tension problem (\ref{test: twice rewritten interior equation}) as 
	\begin{align*}
		2\rho_0\Jt^{-2}\At_i^k\Jt\cp{k}+\kappa\rho_0 \At_i^k\Jt_t\cp{k}=\vt^i_t+\bset{2\Jt^{-1}-\kappa \Jt_t}\At_i^k\rho_0\cp{k}, 
	\end{align*}
	or equivalently, setting $\kappa' = \kappa/2$,
	\begin{align}
		\label{Euler for d}\Jt
                ^{-2}\At_\cdot^k\Jt\cp{k}+\kappa'\At_\cdot^k\Jt_t\cp{k}=\underbrace{\tfrac{1}{2\rho_0}\bset{\vt_t+\bset{2\Jt^{-1}-\kappa
                      \Jt_t}\At_\cdot^k\rho_0\cp{k}}}_{\tilde{\mathcal{V}}
                _t}. 
	\end{align}
The fundamental theorem of calculus provides that \eqref{Euler for d}
is equivalently written as
	\begin{align}
		\label{Euler for divergence}
D\Jt+\kappa'D\Jt_t=\tilde{\mathcal{V}}
                _t - \underbrace{\bset{\Jt \cp{k } \int_{0}^{t}\p_t (\Jt
^{-2} \At_\cdot ^k)
+
\kappa \Jt_t \cp{k } \int_{0}^{t}\p_t \At_\cdot ^k}}_{\tilde \jmath}.
	\end{align}

\begin{lem}
	[Estimates for $\Jt_{ttt}$ and $\kappa\,\Jt_{tttt}$ via Lemma~\ref{kappa-indep: f+kappa f_t}]\label{lem: JT_ttt and kp JT_tttt}
	\begin{align*}
		\Sup\norm{\Jt_{ttt}(t)}_{1}^2 + \Sup\norm{\kappa\,\Jt_{tttt}(t)}_1^2 \le\int_{0}^{T}\RK.
	\end{align*}
\end{lem}
\begin{proof}
	Taking three time-derivatives of (\ref{Euler for divergence}) produces an equation of the form $f+\kappa f_t=g$: 
	\begin{align}
		\label{interior: two time derivatives} 
 D\Jt_{ttt}+ \kappa' \p_t ( D\Jt_{ttt})	= \p^3_t\tilde{\mathcal{V}}_{t}+\tilde{\jmath}_{ ttt}.
\end{align}
By \eqref{Euler for d} we have that $\p^3_t\tilde{\mathcal{V}}_{t}$ scales like
$\vt_{tttt}+ T\,D\vt_{ttt} + \kappa  \p^5_t\Jt$. Proposition~\ref{prop: estimate for v_ttt
          and sqrt kp v_tttt} therefore implies that $\norm{\p^3_t\tilde{\mathcal{V}}_{t} (t)}_0^2 \le \int_{0}^{T}\RK$. According to
        \eqref{Euler for divergence}, we have that $\tilde{\jmath}_{ ttt}$
        scales like $T \, D\Jt_{ttt}+ T\,\kappa D\Jt_{tttt} +
        T\,D\vt_{ttt} $. So, $\norm{\tilde{\jmath}_{ttt}(t)}_0^2 \le \int_{0}^{T}\RK$.
The fundamental theorem of calculus provides  a good
        estimate for $\Jt_{ttt}$ in $L^\infty(0,T;L^2(\Omega))$.
	Thus, Lemma~\ref{kappa-indep: f+kappa f_t} implies that 
	\begin{align*}
		\Sup \norm{\Jt_{ttt}(t)}_1^2\le\int_{0}^{T}\RK. 
	\end{align*}
We similarly infer from  equation (\ref{interior: two time
  derivatives}) and the fundamental theorem of calculus that 
	\begin{align*}
		\Sup \norm{\kappa \Jt_{tttt}(t)}_1^2\le\int_{0}^{T}\RK. 
	\end{align*}
	This completes the proof. 
\end{proof}

\begin{lem}
	[Normal trace-estimates for $\vt_{tt}$ and $\kappa \vt_{ttt}$] \label{lem:trace est
          v_t.n and kp v_tt.n} 
	\begin{align*}
		\Sup \abs{\hd\vt_{tt}\cdot\nt(t)}_{0.5}^2
+
		\Sup \abs{\kappa\hd\vt_{ttt}\cdot\nt(t)}_{0.5}^2&\le\int_{0}^{T}\RK.
	\end{align*}
\end{lem}
\begin{proof}
  The fundamental theorem of calculus implies the desired
  estimates. For example,
  \begin{align*}
    \abs{\hd\vt_{tt}\cdot \nt(t)}_{0.5}^2 &
\le    C
    \abs{\vt_{tt}\cdot \nt(t)}_{0.5}
    \abs{\hd\vt_{tt}\cdot \nt(t)}_{1.5}
\\
&\le
\NK + C_\delta\int_{0}^{T}\norm{ \vt_{ttt}}_1^2+
\delta\Sup \abs{\hd \vt_{tt}\cdot \nt(t)}_{1.5}^2\le \int_{0}^{T}\RK.
  \end{align*}
\end{proof}

\subsection*{Step 4: The $\kappa$-independent higher-order estimates
  via Proposition~\ref{prop:Hodge}}

The divergence- and normal trace-estimates obtained in Step~3,
together with the curl-estimates of Lemma~\ref{lem: curl estimates},
imply a good  estimate  for $\vt_{tt}$. Hence, successively repeating Step~3 yields
  \begin{align}
\label{pre.kp.ho}
\Sup\sum_{a=0}^2 \norm{\p^a_t \vt(t)}_{4-a}^2
+
\Sup
\sum_{a=0}^2 \norm{\kappa\p^a_t \vt_t(t)}_{4-a}^2
		\le \int_{0}^{T}\RK. 
        \end{align}

We recall the boundary condition \eqref{akp.bc} is
\begin{align}
	\label{boundary condition.k}
\gt^{\alpha\beta}\et\cp{\alpha\beta}\cdot \nt + \kappa
\gt^{\alpha\beta} \vt \cp{\alpha\beta}\cdot \nt
= \underbrace{ \beta(t)-\rho_0^2\Jt^{-2} +\kappa\rho_0^2\Jt^{-1} \Jt_t }_{\tilde\jmath}. 
\end{align}
Writing the boundary condition \eqref{boundary condition.k} as an equation of the form
$f+\kappa f_t =g$ yields 
\begin{align}
	\label{boundary condition.1}
\gt^{\alpha\beta}\et\cp{\alpha\beta}\cdot \nt + \kappa \p_t(
\gt^{\alpha\beta} \et \cp{\alpha\beta}\cdot \nt)
= {\tilde\jmath} + \kappa \et \cp{\alpha\beta}\cdot (\nt \gt ^{\alpha\beta})_t. 
\end{align}
As $\kappa\et$ is of the same regularity as $\et$, the flow map for the
$\kappa$-problem does
not witness an improved boundary-regularity.  The fundamental theorem
of calculus, however, does imply a good estimate for the right-hand
side of \eqref{boundary condition.1}. Hence, we infer from  Lemma~\ref{kappa-indep:
  f+kappa f_t} that the normal trace of $\et$ and $\kappa\vt$ each has
a good estimate in $L^\infty(0,T;H^{4.5}(\Gamma))$:
\begin{align}
  \label{pre.kp.ho.1}
  \Sup \norm{\et(t)}_5^2+  \Sup \norm{\kappa\vt(t)}_5^2 \le \int_{0}^{T}\RK.
\end{align}
We collect the estimates \eqref{pre.kp.ho} and \eqref{pre.kp.ho.1} in
the following
\begin{prop}
[Higher-order $\kappa$-independent estimates]\label{prop: first higher-order.k} 
  \begin{align*}
\Sup\sum_{a=0}^3 \norm{\p^a_t \et(t)}_{5-a}^2
+
\Sup
\sum_{a=0}^3 \norm{\kappa\p^a_t \vt_t(t)}_{5-a}^2
		\le \int_{0}^{T}\RK. 
        \end{align*}
\end{prop}

\subsection*{Step 5: The $\kappa$-independent improved
  boundary-regularity estimates} \label{sec:step-3:-improved.k}
With the estimates provided by  Proposition~\ref{prop: first
  higher-order.k} for $\norm{\p^a_t\et}_{5-a}^2$ and $\norm{\kappa\p^a_t\vt}_{5-a}^2$, $a=0,1,2,3$, we are in
position to establish the improved boundary-regularity estimates.
\begin{prop}
	[The $\kappa$-independent improved boundary-regularity via Lemma~\ref{kappa-indep: f+kappa f_t}] \label{prop: Improved boundary.k} 
	\begin{align*}
		 \Sup \sum_{a=0}^2
                 \abs{\hd^2\p^a_t\vt\cdot\nt(t)}_{2.5-a}^2&
+
\Sup \sum_{a=0}^2 \abs{\hd^2\p^a_t\vt_t\cdot\nt(t)}_{2.5-a}^2
\le\int_{0}^{T}\RK.
	\end{align*}
\end{prop}

\begin{proof}
Taking three time-derivatives of (\ref{boundary condition.k}), we obtain
	\begin{align}
		\label{three time-der: bc.k} 		\begin{aligned}[b]
\gt^{\alpha\beta}\vt_{tt}\cp{\alpha\beta}\cdot \nt
+
\kappa\p_t (\gt^{\alpha\beta}\vt_{tt}\cp{\alpha\beta}\cdot \nt)
			=\p^3_t\tilde\jmath
&- \underbrace{
 \bset{ \p^3_t(\gt^{\alpha\beta}\et\cp{\alpha\beta}\cdot
  \nt)
-
\gt^{\alpha\beta}\vt_{tt}\cp{\alpha\beta}\cdot
  \nt}
}_{\tilde\jmath_A} \\
&-\underbrace{
\bset{\kappa \p^3_t(\gt^{\alpha\beta}\vt\cp{\alpha\beta}\cdot
  \nt)
-
\kappa\p_t (\gt^{\alpha\beta}\vt_{tt}\cp{\alpha\beta}\cdot \nt)
}}_{\tilde\jmath_B} .  
\end{aligned}
\end{align}
  By   Proposition~\ref{prop: first higher-order.k}, the right-hand side
  of \eqref{three time-der: bc.k} has a good estimate in $L^\infty(0,T;H^{0.5}(\Gamma))$.
Lemma~\ref{kappa-indep: f+kappa f_t} therefore provides that 
\begin{align*}
  \Sup \abs{\hd^2 \vt_{tt}\cdot \nt(t)}_{0.5}^2\le \int_{0}^{T}\RK.
\end{align*}
Since $\kappa\p_t (\gt^{\alpha\beta}\vt_{tt}\cp{\alpha\beta}\cdot \nt)
=\kappa\gt^{\alpha\beta}\vt_{ttt}\cp{\alpha\beta}\cdot \nt
+\kappa \vt_{tt}\cp{\alpha\beta} \cdot(\nt\gt^{\alpha\beta})_t
$, it follows from \eqref{three time-der: bc.k} that 
\begin{align*}
  \Sup \abs{\kappa\hd^2 \vt_{ttt}\cdot \nt(t)}_{0.5}^2\le \int_{0}^{T}\RK.
\end{align*}
The higher-order estimates are similarly established.
\end{proof}  

\subsection*{Step 6: Concluding the proof of
  Lemma~\ref{lem.EK}} \label{sec:asymptotic-analysis.k} The sum of the
estimates given in Propositions~\ref{prop: divergence: energy
  estimates}--\ref{prop: Improved boundary.k} competes the proof of
Lemma~\ref{lem.EK}. Taking $\delta$ sufficiently small in the inequality \eqref{est: EK} 
yields a polynomial-type inequality of the form \eqref{poly-type
  ineq}. Hence, for
sufficiently small $T>0$ and independent of $\kappa>0$,
\begin{align}
	\label{the kappa-independent estimate} \Sup\EK\le2\NK,
\end{align} 
where the higher-order energy function $\EK$ is defined in \eqref{defn.kappa E(t)}.
  
\section{Construction of solutions to the $\kappa$-problem
  \eqref{akp}}
\label{sec:appr-kappa-probl}

In this section, we  prove the following
\begin{thm}
	[Solutions to the $\kappa$-problem] \label{thm.akp} For $C^\infty$-class initial data
        $(\rho_0,u_0, \Omega )$ satisfying the conditions \eqref{l.boundedness.r0} and
\eqref{s.compatibility conditions}, and for some
        $T=T_\kappa>0$, there exists a unique
        solution $\vt$ to the $\kappa$-problem \eqref{akp} verifying
$(\vt,\vt_t,\dots,\vt_{tttt} )|_{t=0} =(u_0, \mathrm{v}_1,\dots,\mathrm{v}_4)$   and
      \begin{align}
        \label{cont.est.v}
 \Sup \norm{\vt_{tttt}(t)}_{0}^2+\int_{0}^{T}\abs{\vt_{tttt}\cdot \nt}_1^2
 +\sum_{a=0}^4\int_{0}^{T}\norm{\p^a_t\vt}_{9-2a}^2<\infty. 
      \end{align}
\end{thm}
\begin{rem}
  We recall that the initial data $\mathrm{v}_a$, $a=1,2,3,4$, is
  defined in Section~\ref{sec:comp-cond-kappa}.
\end{rem}
We establish Theorem~\ref{thm.akp} via a
succession of two asymptotic estimates that correspond with two intermediate
approximate problems defined below.
Each of the two intermediate problems 
involves the use of a convolution operator that smooths in the 
direction tangent to the moving boundary. We  use 
$\epsilon>0$ and $\mu>0$ as the smoothing parameters.

Our first intermediate problem, which we call the
$\kappa\epsilon$-problem, is defined by smoothing the moving domain
of the $\kappa$-problem.  The overall structure of the
$\kappa\epsilon$-problem matches that of the $\kappa$-problem. It
naturally follows that
the $\epsilon$-independent a priori estimates closely resemble the
$\kappa$-independent a priori estimates of Section \ref{sec:kappa
  independent estimates}. The $\epsilon=0$ formal limit of the
$\kappa\epsilon$-problem is the $\kappa$-problem.

 Our second intermediate problem, which we call the  $\mu$-problem, is
defined by smoothing the $\kappa$-artificial viscosity term  $\kappa g _{\epsilon}
^{\alpha\beta} v\cp{\alpha\beta}\cdot n _{\epsilon}$ appearing in the boundary
condition of the $\kappa\epsilon$-problem. The  
$\mu$-problem is a nonlinear heat-type problem with Neumann-type boundary
conditions. The $\mu=0$ formal limit of
the heat-type $\mu$-problem is
equivalent to the $\kappa\epsilon$-problem. The key 
to obtaining the $\mu$-independent a
priori estimates is   that the diffusive term of
the heat-type $\mu$-problem yields a trace-estimate on the boundary $\Gamma$.

  \subsection{Horizontal convolution-by-layers}
\label{sec:horiz-conv-layers}

   For $\epsilon>0$, we let $0\le \bar{\varphi}_{\epsilon}\in
  C^\infty_0(\mathbb{R}^2)$ with
  $\supp(\bar{\varphi}_{\epsilon})\subset \overline{B(0,\epsilon)}$
  denote the family of standard mollifiers on $\mathbb{R}^2$. With
  $x_h=(x_1,x_2)$, we define the operation of \textit{horizontal
    convolution-by-layers} as follows:
  \begin{align*}
    \Lambda_{\epsilon}f(x_h,x_3)= \int_{\mathbb{R}^2}
    \bar{\varphi}_{\epsilon}(x_h-y_h) f(y_h,x_3)dy_h\ \ \text{for }
    f\in L^1_{loc}(\mathbb{R}^2).
  \end{align*}
  By standard properties of convolution, there exists a constant $C$
  which is independent of $\epsilon$, such that for $s\ge 0$,
  \begin{align*}
    \abs{\Lambda _{\epsilon}f}_s\le C \abs{f}_s\ \ \ \forall f\in
    H^s(\Gamma).
  \end{align*}
  Furthermore,
  \begin{align}
    \label{hc.est} \epsilon\abs{\Lambda _{\epsilon}f}_1\le C
    \abs{f}_0\ \ \ \forall f\in L^2(\Gamma).
  \end{align}
We recall the local coordinates near $\Gamma $ are defined in
Section~\ref{sec:local-coord-near}. We set
\begin{align}
\label{epsilon.0}
  \epsilon_0=\min_{\local=1}^K\dist(\supp\xl\circ\thetal,\p\VLP\setminus\DL).
\end{align}
We define the   horizontally-convolved vector field  ${v} _{\epsilon}$
on $\Gamma$
of a given vector   ${v}$  by\label{n.ve}
\begin{align}
  \label{horizon.convolved.v}
{v}_\epsilon= \sum_{l=1}^K\LE  [(\xl v  )\circ \thetal ] \circ
 \thetal^{-1}  \quad{\rm on}\ \Gamma.
\end{align} 
Given a sufficiently smooth  vector field $\vc$, we set $\check{\eta}=
e+\int_{0}^{t}\vc$ in $\Omega$ and
$\check{\eta}_{\epsilon}  =e+\int_0^t{\vc}_{\epsilon} $ on $\Gamma$. We define
$\ec$   to be the solution of the following time-dependent
elliptic Dirichlet problem:
\begin{subequations}
\label{elliptic.etm}
  \begin{alignat}{2}
    \label{etm.m}
  \Delta\ec&=\Delta\check{\eta}&& \text{in } \Omega,
\\
\ec&=\check{\eta}_{\epsilon}&\ \ \ \ &\text{on }\Gamma.
  \end{alignat}
\end{subequations}
We define the following  $\epsilon$-approximate Lagrangian variables:
\begin{align*} 
\Ac_{\epsilon}  =\bset{D\ec  }^{-1},\ \ \Jc_{\epsilon}  =\det D\ec  ,\ \ \ac_{\epsilon} =\Jc_{\epsilon}  {\Ac}_{\epsilon}  ,\ \bset{{\gc}_{\epsilon} }_{\alpha\beta}=\ec \cp{\alpha}\cdot\ec \cp{\beta}, \ {\rm and}\  \sqrt{{\gc}_{\epsilon} } {\nc}_{\epsilon}  =\bset{{\ac}_{\epsilon} }^TN.
\end{align*}

\subsection{The $\kappa\epsilon$-problem and its a
  priori estimates}
\label{sec:an-interm-parab}

We define an intermediate approximate problem, which we will refer to
as the $\kappa\epsilon$-problem, that is asymptotically
consistent with the $\kappa$-problem \eqref{akp}.  To indicate  the
dependence on the smoothing parameter $\epsilon$  of all the variables in the $\kappa\epsilon$-problem,
we place the symbol $\check{\ }$ over each of the
variables.    
In the $\kappa\epsilon$-problem, we smooth the
moving boundary  and use the corresponding
twice-mean-curvature function
\begin{align*}
  \sigma H(\ec)=- \sigma \gm ^{\alpha\beta}\ec \cp{\alpha\beta}\cdot \nm.
\end{align*}

\begin{defn}
	[The $\kappa\epsilon$-problem] \label{defn.mkp} For 
        $\kappa>0$ and $\epsilon>0$, we define $\vc$ as  the solution of 	\begin{subequations}
		\label{mkp} 
		\begin{alignat}
			{4} \label{mkp.momentum} \vc^i_t+2\Amu_i^k 
 \pset{\rho_0 \Jm^{-1}}\cp{k}-\kappa \Amu_i^k(\rho_0\JT)\cp{k}&=0&&in\ \Omega\times(0,T_\kappa(\epsilon)],\\
			\label{mkp.bc}
                        \rho_0^2 \Jm^{-2}-\kappa\rho_0^2\Jm^{-1}
                        \JT&=\Hm&\ \ &on\ \Gamma\times(0,T_\kappa(\epsilon)],\\
			\label{mkp: init cond} \pset{\ec,\vc}|_{t=0}&=\pset{e,u_0}&&on\ \Omega.
		\end{alignat}
	\end{subequations}
The function $\JT$ appearing in the left-hand side of equations
\eqref{mkp.momentum} and \eqref{mkp.bc} is defined as
\begin{align}
\label{mkp.J}
    \JT= \Jm\Div_{\ec}\vc.
\end{align}
 The function $\Hm$ appearing in the right-hand side of \eqref{mkp.bc}
is defined as
\begin{align}
\label{mkp.H}
  \Hm =\beta_{\epsilon}(t) -\sigma{\gm^{\alpha\beta}  {\ec
                          \cp{\alpha\beta}\cdot \nm} }
-
\kappa \gm^{\alpha\beta} \vc \cp{\alpha\beta}\cdot \nm.
\end{align}
The  function $\beta_{\epsilon}(t)$ appearing in
        the right-hand side of \eqref{mkp.H} is defined as
\label{n:beta.e zero} 
	\begin{align}
		\label{mkp:beta zero}
                \begin{aligned}[b]
                  \beta_{\epsilon}(t)=\beta&+\sum_{a=0}^{3}
                  \frac{t^a}{a!}\p^a_t\bset{ \rho _0^2 \Jm^{-2} - \beta+\sigma{\gm^{\alpha\beta}  {\ec
                      \cp{\alpha\beta}\cdot \nm} }
                    }
                  \\
&                  + \kappa \sum_{a=0}^{3} \frac{t^a}{a!}\p^a_t\bset{- \rho_0^2\Jm^{-1}\JT+
                    \gm^{\alpha\beta} \vc \cp{\alpha\beta}\cdot
                    \nm}|_{t=0}.
                \end{aligned}
	\end{align}
\end{defn}
\begin{rem}
Elliptic regularity provides that 
  The  $\epsilon$-approximate Lagrangian flow map $\ec$ defined by \eqref{elliptic.etm} satisfies $\norm{\ec}_s\le C \norm{\check{\eta}}_s$
  for a positive constant $C$ independent of $\epsilon$. Standard
  properties of convolution provide that setting
  $\epsilon=0$  in \eqref{elliptic.etm} yields $\ec=\check{\eta}$. It follows
  that the $\kappa\epsilon$-problem \eqref{mkp} is asymptotically consistent with
  the $\kappa$-problem \eqref{akp}.
\end{rem}

\subsubsection{The constant-in-time vectors $\mathbf{v}_a$ for $a=1,2,3,4$} \label{sec:comp-cond-mu}   Letting $\vc$ solve
\eqref{mkp}, the vector field $\p^a_t\vc|_{t=0}$  for all $a\in \mathbb{N}$
is computed as follows:
\begin{align*} 
	\p^a_t\vc|_{t=0}=\frac{\p^{a-1}}{\p t^{a-1}}\Pset{\kappa
          \Amu_\cdot^k(\rho_0 \JT)\cp{k}
          -2\Amu_\cdot^k\pset{\rho_0\Jm^{-1}}\cp{k}}\big|_{t=0}\ \ \
        \text{on } \Omega. 
\end{align*} 
This formula makes it clear that each $\p_t^a\vc|_{t=0}$ is a function
of space-derivatives of  the initial data $u_0$, $\LE u_0$  and $\rho_0$.
We define the constant-in-time vectors $\mathbf{v}_a$   as\label{n.bfv}
\begin{align}
	\label{defn.va.bf} \mathbf{v}_a=\frac{\p^{a-1}}{\p t^{a-1}}\Pset{\kappa \Amu_\cdot^k(\rho_0\JT)\cp{k} -2\Amu_\cdot^k\pset{\rho_0\Jm^{-1}}\cp{k}}\big|_{t=0}, \ \ \ \text{for}\ a=1,2,3,4. 
\end{align}
We have that $\mathbf{v}_a\to \mathrm{v}_a$ as $\epsilon\to0$ where
$\mathrm{v}_a$ are defined in Section~\ref{sec:comp-cond-kappa}. We
use (\ref{mkp.bc}) and the definition \eqref{mkp:beta zero}  of
$\beta_{\epsilon}(t)$ to compute the following identities: for $a=0,1,2,3,$
\begin{align}
	\label{mkp: compat cond} \partial^a_t\bset{ \rho_0^2\Jm^{-2}
          -\kappa\rho_0^2\Jm^{-1}\JT}\big|_{t=0} =
        \p^a_t\bset{\beta_{\epsilon}(t)}\big|_{t=0} -         \p^a_t\bset{\sigma{\gm^{\alpha\beta}  {\ec
                          \cp{\alpha\beta}\cdot \nm} }
+
\kappa \gm^{\alpha\beta} \vc \cp{\alpha\beta}\cdot \nm}\big|_{t=0}.
\end{align}

\subsubsection{A priori estimates for the $\kappa\epsilon$-problem \eqref{mkp}}
\label{sec:priori-estim-kapp}

For $\epsilon>0$, we define the following higher-order energy function:
\begin{align}
\label{EM.l}
  \EM
=1+   \norm{\vc_{tt}(t)}_{0}^2 
 +
\int_{0}^{t}\abs{  \vc_{tt}\cdot \nc_{\epsilon}}_{1}^2
+
\sum_{a=0}^2\int_{0}^{t}\norm{\p^a_t\vc}_{5-2a}^2 . 
\end{align}
 
\begin{defn}
 	[Notational convention for constants depending on $1/
        \delta\kappa>0$] We let $\PM$  denote a generic polynomial with constant and coefficients depending on $1/\delta\kappa>0$.
	
	We define the constant $\MM>0$ by 
	\begin{align*}
		\MM=\PM\pset{\norm{u_0}_{100}, \norm{\rho_0}_{100}}. 
	\end{align*}
	
	We let $\RM\label{n:eRM}$ denote generic lower-order terms satisfying 
	\begin{align*}
		\int_0^T\RM\le \MM+\delta \Sup \EM+T\,\PM(\Sup \EM). 
	\end{align*}
\end{defn}

We assume       $T>0$ is taken sufficiently small to ensure that
\begin{align*}
   \frac{1}{2}\le \Jc  \le \frac{3}{2}\ \ \text{and}\ \
\frac{1}{2}\le \Jc_{\epsilon}  \le \frac{3}{2}\ \ \text{for all } t\in[0,T] \text{
    and } x\in {\Omega}.
\end{align*}

\begin{lem}[A priori estimates for the $\kappa\epsilon$-problem]
\label{lem.m.i}  We let $\vc$  solve the $\kappa\epsilon$-problem
\eqref{mkp} on a time-interval $[0,T]$ for some
$T=T_\kappa(\epsilon)>0$. Then
independent of $\epsilon$,
\begin{align}
  \label{mu.ind}
\Sup\EM \le\int_{0}^{T}\RM.
\end{align}
\end{lem}
  
We will establish Lemma~\ref{lem.m.i} in the following four steps:

\subsection*{Step~1: The  $\epsilon$-independent curl-estimates }
Taking the $\epsilon$-approximate Lagrangian curl of the equations (\ref{mkp.momentum}) yields 
\begin{align}
	\label{mkp.curl.vt} 
		\curl_{\ec} \vc_t=0. 
\end{align}
 Integrating the identity (\ref{mkp.curl.vt}) in time from $0$ to $t\in (0,T]$ provides that
\begin{align}
	\label{mkp.curl.v} \curl_{\ec}\vc=\curl u_0+\varepsilon_{\cdot ji}\int_0^t\p_t \bset{\Ac_{\epsilon}}{}_j^s \vc^i\cp{s}. 
\end{align} 
We may therefore infer the curl estimates from Lemma~\ref{lem: curl
  estimates}. We record this fact as

\begin{lem}[The $\epsilon$-independent curl-estimates for $\p^a \vc$, for $a=0,1,2$]
	\label{lem.curl.em} 
	\begin{align*}
			\sum_{a=0}^2\int_{0}^{T}\norm{\curl \p^a_t\vc }_{4-2a}^2 
			&\le \int_{0}^{T}\RM. 
	\end{align*}
\end{lem} 

\subsection*{Step~2: The  $\epsilon$-independent estimate for  $\vc_{tt}$}
We equivalently write the momentum equations \eqref{mkp.momentum} as
\begin{align}
\label{mkp.m.e} 
\rho_0  \vc^i_t+\bset{\ac_{\epsilon}}_i^k 
 \pset{\rho_0^2 \Jc_{\epsilon}^{-2}}\cp{k}-\kappa
\rho_0 \Jc_{\epsilon}^{-1}\bset{\ac_{\epsilon}}_i^k(\rho_0\JTc)\cp{k}&=0&&\text{in } \Omega\times(0,T_\kappa(\epsilon)].
\end{align}

\begin{lem}
  [Energy estimates for the action of $\p^{2}_t$ in
  \eqref{mkp.m.e}]\label{lem.mk.2}
  \begin{align*}
    \Sup \norm{\vc_{tt}(t)}_0^2 +
    \int_{0}^{T} \norm{\JTc{}_{tt}}_0^2
+
\int_{0}^{T}\abs{\hd \vc_{tt}\cdot \nc_{\epsilon}}_0^2\le \int_{0}^{T}\RM.
  \end{align*}
\end{lem}
\begin{proof}
  Testing the action of $\p^2_t$ in \eqref{mkp.m.e} against $\vc_{tt}$ in the
  $L^2(\Omega)$-inner product and integrating by parts with respect to
  $\p_k$ in the interior integrals  $\int_{\Omega} \bset{\ac_{\epsilon}}_i^k \p^2_t (\rho_0^2\Jc_{\epsilon}^{-2})\cp{k} \vc^i_{tt}$ and $- \kappa \int_{\Omega} \bset{\ac_{\epsilon}}_i^k \p^2_t\bbset{\rho_0 \Jc_{\epsilon}^{-1}  (\rho_0\JTc
)\cp{k}} \vc^i_{tt}$
yields
	\begin{align}
\label{t.vtm.4}
\begin{aligned}[b]
			\frac{1}{2}\frac{d}{dt}\int_\Omega \rho_0\abs{ \vc_{tt}}^2 &- \underbrace{\int_\Omega \p^2_t\bset{\rho_0^2\Jc_{\epsilon}^{-2} }\bset{\ac_{\epsilon}}_i^k\vc_{tt}^i\cp{k}}_{\RM} + \kappa \underbrace{\int_\Omega \p^2_t\bset{\rho_0^2\Jc_{\epsilon}^{-1}\JTc }\bset{\ac_{\epsilon}}_i^k\vc_{tt}^i\cp{k}}_{\I} \\
&			+ \underbrace{\int_\Gamma
                          \p^2_t\bset{\rho_0^2\Jc_{\epsilon}^{-2}
-                          \kappa \rho_0^2\Jc_{\epsilon}^{-1} \JTc}\bset{\ac_{\epsilon}}_i^k\vc_{tt}^iN^k}_{\II}=\RM.
\end{aligned}
	\end{align}
	It is  convenient to write the equation \eqref{t.vtm.4} as 
	\begin{align}
		\label{test:div.m} 
			\frac{1}{2} \frac{d}{dt} \int_\Omega
                        \rho_0\abs{ \vc_{tt}}^2 
                        + \kappa\int_\Omega \rho_0^2
                        \Jc_{\epsilon}^{-1}\abs{ \JTc{}_{tt}}^2 + \II  = \RM, 
	\end{align}
	where the error created in order to write  $\bset{\ac_{\epsilon}}_i^k \p^2_t \vc^i
        \cp{k}$  as $\JTc{}_{tt}$ in  $\I$ is of lower-order and so is absorbed in $\RM$. 

	 We write the boundary integral $\II$ appearing in the
         left-hand side of \eqref{test:div.m} as
         \begin{align}
           \label{t.vtm.4.i}
           \II = \int_{\Gamma} \Htc{}_{tt} \sqrt{\gc_{\epsilon}}\vc_{tt}\cdot
           \nc_{\epsilon},
         \end{align}
where we have used the boundary condition \eqref{mkp.bc} and the
formula \eqref{formula: outward normal}. The definition
\eqref{mkp.H} provides that $\Htc=\beta_{\epsilon}(t)-\sigma{\gc_{\epsilon}^{\alpha\beta}  {\ec
                          \cp{\alpha\beta}\cdot \nc_{\epsilon}} }
-
\kappa \gc_{\epsilon}^{\alpha\beta} \vc \cp{\alpha\beta}\cdot \nc_{\epsilon}$. Thus,
         \begin{align}
           \label{t.vtm.4.i.1}
                      \II =\kappa \int_{\Gamma} \sqrt{\gc_{\epsilon}}\gc_{\epsilon} ^{\alpha\beta}           \vc_{tt}\cp{\beta}\cdot \nc_{\epsilon}
           \vc_{tt}\cp{\alpha}\cdot \nc_{\epsilon}
+ 
\kappa \underbrace{
\int_{\Gamma}           \vc_{tt}^i\cp{\beta} \bpset{ \nc_{\epsilon}^i
  \sqrt{\gc_{\epsilon}}\gc_{\epsilon} ^{\alpha\beta} \nc_{\epsilon}^j}\cp{\alpha}
           \vc_{tt}^j }_{\j}
+ \RM.
         \end{align}
Thanks to Lemma~\ref{lem.j},
\begin{align}
\label{52.tr}
  \int_{0}^{T} \abs{\vc_{tt}}_0^2\le  C \int_{0}^{T} \abs{\vc_{tt}}_{0.25}^2\le \int_{0}^{T}\RM.
\end{align}
Employing the Cauchy-Schwarz inequality and Young's
inequality with $\delta>0$ thus yields
\begin{align*}
  \int_{0}^{T}\j -
\underbrace{ \int_{0}^{T} \vc_{tt}\cp{\beta}\cdot \nc_{\epsilon}\cp{\alpha}
  \sqrt{\gc_{\epsilon}} \gc_{\epsilon} ^{\alpha\beta} \vc_{tt}\cdot \nc_{\epsilon} }_{\j'}\le \int_{0}^{T}\RM.
\end{align*}
Integrating by parts with respect to $\hd_ \beta$, we similarly have that
$\int_{0}^{T}\j'\le \int_{0}^{T}\RM$.  We have therefore established that the
identity \eqref{t.vtm.4.i.1}  is equivalently written as
\begin{align*}
  \II = \kappa \int_{\Gamma} \sqrt{\gc_{\epsilon}}\gc_{\epsilon} ^{\alpha\beta}           \vc_{tt}\cp{\beta}\cdot \nc_{\epsilon}
           \vc_{tt}\cp{\alpha}\cdot \nc_{\epsilon}
+ \RM.
\end{align*}
Using this identity for $\II$ in the time-integral of
\eqref{test:div.m} completes the proof.
\end{proof}
\begin{prop}
  [The $\epsilon$-independent estimates for $\vc_{tt}$]\label{prop.m.vt.2}
  \begin{align*}
    \Sup \norm{\vc_{tt}(t)}_0^2 
+
    \int_{0}^{T}\abs{\vc_{tt}\cdot \nc_{\epsilon}}_1^2
 + \int_{0}^{T}\norm{\vc_{tt}}_1^2 \le \RM.
  \end{align*}
  \begin{proof}
    The desired  $L^\infty(0,T;L^2(\Omega))$-estimate is provided by
    Lemma~\ref{lem.mk.2}.
The inequality \eqref{52.tr} and the trace-estimate stated in
Lemma~\ref{lem.mk.2} establish the desired
$L^2(0,T;H^1(\Gamma))$-estimate. 
 We infer from
    the arguments proving Proposition \ref{prop: estimate for v_ttt
          and sqrt kp v_tttt} that Lemma~\ref{lem.mk.2} implies a
        divergence- and normal trace-estimate for $\vc_{tt}$.   Thanks
        to the curl-estimate for $\vc_{tt}$ stated in
        Lemma~\ref{lem.curl.em}, we have by
        Proposition~\ref{prop:Hodge} that the proof is complete.
  \end{proof}
\end{prop}

\subsection*{Step~3: The $\epsilon$-independent estimate for $\vc_t$}

\begin{prop}
  [The $\epsilon$-independent estimate for $\vc_t$]\label{prop.vtm.1}
  \begin{align*}
    \int_{0}^{T}\norm{\vc_t}_3^2\le \int_{0}^{T}\RM.
  \end{align*}
\end{prop}
\begin{proof}
We write the boundary condition \eqref{mkp.bc} as
\begin{align*}
\kappa \gc_{\epsilon}^{\alpha\beta} \vc \cp{\alpha\beta}\cdot \nc_{\epsilon}
=
\underbrace{\kappa\rho_0^2 \Jc_{\epsilon}^{-1}\JTc
-
\rho_0^2 \Jc_{\epsilon}^{-2}+
  \beta_{\epsilon}(t)-\sigma{\gc_{\epsilon}^{\alpha\beta}  {\ec
                          \cp{\alpha\beta}\cdot \nc_{\epsilon}} }
}_{\jtt}.
\end{align*}
Taking a time-derivative of this equation yields
  \begin{align}
\label{mkp.bc.t}
\kappa \gc_{\epsilon}^{\alpha\beta} \vc_t \cp{\alpha\beta}\cdot \nc_{\epsilon}
=\jtt_t+ \kappa\vc \cp{\alpha\beta}\cdot \p_t(\nc_{\epsilon} \gc_{\epsilon} ^{\alpha\beta}).
\end{align}
  The right-hand side of \eqref{mkp.bc.t} scales like $T \,\hd^2 \vc_t 
  $. We infer that
  \begin{align}
\label{tr.vtm.t}
    \int_{0}^{T}\abs{\hd^2\vc_t\cdot\nc_{\epsilon}}_{0.5}^2\le \int_{0}^{T}\RM.
  \end{align}

We equivalently write the momentum  equations \eqref{mkp.m.e}  as 
  \begin{align*} 
    \bset{\Ac_{\epsilon}}_i^k\JTc\cp{k}=
\underbrace{
( \kappa    \rho_0)^{-1}\bbset{\vc^i_t+2\bset{\Ac_{\epsilon}}_i^k(\rho_0\Jc_{\epsilon}^{-1})\cp{k} 
-\kappa\JTc\bset{\Ac_{\epsilon}}_i^k \rho_0\cp{k}}
}_{\htt}.
  \end{align*}
Using the fundamental theorem of calculus, we have that
\begin{align}
  \label{htt.A.B}
  D\JTc= \htt + \JTc \cp{k} \int_{0}^{t}\p_t \bset{\Ac_{\epsilon}}_i^k.
\end{align}
Since the right-hand side of \eqref{htt.A.B} scales
  like $\vc_{t}+ D\vc$, a time-derivative of \eqref{htt.A.B}  yields
\begin{align*}
  \int_{0}^{T}\norm{D\JTc{}_t}_1^2\le \int_{0}^{T}\RM.
\end{align*}
Lemma~\ref{lem.mk.2} provides a good estimate for 
$\JTc{}_t$ in $L^2(0,T;L^2(\Omega))$. Thus,
\begin{align}
\label{div.vtm.t}
  \int_{0}^{T}\norm{\JTc{}_t}_2^2\le \int_{0}^{T}\RM.
\end{align} 
 Via Proposition~\ref{prop:Hodge}, the curl-, normal trace- and divergence-estimates for $\vc_t$,
respectively given in Lemma~\ref{lem.curl.em} and inequalities 
\eqref{tr.vtm.t} and \eqref{div.vtm.t}, 
complete the proof.
\end{proof}
\subsection*{Step~4: Concluding the proof of Lemma~\ref{lem.m.i}}
Repeating Step~3 establishes
\begin{prop}
  [The $\epsilon$-independent estimate for $\vc$]\label{prop.vtm.0}
  \begin{align*}
    \int_{0}^{T}\norm{\vc}_5^2\le \int_{0}^{T}\RM.
  \end{align*}
\end{prop}
 The sum of the $\epsilon$-independent estimates stated in Propositions~\ref{prop.m.vt.2},
 \ref{prop.vtm.1} and \ref{prop.vtm.0}  completes the
proof of Lemma~\ref{lem.m.i}.

\

As an intermediate step in proving Theorem~\ref{thm.akp}, we will
establish the following
\begin{thm}
	[Solutions to the $\kappa\epsilon$-problem] \label{thm.mkp} For
        $C^\infty$-class initial data $(\rho_0,u_0, \Omega )$  satisfying  the conditions \eqref{l.boundedness.r0} and
\eqref{s.compatibility conditions}, and for some
        $T=T_\kappa(\epsilon)>0$, there
        exists a  unique
        solution $\vc$ to the $\kappa\epsilon$-problem \eqref{mkp} verifying
$(\vc,\vc_t,\dots,\vc_{tttt} )|_{t=0} =(u_0, \mathbf{v}_1,\dots,\mathbf{v}_4)$  and
      \begin{align}
        \label{mu.cont.est.v}
 \Sup \norm{\vc_{tttt}(t)}_{0}^2
 +\sum_{a=0}^4\int_{0}^{T}\norm{\p^a_t\vc}_{9-2a}^2
 +\sum_{a=0}^4\int_{0}^{T}\abs{  \p^a_t \vc\cdot \nc_{\epsilon}}_{9-2a}^2<\infty. 
      \end{align}
\end{thm}

\begin{rem}
  We recall that the initial data $\mathbf{v}_a$, $a=1,2,3,4$, is
  defined in Section~\ref{sec:comp-cond-mu}.
\end{rem}

\subsection{Deriving a heat-type problem with Neumann-type boundary conditions}
\label{sec:deriving-heat-type}
We will  derive  a nonlinear heat-type problem with
Neumann-type boundary conditions which is  equivalent to the
$\kappa\epsilon$-problem \eqref{mkp}.
\subsubsection{Rewriting the
  momentum equations
  \eqref{mkp.momentum} via  \eqref{mkp.curl.v}} 
\label{sec:rew1}
Setting
 \begin{align}
   \label{varrho}
   \check{\varrho}=\kappa \rho_0\Jm,
 \end{align}
and using the definition \eqref{mkp.J}, the momentum equations \eqref{mkp.momentum} are equivalently written as
 \begin{align}
     \label{mkp.m.1}
   \vc_t - \check{\varrho}\Amu_\cdot^k (\Div_{\ec} \vc) \cp{k}=
\Div_{\ec} \vc \Amu_\cdot^k\check{\varrho}\cp{k}-2 \Amu_\cdot^k(\rho_0\Jm
   ^{-1})\cp{k}.
 \end{align}
Given a sufficiently smooth vector $\vc$, the
identity $-\Delta \vc = \curl\curl \vc - D \Div \vc$ in the $\epsilon$-approximate Lagrangian
variables is the identity 
\begin{align} 
  -\Amu_r^j[\Amu_r^k \vc \cp{k}]\cp{j}= \curl_{\ec}\curl_{\ec} \vc - \Amu_\cdot ^s
(  \Div_{\ec} \vc)\cp{s}.
\label{HH.Lagrangian}
\end{align} 
Using \eqref{HH.Lagrangian}, the equations \eqref{mkp.m.1} are
equivalently  written as
\begin{align}
  \label{mkp.m.2}
  \begin{aligned}[b]    \vc_t&-\check{\varrho} \Amu_r^j\bpset{\Amu_r^k \vc\cp{k}}\cp{j}= \check{\varrho}    \curl_{\ec} (\curl_{\ec} \vc )+ 
\Div_{\ec} \vc \Amu_\cdot^k\check{\varrho}\cp{k}-2 \Amu_\cdot^k(\rho_0\Jm   ^{-1})\cp{k}.
  \end{aligned}
\end{align}
Thanks to the vorticity equation \eqref{mkp.curl.v}, we further have
that \eqref{mkp.momentum} is equivalent to
\begin{align}
  \label{mkp.m.212}
  \begin{aligned}[b]
    \vc_t& - \check{\varrho} \Amu_r^j\bpset{\Amu_r^k \vc \cp{k}}\cp{j}= \Kc,
  \end{aligned}
\end{align} 
where the vector field $\Kc$ appearing in the right-hand side
of \eqref{mkp.m.212} is defined as
\begin{align}
  \label{mkp.m.21}
  \Kc=
\check{\varrho} 
    \curl_{\ec}{\pset{\curl u_0+\varepsilon_{\cdot ji}\int_0^t\p_t \Amu_j^s \vc^i\cp{s}} }
+ \Div_{\ec} \vc \Amu_\cdot^k\check{\varrho}\cp{k}-2 \Amu_\cdot^k(\rho_0\Jm   ^{-1})\cp{k}.
\end{align}
\subsubsection{Deriving a Neumann-type boundary condition for $\vc$}
\label{sec:rew2} 

We decompose any vector field $\xi\in \mathbb{R}^3$ evaluated on
$\Gamma$ into tangential components   and a normal
component  as
\begin{align}
\label{v.decomp} 
  \xi=\xi^\alpha \gm^{\alpha\beta}\ec \cp{\beta}+ \xi^3 \nm,
\end{align}
where $\xi^\alpha=\xi\cdot \ec \cp{\alpha}$ and $\xi^3=\xi\cdot \nm$.
In the special case of a flat  boundary, we
may use the standard orthogonal unit vectors $e_k$ to write
\eqref{v.decomp} as
\begin{align}
\label{v.decomp.simple}
  \xi= \xi^ \alpha e_\alpha+ \xi^3e_3.
\end{align}
For the special case where \eqref{v.decomp} takes the form
\eqref{v.decomp.simple}, we have that
\begin{align}
\label{DN.simple}
\frac{\p \vc^i}{\p N}= [(\curl \vc)\times N]^i
+\vc^3 \cp{\alpha} e^i_\alpha + \bset{\Div \vc  -\vc^ \alpha\cp{\alpha} }N^i.
\end{align}

 We define the vector $\err$ 
 as the sum
\begin{align}
\label{err}
  \err=\err_{\curl}+\err_{\Div}, 
\end{align}
where $\err_{\curl}$ is a tangent vector and  $\err_{\Div}$ is a
normal vector, respectively defined as
\begin{subequations}
  \begin{align}
    \label{err.curl}
    \err_{\curl} &= \frac{\sqrt{\gm}}{\Jm} \bbset{ \gm^{\alpha\beta}\vc^3
      \cp{\beta} + \vc^\gamma \gm^{\gamma\delta}\gm^{\alpha\beta}\ec
      \cp{\delta\beta} \cdot \nm}\ec\cp{\alpha},
\\
    \label{err.div}
    \err_{\Div}&= - \frac{\sqrt{\gm}}{\Jm} \bbset{ \gm^{\alpha\beta} \vc^
      \alpha\cp{\beta} + \vc ^{\alpha}
      (\sqrt{\gm}\gm^{\alpha\beta})\cp{\beta} + \vc^3H(\ec)}\nm.
  \end{align}
\end{subequations}
For the moving boundary $\Gamma_{\epsilon}(t)=\ec(\Gamma)$, the identity \eqref{DN.simple} is
written as
\begin{align}
  \label{DN.1}
\frac{\Jm}{\sqrt{\gm}} N ^j \Amu_r^j \Amu_r^k \vc ^i\cp{k}&= \bpset{(\curl_{\ec} \vc
 ) \times \nm}^i + \nm^i \Div_{\ec} \vc +
\frac{\Jm}{\sqrt{\gm} }   \err^i,
\end{align}
or equivalently,
\begin{align}
  \label{DN.2}
N ^j \Amu_r^j \Amu_r^k \vc \cp{k}&=
\sqrt{\gm} \Jm^{-1} {(\curl_{\ec}
 \vc)\times \nm}
+\sqrt{\gm} \Jm^{-1}  (\Div_{\ec} \vc )\nm+   \err.
\end{align}
Thanks to \eqref{mkp.curl.v}, we have that \eqref{DN.2} multiplied by
$\check{\varrho} =\kappa\rho_0\Jm$ is equivalent
to
\begin{align}
\label{DN.33}
\check{\varrho} N ^j \Amu_r^j \Amu_r^k \vc \cp{k}
=\hc_{\curl}
+\frac{\sqrt{\gm}}{\rho_0}
\kappa\rho_0^2( \Div_{\ec} \vc )\nm + \check{\varrho}\err,
\end{align}
where the tangential vector field $\hc_{\curl}$ appearing in the right-hand
side of \eqref{DN.33} is defined  via the $\epsilon$-approximate Lagrangian vorticity
equation \eqref{mkp.curl.v}  as
\begin{align}
  \label{DN.3}
  \hc_{\curl}= 
\check{\varrho} \sqrt{\gm}\Jm^{-1} {\bpset{\curl
  u_0+\varepsilon_{\cdot ji}\int_0^t\p_t \Amu_j^s \vc^i\cp{s}}\times \nm}.
\end{align}
We write the boundary condition
\eqref{mkp.bc} as 
\begin{align}
  \label{mkp.div.bc}
\kappa \rho_0^2  \Div _{\ec}\vc =   \rho_0^2 \Jm ^{-2} -
\beta_{\epsilon}(t) +\sigma { \gm
                  ^{\alpha\beta}  {\ec\cp{\alpha\beta}\cdot
                    \nm}} 
+
\kappa \gm ^{\alpha\beta} \vc\cp{\alpha\beta}\cdot \nm. 
\end{align}
Setting
\begin{align}
  \label{DN.gfk}
\gfk ^{\alpha\beta}&= \kappa \frac{\sqrt{\gm}
}{\rho_0}\gm ^{\alpha\beta},
\\
\label{DN.hn}
  \hc_{\Div}&=
 \frac{\sqrt{\gm}
}{\rho_0}\Bbset{ \rho_0^2 \Jm ^{-2} -
\beta_{\epsilon}(t) +\sigma { \gm
                  ^{\alpha\beta}  {\ec\cp{\alpha\beta}\cdot
                    \nm}} 
}\nm,
\end{align} 
we find by using the identity \eqref{mkp.div.bc} for
$\kappa\rho_0^2              \Div _{\ec} \vc $ in the equation \eqref{DN.33} that
\begin{align}
\label{DN.4}
\check{\varrho} N ^j \Amu_r^j \Amu_r^k \vc \cp{k}=\hc_{\curl}+\hc_{\Div}+
\gfk^{\alpha\beta}\bset{\vc\cp{\alpha\beta}\cdot \nm}\nm+ \rerr.
\end{align}
We  define  the vector field $\hc$ as the sum
\begin{align}
\label{DN.h}
  \hc=\hc_{\curl}
+ \hc_{\Div} +\gfk ^{\alpha\beta} \bset{\vc\cp{\alpha\beta}\cdot \nm}\nm+\check{\varrho}[\err_{\curl}+\err_{\Div}],
\end{align}
with the vectors
 $\hc_{\curl}$, $\hc_{\Div}$, $\err_{\curl}$ and $\err_{\Div}$  
 respectively  given in \eqref{DN.3}, \eqref{DN.hn}, \eqref{err.curl}
 and 
\eqref{err.div}, and the function
$\check{\varrho}$ defined in \eqref{varrho}. We equivalently express the identity \eqref{DN.4} as
\begin{align}
  \label{DN.5}
\check{\varrho}  N ^j \Amu_r^j \Amu_r^k \vc \cp{k}=
\hc.
\end{align}

\begin{rem} We record that the identities \eqref{mkp.div.bc} and
  \eqref{DN.hn} establish the following identity:
  \begin{align}
        \label{hp.bc.n.1}
    \check{\varrho} \sqrt{\gm} \Jm^{-1} \Div _{\ec} \vc&=\hc_{\Div} \cdot \nm+{\gfk  ^{\alpha\beta}{\vc\cp{\alpha\beta}\cdot \nm}}.
  \end{align}
\end{rem}
\subsubsection{The heat-type $\kappa\epsilon$-problem  with Neumann-type
  boundary conditions}
\label{sec:an-equiv-stat}

   \begin{defn}[The heat-type $\kappa\epsilon$-problem] For
    $\kappa>0$  and $\epsilon>0$ given, we define  $\vc$ as the solution of  the nonlinear heat-type system
	\begin{subequations}
		\label{hp} 
		\begin{alignat}
			{4} \label{hp.momentum}     
      \vc_t -  \re \Amu_r^j\bpset{\Amu_r^k \vc
      \cp{k}}\cp{j}&=\Kc
&\ &\text{in }
      \Omega\times(0,T_\kappa(\epsilon)],\\
\label{hp.bc}
\begin{aligned}[b]
\re N^j\Amu_r^j
\Amu_r^s \vc \cp{s} 
\end{aligned}
&=
\hc+c(t)
&\ \ \ &\text{on } \Gamma\times(0,T_\kappa(\epsilon )],\\
			\label{hp.ic} \pset{\ec,\vc}|_{t=0}&=\pset{e,u_0}&&on\ \Omega.
		\end{alignat}
	\end{subequations}
The bounded, nonnegative function $\check{\varrho}$ is defined in
\eqref{varrho}. The vector fields
$\Kc$ and $\hc$ are respectively defined in \eqref{mkp.m.21} and
\eqref{DN.h}.
 The vector field $c(t)$ appearing in the right-hand side of
\eqref{hp.bc} is defined as
\begin{align}
  \label{hp.c}
  \begin{aligned}[b]
  c(t)=
  \sum_{a=0}^2\frac{t^a}{a!} &\p^a_t\bbset{  \re N^j\Amu_r^j
\Amu_r^s \vc \cp{s} -\hc}|_{t=0}.
\end{aligned}
\end{align}
  \end{defn}

\begin{rem}
\label{rem.e.va} The vector fields $\mathbf{v}_a$, for $a=1,2,3,4,$  
defined in Section~\ref{sec:comp-cond-mu},  are
  equivalently defined using   the $(a-1)$th time-derivative of the  momentum equations \eqref{hp.momentum} evaluated at time $t=0$.
\end{rem}

\begin{rem}
  According to Appendix~\ref{sec:appendix:equiv}, the heat-type
  $\kappa\epsilon$-problem \eqref{hp} is equivalent to the $\kappa\epsilon$-problem \eqref{mkp}.
\end{rem}
\begin{prop}
	[Solutions to the
heat-type $\kappa\epsilon$-problem \eqref{hp}] \label{prop.hp}
For        $C^\infty$-class initial data $(\rho_0, u_0,  \Omega )$
satisfying  the conditions  \eqref{l.boundedness.r0} and
\eqref{s.compatibility conditions}, and for some
        $T=T_\kappa(\epsilon)>0$, there exists a unique $\vc\in
        L^2(0,T;H^9(\Omega))$ with $\p^a_t \vc\in
L^2(0,T;H^{9-2a}(\Omega))$, for $a=1,2,3,4$, and $\vc_{tttt}\in
L^\infty(0,T;L^2(\Omega))$ that solves the nonlinear heat-type $\kappa\epsilon$-problem  \eqref{hp} on a time-interval
        $[0,T]$ and verifies $(\vc,\vc_t,\dots,\vc_{tttt} )|_{t=0} =(u_0, \mathbf{v}_1,\dots,\mathbf{v}_4)$. 
\end{prop}

\subsection{The $\mu$-problem and its a
  priori estimates}
\label{sec:defining-heat-type}
We now define our second intermediate problem, which we term
the $\mu$-problem. The $\mu$-problem is a system of nonlinear
heat-type equations that is asymptotically consistent with the
heat-type $\kappa\epsilon$-problem \eqref{hp}.
To indicate  the dependence on the smoothing paramater
$\mu$  of all the variables in the following problem,
we place the symbol $\mathring{\ }$ over each of the variables.   

 Given a sufficiently smooth vector field $\vr$, we set
$\mathring{\eta}= e+\int_{0}^{t}\vr$ in $\Omega$ and
$\mathring{\eta}_{\epsilon}  =e+\int_0^t{\vr}_{\epsilon} $ on
$\Gamma$. We define
$\er$   to be the solution of the following time-dependent
elliptic Dirichlet problem:
\begin{subequations}
\label{elliptic.er}
  \begin{alignat}{2}
    \label{er.m}
  \Delta\er&=\Delta\mathring{\eta}&& \text{in } \Omega,
\\
\er&=\mathring{\eta}_{\epsilon}&\ \ \ \ &\text{on }\Gamma.
  \end{alignat}
\end{subequations}
We define the following  $\epsilon$-approximate Lagrangian variables:
\begin{align*} 
	\Ar  =\bset{D\er  }^{-1},\ \ \Jr  =\det
        D\er  ,\ \ \ar =\Jr \Ar  ,\ \bset{\gr}_{\alpha\beta}=\er \cp{\alpha}\cdot \er \cp{\beta}, \ \text{and}\  \sqrt{\gr } \nr  =\bset{\ar }^TN.
\end{align*}

We recall that the convolution operator $\LM$ for $\mu>0$ is defined
via Section~\ref{sec:horiz-conv-layers}.

  \begin{defn}[The $\mu$-problem] Given
    $\kappa>0$ and $\epsilon>0$, for   $\mu>0$ we define  $\vr$ as the solution of  the nonlinear heat-type system
	\begin{subequations}
		\label{shp} 
		\begin{alignat}
			{4} \label{shp.m}     
      \vr_t -  \rro \Aru_r^j\bpset{\Aru_r^k \vr
      \cp{k}}\cp{j}&=\Kr
&\ &\text{in }
      \Omega\times(0,T_\kappa(\epsilon\mu)],\\
\label{shp.bc}
\begin{aligned}[b]
\rro N^j\Aru_r^j
\Aru_r^s \vr \cp{s} 
\end{aligned}
&=
\hrm+\crr
&\ \ \ &\text{on } \Gamma\times(0,T_\kappa(\epsilon\mu)],\\
			\label{shp.ic} \pset{\er,\vr}|_{t=0}&=\pset{e,u_0}&&on\ \Omega.
		\end{alignat}
	\end{subequations}
The bounded, nonnegative function $\rro$ is defined as
\begin{align}
  \label{shp.rro}
  \rro=\kappa \rho_0\Jr.
\end{align}
 The vector field
$\Kr$ appearing in the right-hand side of \eqref{shp.m} is defined as
\begin{align}
  \label{shp.K}
  \Kr=
\rro 
    \curl_{\er}{\pset{\curl u_0+\varepsilon_{\cdot ji}\int_0^t\p_t \Aru_j^s \vr^i\cp{s}} }
+ \Div_{\er} \vr \Aru_\cdot^k\rro\cp{k}-2 \Aru_\cdot^k(\rho_0\Jr   ^{-1})\cp{k}.
\end{align}
  The vector field $\hrm$  appearing in the right-hand side of
  \eqref{shp.bc} is given by   
  \begin{align}
    \label{shp.h}
    \begin{aligned}[b]
     \hrm=  \hr_{\curl}
+ \hr_{\Div}
 +\sum_{\local=1}^K \sqrt{\xl} \Bpset{\LM\Bbset{\gfrk
    ^{\alpha\beta}\LM\bbset{\pset{\sqrt{\xl}\vr\cp{\alpha\beta}\cdot \nr}\circ
  \thetal}}}\circ\thetal^{-1}\nr
+\rro\Bbset{\br^{\mu}_{\curl}+\br^{\mu}_{\Div}},
   \end{aligned} 
 \end{align}
 where $\gfrk ^{\alpha\beta}$ appearing in the right-hand side of
 \eqref{shp.h}  is defined as 
 \begin{align}
   \label{shp.g}
   \gfrk ^{\alpha\beta}= \kappa \frac{\sqrt{\gr}
}{\rho_0}\gr^{\alpha\beta} \circ\thetal,
 \end{align}
   and  $\hr_{\curl}$, $\hr_{\Div}$, $\br^{\mu}_{\curl}$, $\br^{\mu}_{\Div}$appearing in the right-hand side of \eqref{shp.h} are defined as 
 \begin{subequations}
\label{he.summands}
   \begin{align}
     \hr_{\curl}&= \rro \sqrt{\gr}\Jr^{-1} {\bpset{\curl
         u_0+\varepsilon_{\cdot ji}\int_0^t\p_t \Aru_j^s
         \vr^i\cp{s}}\times \nr},
     \\
     \hr_{\Div}&= \frac{\sqrt{\gr}
}{\rho_0}
\Bbset{ \rho_0^2
\Jr ^{-2} - \beta_{\epsilon}(t) +\sigma { \gr
         ^{\alpha\beta}  {\er\cp{\alpha\beta}\cdot \nr}}
     }\nr,
     \\
\label{he.summands.bc}
     \br^{\mu}_{\curl} &= \frac{\sqrt{\gr}}{\Jr} \Bbset{
\sum_{\local=1}^K      \gr^{\alpha\beta} \sqrt{\xl}\bpset{\LM\bbset{(
  \sqrt{\xl}\vr\cdot\nr) \cp{\beta}\circ\thetal}}\circ\thetal ^{-1} + \vr^\gamma
       \gr^{\gamma\delta}\gr^{\alpha\beta}\er \cp{\delta\beta} \cdot
       \nr}\er\cp{\alpha},\\
\label{he.summands.bd}
     \br^{\mu}_{\Div}&= - \sum_{\local=1}^K\sqrt{\xl}\LM\Bpset{\sqrt{\xl}\frac{\sqrt{\gr}}{\Jr} \bbset{
       \gr^{\alpha\beta} \vr^ \alpha\cp{\beta} + \vr ^{\alpha}
       (\sqrt{\gr}\gr^{\alpha\beta})\cp{\beta} + \vr^3H(\er)} \circ
     \thetal} \circ \thetal ^{-1}\nr.
   \end{align}
 \end{subequations}
 The vector field $\crr$ appearing in the right-hand side of
\eqref{shp.bc} is defined as
\begin{align}
  \label{shp.c}
  \begin{aligned}[b]
  \crr=
  \sum_{a=0}^2\frac{t^a}{a!} &\p^a_t\bbset{  \rro N^j\Aru_r^j
\Aru_r^s \vr \cp{s} -\hrm}|_{t=0}.
\end{aligned}
\end{align}
  \end{defn}

\begin{rem}
The fixed-point solution to the $\mu$-problem \eqref{shp} is
established in Appendix~\ref{sec:appendix:fp}.
\end{rem}
\begin{rem}\label{rem.h}
For any $h$ defined on $\Gamma$, we have that $h=\sum_{\local=1}^K \xl
h=\sum_{\local=1}^K \sqrt{\xl} \bpset{[\sqrt{\xl}
h\circ\thetal ] \circ\thetal ^{-1}}$. 
\end{rem}

For $\mu>0$, we define the following higher-order energy function:
\begin{align}
\label{def.EE}
  \EE=1+\norm{ \vr_{tt}(t)}_{0}^2
  +\sum_{a=0}^2\int_{0}^{t}\norm{\p^a_t\vr}_{5-2a}^2
& +\sum_{a=0}^2\sum_{\local=1}^K\int_{0}^{t}\abs{  \LM\bset{(\sqrt{\xl}\p^a_t \vr\cdot \nr)\circ\thetal }  }_{5-2a,\DL}^2.
\end{align}

\begin{defn}
 	[Notational convention for constants depending on $1/ \delta\kappa\epsilon>0$] We let $\PE\label{n:eP}$ denote a generic polynomial with constant and coefficients depending on $1/\delta\kappa\epsilon>0$.
	
	We define the constant $\ME>0$ by 
	\begin{align}
		\ME=\PE\pset{\norm{u_0}_{100}, \norm{\rho_0}_{100}}. 
	\end{align}
	
	We let $\RE\label{n:eR}$ denote generic lower-order terms satisfying 
	\begin{align*}
		\int_0^T\RE\le \ME+\delta \Sup \EE+T\,\PE(\Sup \EE). 
	\end{align*}
\end{defn}

\begin{lem}[A priori estimates for the   $\mu$-problem]
\label{lem.e.i} We  let $\vr$
solve the   $\mu$-problem 
\eqref{shp} on a time-interval $[0,T]$ for some
$T=T_\kappa(\epsilon\mu)>0$. Then
independent of $\mu$,
\begin{align}
  \label{ep.ind}
\Sup\EE \le\int_{0}^{T}\RE.
\end{align}
\end{lem}
We will establish Lemma~\ref{lem.e.i} in the following four steps:

\subsection*{Step~1: The  $\mu$-independent estimates for $\vr_{tt}$}

Testing two time-derivatives of \eqref{shp.m} against $\vr_{tt}$
in the $L^2(\Omega)$-inner product and integrating by parts in the
interior integral  $-\int_{\Omega}\p^2_t\bpset{\rro
  \bset{\Ar}_r^j\bpset{\bset{\Ar}_r^k \vr\cp{k}}\cp{j}}\vr_{tt}$  yields
\begin{align}
  \label{tv.tt}
  \begin{aligned}[b]
    \frac{1}{2}\frac{d}{dt}\int_{\Omega} \abs{\vr_{tt}}^2+
    \int_{\Omega}\rro&\abs{\Aru_\cdot^k \vr_{tt}\cp{k}}^2
-
\underbrace{
 \int_{\Gamma} 
  [\hrm+\crr]_{tt}\vr_{tt}
}_{\II}
\\
&= 
\underbrace{
\int_{\Omega} \p^2_t\bpset{\Kr
- 
( \rro  \bset{\mathring{A}_{\epsilon}}_r^j) \cp{j} \bset{\mathring{A}_{\epsilon}}_r^k\mathring{v}
      \cp{k} }\mathring{v}_{tt}
}_{\RE}+\RE.
  \end{aligned}
\end{align}
The definition \eqref{shp.h} of the vector field $\hrm$ provides that 
\begin{align*}
  \II= \underbrace{- \sum_{\local=1}^K \int_{\DL}
\LM\bbset{\pset{\sqrt{\xl}\vr_{tt}\cp{\alpha\beta}\cdot
        \nr}\circ\thetal}\gfrk ^{\alpha\beta} \LM \bbset{(\sqrt{\xl}\vr_{tt}
  \cdot \nr) \circ\thetal} }_{\II'}-
\underbrace{
  \int_{\Gamma}(\rro[\br^{\mu}_{\curl}+\br^{\mu}_{\Div}])_{tt}
  \vr_{tt}
}_{\j}+ \RE.
\end{align*}
We integrate by parts with respect to $\hd_\alpha$ in $\II'$ to find
that
\begin{align*}
  \II'=&
\sum_{\local=1}^K \int_{\DL}
\LM\bbset{\pset{\sqrt{\xl}\vr_{tt}\cp{\beta}\cdot
        \nr}\circ\thetal}\gfrk ^{\alpha\beta} \LM
    \bbset{(\sqrt{\xl}\vr_{tt} \cp{\alpha}
  \cdot \nr) \circ\thetal} 
\\
&+
\underbrace{
\sum_{\local=1}^K \int_{\DL}
\LM\bbset{(\vr_{tt}\cp{\beta}\cdot (\nr \sqrt{\xl})\cp{\alpha})\circ\thetal}\gfrk ^{\alpha\beta}\LM\bbset{(\sqrt{\xl}\vr_{tt}
  \cdot \nr)\circ\thetal}}_{\II'_A}
\\
&+
\underbrace{
\sum_{\local=1}^K \int_{\DL}\LM\bbset{(\sqrt{\xl}\vr_{tt}\cp{\beta}\cdot \nr)\circ\thetal}\hd_\alpha\gfrk ^{\alpha\beta}\LM\bbset{(\sqrt{\xl}\vr_{tt}
  \cdot \nr)\circ\thetal}
}_{\II'_B}
\\
&+
\underbrace{
\sum_{\local=1}^K \int_{\DL}\LM\bbset{(\sqrt{\xl}\vr_{tt}\cp{\beta}\cdot \nr)\circ\thetal}\gfrk ^{\alpha\beta}\LM\bbset{(\vr_{tt}
  \cdot (\nr \sqrt{\xl})\cp{\alpha})\circ\thetal}}_{\II'_C}.
\end{align*}
We employ an $H^{-0.5}(\DL)$-duality pairing in the integrals
$\II'_A$, $\II'_B$ and $\II'_C$. For example, 
\begin{align*}
  \II'_A&\le C \abs{\hd \vr_{tt} \circ\thetal}_{-0.5,\DL}\abs{\LM[\pset{\sqrt{\xl}\vr_{tt}\cdot
      \nr}\circ\thetal]}_{0.5,\DL}
\\
&\le \delta \norm{\vr_{tt}}_1^2+ C_\delta \abs{\LM[\pset{\sqrt{\xl}\vr_{tt}\cdot
      \nr}\circ\thetal]}_{0.5,\DL}^2.
\end{align*}
Thanks to Lemma~\ref{lem.j}, $\abs{\LM[{ (\sqrt{\xl}\vr_{tt}\cdot  \nr)\circ\thetal}]}_{0.5,\DL}^2\le C \abs{{\vr_{tt}}}_{0}\abs{\LM[\pset{\sqrt{\xl}\vr_{tt}\cdot
      \nr}\circ\thetal]}_{1,\DL}\le \RE + \delta \abs{\LM\bpset{ \vr_{tt}\cdot
      \nr}}_{1,\DL}^2$. Thus, we infer that
  \begin{align*}
    \II= \sum_{\local=1}^K \int_{\DL}
\LM\bbset{\pset{\sqrt{\xl}\vr_{tt}\cdot
        \nr}\cp{\beta}\circ\thetal}\gfrk ^{\alpha\beta} \LM
    \bbset{(\sqrt{\xl}\vr_{tt}
  \cdot \nr)\cp{\alpha} \circ\thetal} 
+\RE,
  \end{align*}
where we have used the definitions \eqref{he.summands.bc} and
\eqref{he.summands.bd} in analyzing $\j$. We infer from the  time
integral of \eqref{tv.tt} that
\begin{align}
\label{ep.vtt}
\Sup\norm{\vr_{tt}(t)}_{0}^2+  \int_{0}^{T}\norm{\vr_{tt}}_1^2 +
\sum_{\local=1}^K\int_{0}^{t}\abs{  \LM\bset{(  \sqrt{\xl}\vr_{tt}\cdot \nr)\circ\thetal }  }_{1,\DL}^2\le \int_{0}^{T}\RE.
\end{align}
\subsection*{Step~2: The  $\mu$-independent  estimates for $\vr_{t}$}
Similar to Step~1,  
testing the action of $\hd^2\p_t$ in the equations \eqref{shp.m} against $\hd^2\vr_{t}$
in the $L^2(\Omega)$-inner product yields
 \begin{align}
\label{rep.vt}
\Sup\norm{\hd^2 \vr_{t}(t)}_{0}^2+  \int_{0}^{T}\norm{\hd^2\vr_{t}}_1^2 + \sum_{\local=1}^K\int_{0}^{t}\abs{  \LM\bset{(\sqrt{\xl}\vr_t\cdot \nr)\circ\thetal }  }_{3,\DL}^2\le \int_{0}^{T}\RE.
\end{align}
The inequality \eqref{rep.vt} provides that
$\int_{0}^{T}\abs{\vr_t}_{2.5}^2 \le \int_{0}^{T}\RE $. We infer from Step~2 of the proof of
Lemma~\ref{lem.lv.1} that by viewing a time-derivative of the
equations \eqref{shp.m} as an elliptic Dirichlet problem for
$\vr_t$ that
\begin{align}
\label{ep.vt.3} 
\int_{0}^{T}\norm{\vr_{t}}_3^2 \le \int_{0}^{T}\RE.
\end{align}

\subsection*{Step~3: The   $\mu$-independent  estimates for $\vr$}
 Repeating
Step~1, we
test four tangential-derivatives of \eqref{shp.m} against 
$\hd^4 \vr$
in the $L^2(\Omega)$-inner product and integrate by parts in the
integral   $-\int_{\Omega}\hd^4\bpset{\rro
  \Aru_r^j\bpset{\Aru_r^k \vr\cp{k}}\cp{j}} \hd^4\vr$ to find that
\begin{align}
  \label{tv.h4}
  \begin{aligned}[b]
    \frac{1}{2}\frac{d}{dt}\int_{\Omega} \abs{\hd^4\vr}^2
+
\sum_{\local=1}^K \int_{\DL}
\LM&\bbset{\hd^4\pset{\sqrt{\xl}\vr\cdot
        \nr}\cp{\beta}\circ\thetal}\gfrk ^{\alpha\beta} \LM
    \bbset{ \hd^4(\sqrt{\xl}\vr
  \cdot \nr) \cp{\alpha} \circ\thetal} 
\\
+
    \int_{\Omega}\rro\abs{ \Aru_\cdot^k \hd^4 \vr
      \cp{k}}^2 
=&\underbrace{\int_{\Gamma}  \bset{\rro \sqrt{\gr}\Jr^{-1}\bpset{
\varepsilon_{\cdot ji}  \int_{0}^{t}\p_t \Aru_j^s \hd^4\vr^i}
\cp{s}\times \nr }\cdot\hd^4 \vr}_{\RE} 
\\
&+
\underbrace{
\int_{\Omega} \hd^4\Bbset{\Kr
-(\rro 
\Aru_r^j) \cp{j} \Aru_r^k \vr
      \cp{k} 
}\hd^4\vr
}_{\I} 
+\RE.
  \end{aligned}
 \end{align}
We use an $H^{-0.5}(\DL)$-duality pairing in analyzing the first
term appearing in the right-hand side of \eqref{tv.h4} and conclude a
good estimate thanks to the time-integral.
Tangentially integrating by parts in $\I$, we find that
\begin{align*}
  \int_{0}^{T}\I \le \int_{0}^{T}\RE.
\end{align*}
Hence, the time-integral of \eqref{tv.h4} yields
\begin{align}
  \label{tv.h4.es}
  \Sup \norm{\hd^4 \vr(t)}_0^2 + \int_{0}^{T}\norm{\hd^4 \vr}_1^2
+\sum_{\local=1}^K\int_{0}^{t}\abs{  \LM\bset{(\sqrt{\xl} \vr\cdot \nr)\circ\thetal }  }_{5,\DL}^2\le \int_{0}^{T}\RE.
\end{align}
The inequality \eqref{tv.h4.es} yields
\begin{align}
  \label{tv.g.45}
  \int_{0}^{T}\abs{\vr}_{4.5}^2 \le \int_{0}^{T}\RE.
\end{align}
Viewing the equations \eqref{shp.m} as an elliptic Dirichlet
problem 
for $\vr$, we infer from elliptic regularity and the inequalities \eqref{ep.vt.3}
and \eqref{tv.g.45} that
\begin{align}
  \label{tv.5}
  \int_{0}^{T}\norm{\vr}_5^2 \le \int_{0}^{T}\RE.
\end{align}

\subsection*{Step~4: Concluding the proof of Lemma~\ref{lem.e.i}}
The sum of the inequalities \eqref{ep.vtt}, \eqref{ep.vt.3},
\eqref{tv.h4.es} and \eqref{tv.5} completes the
proof of Lemma~\ref{lem.e.i}.
Taking $\delta$ sufficiently small in the inequality \eqref{ep.ind}
yields a polynomial-type inequality of the form \eqref{poly-type
  ineq}. Hence, for
sufficiently small $T=T _{\kappa}(\epsilon)>0$ and independently of $\mu>0$,
\begin{align}
\label{EE.ME}
\Sup \EE \le2\ME,
\end{align}
where the higher-order energy function $\EE$ is defined in \eqref{def.EE}.

\subsection{The proof of Proposition~\ref{prop.hp}}
\label{sec:concl-proof-p}
  Proposition~\ref{prop.shp} establishes the existence and uniqueness
  of a
  solution $\vr $ to the $\mu$-problem \eqref{shp}.  
Given the
  $\mu$-independent estimate
  \eqref{EE.ME}, standard compactness arguments provide for the
  existence of the strong convergence, as $\mu$ tends to zero,
  \begin{alignat*}{4}
    \er &\to\ec\ &\text{in }L^2(0,T; H^4(\Omega)),  &&
\Kr &\to    \Kc&&\text{in } L^2(0,T; H^2(\Omega)),\\
    \vr &\to \vc\ &\text{in }L^2(0,T; H^4(\Omega)),&&
  \hrm &\to \hc\ & &\text{in } L^2(0,T;  H^{1.5}(\Gamma)),\\
    \vr_t &\to \vc_t\ &\text{in }L^2(0,T; H^2(\Omega)),&\qquad&  \crr &\to c(t)\ && \text{in } L^2(0,T; H^{1.5}(\Gamma)).
  \end{alignat*}
The definitions \eqref{shp.K},
\eqref{shp.h}, \eqref{shp.c}  respectively define
$\Kr$, $\hrm$,
$\crr$. The limiting vectors $\Kc$, $\hc$ and $c(t)$ are
respectively defined by \eqref{mkp.m.21}, \eqref{DN.h} and \eqref{hp.c}. Letting
$\phi\in L^2\pset{0,T;H^1(\Omega)}$, we have that the variational form
of  the $\mu$-problem \eqref{shp} is
\begin{align*}
  \int_0^T\int_\Omega \vr_t\cdot\phi
+
 \int_0^T\int_\Omega\Aru_r^k\vr\cp{k}\bpset{\rro\Aru_r^j\phi}\cp{j}
&=
\int_0^T\int_\Omega \Kr\cdot \phi
+
\int_0^T\int_\Gamma \bset{\hrm + \crr}\cdot\phi.
\end{align*}
We infer from the strong convergence of the sequences
$\pset{\er,\vr,\vr_{t}, \Kr,\hrm,\crr }$    that the limit
$\pset{\ec,\vc,\vc_t,\Kc, \hc, c(t)}$ satisfies
\begin{align*}
  \int_0^T\int_\Omega \vc_t\cdot\phi
+
 \int_0^T\int_\Omega\Amu_r^k\vc\cp{k}\bpset{\varrho\Amu_r^j\phi}\cp{j}
&=
\int_0^T\int_\Omega \Kc\cdot \phi
+
\int_0^T\int_\Gamma \bset{\hc + c(t)}\cdot\phi.
\end{align*}
Thus, $\vc$ is a solution of the nonlinear heat-type  $\kappa\epsilon$-problem
\eqref{hp} on a  
time-interval $[0,T]$  for some $T=T _{\kappa}(\epsilon)>0$. Standard arguments provide that
 $\ec(0)=e$ and  $(\vc,\vc_t,\dots,\vc_{tttt} )|_{t=0} =(u_0,
 \mathbf{v}_1,\dots,\mathbf{v}_4)$. Furthermore, according to the
 inequality \eqref{EE.ME}, 
\begin{align*}
\Sup\norm{\vc_{tt}(t)}_{0}^2
  +\sum_{a=0}^2\int_{0}^{T}\norm{\p^a_t\vc}_{5-2a}^2
+\sum_{a=0}^2\int_{0}^{T}\abs{  \p^a_t \vc\cdot \nm}_{5-2a}^2
\le 2\ME. 
\end{align*}
By the higher-order regularity stated in Proposition~\ref{prop.shp}, we infer that 
\begin{align}
  \label{e0.m}
 \Sup \norm{\vc_{tttt}(t)}_{0}^2
 +\sum_{a=0}^4\int_{0}^{T}\norm{\p^a_t\vc}_{9-2a}^2
 +\sum_{a=0}^4\int_{0}^{T}\abs{  \p^a_t \vc\cdot \nm}_{9-2a}^2<\infty. 
\end{align}
\subsection{The proof of Theorem~\ref{thm.mkp}}
\label{sec:concl-proof-theor}
By Lemma~\ref{lem.equiv}, the heat-type $\kappa\epsilon$-problem
\eqref{hp} is equivalent to the $\kappa\epsilon$-problem
\eqref{mkp}. Hence,  Proposition~\ref{prop.hp} establishes the existence
and uniqueness of a solution to the $\kappa\epsilon$-problem on a time-interval $[0,T]$ for some $T=T _{\kappa}(\epsilon)>0$ verifying $(\vc,\vc_t,\dots,\vc_{tttt} )|_{t=0} =(u_0, \mathbf{v}_1,\dots,\mathbf{v}_4)$. 
The inequality \eqref{e0.m} establishes the inequality  
\eqref{mu.cont.est.v}. 
 
\subsection{The proof of Theorem~\ref{thm.akp}}
\label{sec:proof-theor-refthm}
 
A unique solution to the $\kappa\epsilon$-problem \eqref{mkp} exists
by Theorem~\ref{thm.mkp}.

 Taking $\delta$ sufficiently small in the inequality \eqref{mu.ind}
yields a polynomial-type inequality of the form \eqref{poly-type
  ineq}. Hence, for
sufficiently small $T=T _{\kappa}>0$ and independently of $\epsilon>0$,
\begin{align}
\label{EM.MM}
\Sup \EM \le2\MM,
\end{align} 
where the higher-order energy function $\EM$ is defined in \eqref{EM.l}.

For   $\phi\in L^2 (0,T;H^1(\Omega))$ such that $\phi\cdot \nm\in L^2(0,T;H^1(\Gamma))$, the variational equation  for the $\kappa\epsilon$-problem \eqref{mkp} is 
\begin{align}
\label{mkp.var}
\begin{aligned}[b]
  \int_{0}^{T}\int_{\Omega} \rho_0 \vc_t \cdot\phi -
  \int_{0}^{T}\int_{\Omega}  \rho_0^2 \Jm ^{-2} \amu_i^k \phi^i \cp{k} +
  \kappa\int_{0}^{T} \int_{\Omega} \rho_0\JT \amu_i^k (\rho_0 \Jm
  ^{-1} \phi^i ) \cp{k} &\\
+ \int_{0}^{T}\int_{\Gamma}\beta_{\epsilon}(t) \sqrt{\gm} \phi\cdot \nm+
\int_{0}^{T}\int_{\Gamma}\bset{\sigma\ec\cp{\beta} + \kappa \vc\cp{\beta}} ^i (\nm^i\sqrt{\gm}\gm
^{\alpha\beta} \phi\cdot \nm)\cp{\alpha}&=0.
\end{aligned}
\end{align}  We infer from Section~\ref{sec:concl-proof-p} and the
pointwise convergence $\beta_{\epsilon}(t)\to \beta(t)$  that
the variational equation \eqref{mkp.var} converges to the variational equation  \eqref{akp.var} as $\epsilon$ tends to zero.
Hence, 
the $\epsilon=0$ limit $\vt$ of the solutions $\vc$ to the $\kappa\epsilon$-problem
\eqref{mkp} solves the $\kappa$-problem \eqref{akp}. By
Section~\ref{sec:comp-cond-mu}, the solution $\vt$ of the
$\kappa$-problem verifies 
$(\vt,\vt_t,\dots,\vt_{tttt} )|_{t=0} =(u_0,
\mathrm{v}_1,\dots,\mathrm{v}_4)$. According to \eqref{EM.MM},
\begin{align*}
\Sup \norm{\vt_{tt}(t)}_{0}^2 
 +
\int_{0}^{T}\abs{  \vt_{tt}\cdot \nt}_{1}^2
+
\sum_{a=0}^2\int_{0}^{T}\norm{\p^a_t\vt}_{5-2a}^2\le 2\MM.
\end{align*} It follows from the proof of Lemma~\ref{lem.m.i}  that by including two more time-derivatives in the definition \eqref{EM.l}  of the
energy function $\EM$, the solution $\vt$ of the
$\kappa$-problem \eqref{akp} satisfies
\begin{align*}
\Sup \norm{\vt_{tttt}(t)}_{0}^2 
 +
\int_{0}^{T}\abs{  \vt_{tttt}\cdot \nt}_{1}^2
+
\sum_{a=0}^4\int_{0}^{T}\norm{\p^a_t\vt}_{9-2a}^2<\infty. 
\end{align*}
This establishes the inequality \eqref{cont.est.v}.
 
\section{Well-posedness of the surface tension problem \eqref{Euler}} \label{sec:uniqueness-solutions}
In this section, we prove 
Theorem~\ref{thm.s.main} via the $\kappa$-independent a priori estimates of 
Section~\ref{sec:kappa independent estimates}.  
  
\subsection{Existence} \label{sec:limit-as-kappato0} We obtain a
solution $v$ to surface tension problem \eqref{Euler} in the
limit of $\vt$ as the parabolic parameter $\kappa$ tends
to zero.  According to Remark~\ref{rem.beta},  $\beta(t)=\beta$ in the
$\kappa=0$ limit.  Letting $\kappa=0$ in 
Section \ref{sec:kappa independent estimates}, we therefore conclude that the right-hand side of the inequality (\ref{the kappa-independent estimate})
depends only on  $M_0=P(E(0))$. That is, 
  for sufficiently small $T>0$ the   energy function
  $E(t)$ defined in \eqref{s.energy function} satisfies
\begin{align}
\label{s.go.a}
  \sup_{t\in[0,T]}E(t)\le2M_0.
\end{align}
 The
assumption (\ref{assum.Jt}) on $J$ remains valid by taking $T>0$ even smaller if necessary.
 Hence,
\begin{align*} 
	f(t)=\rho_0J^{-1}(t)\ge\tfrac{4}{3}\lambda. 
\end{align*}
Taking $T>0$ even smaller if necessary, we ensure that $\rho(t)=f\circ\eta^{-1}(t)$ satisfies 
\begin{align*}
	\rho(t) \ge \lambda\ \ \text{in } \overline{\Omega}(t). 
\end{align*}
Since $p(t)=\rho^2(t)-\beta> -\beta$ on $\Gamma(t)$, the boundary
condition \eqref{Laplace-Young} establishes that
\begin{align*}
	\sigma H(t) > -\beta\ \ \text{on } \Gamma(t).
\end{align*}

\subsection{Optimal regularity for the initial
  data} \label{sec:optim-regul-init} In order to obtain the $H^5(\Omega)$-regularity of our
existence theory, we  assumed that the given
initial data is of  $C^\infty$-class in 
Section~\ref{sec:assum-cinfty-class-1}. 
In fact, by virtue  of the estimate \eqref{s.go.a}, it  
  suffices for the regularity of the initial data to be such that $E(0)<\infty$.

\subsection{Uniqueness} \label{sec:exitence-uniqueness} We define 
\begin{align*}
	\mathbf{E}(v,t)=1
+
\sum_{a=0}^6\norm{\p^a_t\eta(t)}_{6-a}^2
+
 \abs{\p^4_tv\cdot n(t)}_1^2
+
\sum_{a=0}^4\abs{ \hd^2\!\p^a_t \eta\cdot n(t)}_{4.5-a}^2
. 
\end{align*}
We suppose that $(\eta_1,v_1,f_1)$ and $(\eta_2,v_2,f_2)$ are two solutions of the compressible surface tension problem \eqref{Euler} with $\mathbf{E}(v,0)$, for $v=v_1,v_2$, bounded by some $\mathbf{M}_0>0$.

Then by setting 
\begin{align*}
	\zeta=\eta_1-\eta_2, \ \ w=v_1-v_2, \ \ \varrho=\bset{f_1}^2-\bset{f_2}^2 ,
\end{align*}
 we have that $(\zeta,w,\varrho)$ satisfies 
\begin{subequations}
	\label{s.uniqueness problem} 
	\begin{alignat}
		{4} \ \zeta&=\int_0^tw&\quad&\text{in}\ \Omega\times(0,T], \\
		\rho_0 w^i_t+[a_1 ]{}_i^k\varrho\cp{k}&=\bset{a_2-a_1}_i^k\pset{\bset{f_2}^2}\cp{k}&\quad&\text{in}\ \Omega\times(0,T],\label{s.uniqueness: interior} \\
&\!\!\!\! \begin{aligned}[b]		\varrho=
                  -\sigma [g_1]^{
                    \alpha\beta}\zeta\cp{\alpha\beta}\cdot
                  n_1
-\sigma [g_1]^{
                    \alpha\beta}\eta_2\cp{\alpha\beta}\cdot
                  n_1\\
+\sigma [g_2]^{
                    \alpha\beta}\eta_2\cp{\alpha\beta}\cdot
                  n_2               
                \end{aligned}
&\quad&\text{on}\ \Gamma \times(0,T],\label{s.uniqueness: boundary} \\
 		(\zeta,w,\varrho)|_{t=0}&=(0,0,0)&\quad&\text{on}\ \Omega. \label{s.uniqueness: IC} 
	\end{alignat}
\end{subequations}
 
We will establish that $w=0$. Setting 
\begin{align}
\label{E.zeta}
	E^\zeta(t)=\sum_{a=0}^5\norm{\p^a_t\zeta(t)}_{5-a}^2, 
\end{align}
we follow Section~\ref{sec:kappa independent estimates} with $\kappa$
set to zero.

The Lagrangian curl operator $\curl_{\eta _1}$ applied to
\begin{align*}
	w^i_t+2\bset{A_1 }{}_i^kf_1\cp{k}=2\bset{A_2}_i^kf_2\cp{k}
\end{align*}
provides the following vorticity equation for the difference $w=v_1-v_2$:
\begin{align}
  \label{un.vort}
  \curl_{\eta _1} w_t= 2\varepsilon_{\cdot ji} \bset{A_1}_j^s (\bset{A_2}_i^kf_2\cp{k})\cp{ s}.
\end{align}
The time integral of the surface tension problem satisfied by $v_2$ yields
\begin{align*}
v_2-u_0=
  - 2  \int_0^{t} \bset{A_2 }_\cdot^k f_2 \cp{k},
\end{align*} 
by which we infer that
\begin{align*}
\int_{0}^{T} \norm{   D^2\int_0^{t} \bset{A_2 }_\cdot^k f_2 \cp{k} }_3^2\le T\,P(\Sup \norm{ v_2 (t)}_5^2)\le C\,T \,\mathbf{M}_0.
\end{align*} 
The curl-estimates for $w$ therefore follow from the analysis proving Lemma~\ref{lem: curl estimates} with the vorticity equation \eqref{un.vort} replacing the homogeneous vorticity equation \eqref{E.Lagrangian vorticity}.

Repeating the energy estimate for  the fourth time-differentiated problem in Proposition~\ref{prop: divergence: energy estimates}, the highest-order term of the interior forcing term $\int_0^t\int_\Omega\p^4_t\bset{ \bset{a_2-a_1}_i^k\pset{\bset{f_2}^2}\cp{k}}w_{tttt}^i$ obeys \begin{align*}
	\int_0^t\int_\Omega \bset{a_2-a_1}_i^k\pset{\bset{f_2}^2}_{tttt}\cp{k}w_{tttt}^i\le C\int_0^T\norm{Dv_{2ttt}}_1\norm{w_{tttt}}_0\le T\,P\pset{\Sup E^\zeta(t)}.
\end{align*}
The boundary integral corresponding with
\begin{align*}
  -\sigma [g_1]^{
                    \alpha\beta}\eta_2\cp{\alpha\beta}\cdot
                  n_1
+\sigma  [g_2]^{
                    \alpha\beta}\eta_2\cp{\alpha\beta}\cdot
                  n_2    
=
-\sigma [g_1]^{
                    \alpha\beta}\eta_2\cp{\alpha\beta}&\cdot
[                  n_1-                  n_2]
-\sigma [g_1-g_2]^{
                    \alpha\beta}\eta_2\cp{\alpha\beta}\cdot
                  n_2
\end{align*}
of (\ref{s.uniqueness: boundary}) is similarly bounded.

We notice that with 
\begin{align*}
	\int_\Omega \rho_0^2 \pset{ J_1}^{-3} \bset{a_1 }{}_p^q\,v_{1ttt}{}^p\cp{q} \bset{a_1 }{}_r^s w^r_{tttt}\cp{s} =& \int_\Omega \rho_0^2 \pset{ J_1}^{-3} \bset{a_1 }{}_p^q\,w_{ttt}^p\cp{q}\bset{a_1 }{}_r^sw^r_{tttt}\cp{s} \\
	&+\int_\Omega \rho_0^2 \pset{ J_1}^{-3}\bset{a_1 }{}_p^q\,v_{2ttt}{}^p\cp{q}\bset{a_1 }{}_r^s w^r_{tttt}\cp{s}, 
\end{align*}
we preserve an energy estimate for $\norm{ \bset{a_1 }{}_r^s w^r_{ttt}\cp{s}}_0^2$ in Proposition~\ref{prop: divergence: energy estimates}.

Following the arguments in Step 3  of Section~\ref{sec:kappa independent estimates} provides control of the divergence and normal trace of the functions of $E^\zeta(t)$. 
Proposition~\ref{prop:Hodge} and the initial condition (\ref{s.uniqueness: IC})  imply that
\begin{align*}
	\Sup E^\zeta(t)\le T\,P\pset{\Sup E^\zeta(t)}, 
\end{align*}
for the higher-order energy function $E^{\zeta}(t)$ defined in
\eqref{E.zeta}. Using the polynomial-type inequality \eqref{poly-type ineq},
we infer  that $w=0$ as desired.

\section{The asymptotic limit as 
  surface tension tends to zero}
\label{sec:uniqueness-solutions.z}

In this section, we establish an existence theory for the
zero surface tension limit of \eqref{Euler}  via a priori estimates  that
are independent of the surface tension parameter $\sigma>0$.  This asymptotic limit
holds whenever   the initial data satisfies the Taylor
sign condition \eqref{Taylor}, and provides the following:
\begin{align*}
0< \nu\int_{\Gamma}\abs{\hd^4\eta\cdot n}^2\le-   \int_{\Gamma}\frac{1}{\sqrt{g}}N^ja_\ell^ja_{\ell}^k (\rho_0^2 J
    ^{-2})\cp{k} \abs{\hd^4\eta\cdot n}^2.
\end{align*}
We recall that according to Theorem~\ref{thm.s.main}, a solution to \eqref{Euler}
  satisfies 
 \begin{align}
   \label{stp.est}
   \Sup  \sum_{a=0}^5\norm{\p^a_t\eta(t)}_{5-a}^2 +
  \abs{v_{ttt}\cdot n(t)}_1^2 
+
  \sum_0^2\abs{\hd^2\p^{a}_tv\cdot n(t)}_{2.5-a}^2 \le C_\sigma,
 \end{align} 
for a finite bound $C _{\sigma}>0$ depending on $1/\sigma$. The
Taylor-sign-condition assumption of Theorem~\ref{thm.z.main} provides for  $\sigma$-independent a priori estimates under the
higher-order energy function $\ES$ defined below in \eqref{defn:ES}.

\subsection{Assuming $C^\infty$-class initial data}
\label{z.sec:assum-cinfty-class-1} In our construction of  solutions to the zero surface tension limit of \eqref{Euler}, we will assume that the
initial data $(\rho_0,u_0, \Omega)$  is of $C^\infty$-class and
satisfy the conditions \eqref{l.boundedness.r0}, \eqref{Taylor} and
\eqref{z.compatibility conditions}, as in
Appendix~\ref{sec:appendix:init}.
Later, in    
Section ~\ref{sec:optim-regul-init.z}, we will recover the optimal regularity
of the initial data stated in Theorem~\ref{thm.z.main}.

\subsection{The $\sigma$-problem}
\label{sec:surf-tens-probl}

For $\sigma>0$ taken sufficiently small, solutions to the following problem are provided by Theorem~\ref{thm.s.main}.  To indicate  the dependence on the surface tension paramater
$\sigma$  of all the variables in the following problem,
we place the symbol $\smallsmile \label{n:vz}$ over each of the variables.   
  
\begin{defn}
	[The  $\sigma$-problem] \label{defn.stp} For
        $\sigma>0$, we define $\vz$ as the solution of
	\begin{subequations}
		\label{stp} 
		\begin{alignat}
			{4} \label{stp.m} \rho_0\vz^i_t + \az_i^k\pset{\rho_0^2 \Jz^{-2}}\cp{k}&=0&&in\ \Omega\times(0,T_\sigma],\\
			\label{stp.bc}
                        \rho_0^2\Jz^{-2}&=\beta_\sigma(t)-\sigma \gz
                        ^{\alpha\beta}\ez\cp{\alpha\beta}\cdot \nz &\ \ &on\ \Gamma\times(0,T_\sigma],\\
			\label{stp: init cond} \pset{\ez,\vz}|_{t=0}&=\pset{e,u_0}&&on\ \Omega.
		\end{alignat} 
	\end{subequations}
The  function $\beta_\sigma(t)$ appearing in
        the right-hand side of \eqref{stp.bc} is defined as
\label{n:beta stp} 
	\begin{align*}
\beta_\sigma(t)=\beta+\sum_{a=0}^{6}
                \frac{t^a}{a!}\p^a_t\bset{\rho  _0^2\Jz^{-2}- \beta +\sigma \gz
                        ^{\alpha\beta}\ez\cp{\alpha\beta}\cdot \nz  }\big|_{t=0}. 
	\end{align*}
\end{defn}

\begin{rem}
  The initial data satisfy the compatibility conditions
  \eqref{z.compatibility conditions}.  Thus,
  \begin{align}
    \label{beta.sigma}
     \beta_\sigma(t)=\beta+\sigma\sum_{a=0}^{6}
                \frac{t^a}{a!}\p^a_t\bset{  \gz
                        ^{\alpha\beta}\ez\cp{\alpha\beta}\cdot \nz }\big|_{t=0}. 
  \end{align}
The $\sigma=0$  formal limit of $\beta_\sigma(t)$ is $\beta$. It
follows that the $\sigma $-problem~\eqref{stp} is
  asymptotically consistent with the   zero surface tension limit of \eqref{Euler}.
\end{rem}
\begin{rem}We use (\ref{stp.bc}) to compute the following identities: for $a=0,\dots,6,$
\begin{align}
	\label{stp: compat cond} \partial^a_t\bset{ \rho_0^2\Jz^{-2}}|_{t=0} = \p^a_t\beta_\sigma(t)|_{t=0}-  \sigma\p^a_t\bset{  \gz
                        ^{\alpha\beta}\ez\cp{\alpha\beta}\cdot \nz}|_{t=0}.
\end{align}
\end{rem}

\subsection{The a priori estimates for the $\sigma$-problem }
\label{sec:sigma independent estimates}

For $\sigma>0$, we define  
\begin{align}
\label{defn:ES}  
\begin{aligned}[b]
  \ES= 1
&+
 \sum_{a=0}^7\norm{\p^a_t\ez(t)}_{4.5-\frac{1}{2}a}^2 
+
\sum_{a=0}^5\norm{\p^a_t\Jz(t)}_{4.5-\frac{1}{2}a}^2  
+
 \norm{\p^{6}_t\Jz(t)}_{1}^2  
\\
+&
\sum_{a=0}^5\norm{\sqrt{\sigma}\p^{a}_t\ez(t)}_{5.5-\frac{1}{2}a}^2  
+
\norm{\sqrt{\sigma}\p^5_t\vz(t)}_{2}^2  
+
\sum_{a=0}^1 \abs{\sqrt{\sigma}\p^{5+a}_t\vz\cdot\nz(t)}_{2-a}^2
\\
+
\sum_{a=0}^3&\norm{\sigma\p^{a}_t\ez(t)}_{6.5-\frac{1}{2}a}^2  
+
\norm{\sigma\vz_{ttt}(t)}_{4}^2  
+
\sum_{a=0}^1 \abs{\sigma\p^{3+a}_t\vz \cdot\nz(t)}_{4-\frac{1}{2}a}^2
+
 \abs{\sigma\p^5_t\vz \cdot\nz(t)}_{2.5}^2.
\end{aligned}
\end{align}

We will   allow constants to depend on
$1/\delta>0$:
\begin{defn}
 	[Notational convention for constants depending on $1/ \delta>0$] We let $\PS$ denote a generic polynomial with constant and coefficients depending on $1/\delta>0$.
	
	We define the constant $\MS>0$ by 
	\begin{align}
		\MS=\PS\pset{\norm{u_0}_{100}, \norm{\rho_0}_{100}}. 
	\end{align}
	
	We let $\RS\label{n:RS}$ denote generic lower-order terms satisfying 
	\begin{align*}
		\int_0^T\RS\le \MS+\delta \Sup \ES+T\,\PS(\Sup \ES). 
	\end{align*}
\end{defn}
We  infer from the estimates \eqref{universal constant} and \eqref{stp.est} that for  $T>0$ 
taken sufficiently small,
\begin{subequations}
\label{z.assum}
  \begin{align}
    \label{z.assum.Jt} \frac{1}{2}\le \Jz \le \frac{3}{2}\ \ \text{for
      all } t\in[0,T] \text{ and } x\in {\Omega}.
  \end{align}
Since the initial data satisfy
the Taylor sign condition  \eqref{Taylor},  we  also assume  that
  \begin{align}
    \label{Taylor.t}
    0< \nu\le- \frac{1}{\sqrt{\gz}}N^j \az_\ell^j \az_\ell ^k (\rho_0^2
    \Jz ^{-2})\cp{k}\ \ \text{for
      all } t\in[0,T] \text{ and } x\in {\Gamma}.
  \end{align}
\end{subequations}

 \begin{lem}
	[A priori estimates for the $\sigma$-problem]
We       let $ \vz $ solve the $\sigma$-problem 
        $\eqref{stp}$ \label{lem.ES} on a time-interval $[0,T]$, for
        some $T = T _ \sigma>0$. Then independent of $1>>\sigma >0$,
	\begin{align}
		\label{est: ES} \Sup\ES\le \int_0^{T}\RS. 
	\end{align}
\end{lem}
 
We will establish  Lemma~\ref{lem.ES} in the following nine steps:  

\subsection*{Step 1: The $\sigma$-independent
  curl-estimates} \label{sec:curl-estimates.s} 
We infer the following lemma
from the arguments proving Lemma~\ref{lem: curl estimates}.
\begin{lem}[The $\sigma$-independent curl-estimates]
	\label{z.lem.curl} 
	\begin{align*}
\Sup \sum_{a=0}^7&\norm{\curl \p^a_t\ez
  (t)}_{3.5-\frac{1}{2}a}^2
+
\Sup\sum_{a=0}^1\norm{\sqrt{\sigma}\curl \p^{3+a}_t\vz
  (t)}_{2.5-\frac{1}{2}a}^2
+
\Sup\norm{\sqrt{\sigma}\curl \p^5_t\vz
  (t)}_{1}^2
\\
&
+
\Sup			\sum_{a=0}^3\norm{\sigma\curl \p^a_t\ez
  (t)}_{5.5-\frac{1}{2}a}^2
+
\Sup \norm{\sigma\curl \p^{3}_t\vz (t)}_{3}^2
\le\int_{0}^{T}\RS.
	\end{align*}
\end{lem}
\subsection*{Step 2: The $\sigma$-independent
 estimates for $\p^7_t\Jz$, $\p^6_t\Jz$ and
 $\sqrt{\sigma}\p^6_t\vz\cdot \nz$} 
We recall that
\begin{align*}
  \fz = \rho_0\Jz ^{-1}
\end{align*}
is the Lagrangian density. Using the identity $\Jz^{-1}\Jz_t
= \Div_{\ez}\vz$, we have that 
\begin{align*}
\p_t  \fz^2 &= -2 \fz ^2\Div_{\ez}\vz,
\\
\p^2_t\fz^2& = -2 \fz^2 \Div_{\ez}\vz_t - 2(\fz^2\Az_r^s)_t\vz^r \cp{s}.
\end{align*}
Letting the operator $-2\bset{ \fz \Az_i^j \p_j}\Jz ^{-1} $ act in the
Euler equations \eqref{stp.m} yields
\begin{align*}
-2\fz^2 \Div_{\ez} \vz _t -2 \fz\Az_i^j \bbset{ \Az_i^k (\rho_0^2
  \Jz ^{-2}) \cp{k}}\cp{j}=  \vz_t^i \Az_i^j\fz^2 \cp{j}.
\end{align*}
Using the Euler equations \eqref{stp.m}  
to write $\vz_t^i=-\rho_0^{-1} \az_i^k \fz^2\cp{k}$, we infer
that $\fz ^ 2$ satisfies
\begin{align}
  \label{wv.J}
\p^2_t\fz^2- 2\fz\Az_i^j \bbset{ \Az_i^k \fz^2\cp{k}}\cp{j}= \underbrace{
-\rho_0^{-1} \az_i^k \fz^2\cp{k}  \Az_i^j\fz^2
\cp{j}-2(\fz^2\Az_i^j)_t  \vz^i  \cp{j}
}_{\Fz}. 
\end{align}
Since $\Fz$ scales like $D\Jz+D\vz$, it follows that
$\p^6_t\Fz$ is in $L^2(\Omega)$.

Similar to the tangential identity \eqref{indentity: smoothed tangential},
\begin{align} 
	\label{z.tid} 
		\rho_0 \Jz ^{-1} \vz_t\cdot\ez\cp{\gamma}
                =\hd_\gamma\Big[
                {\sigma\gz^{\mu\nu}\ez\cp{\mu\nu}\cdot\nz} 
-\beta_ \sigma(t) \Big]\ \ \ \text{on } \Gamma,
\end{align}
thanks to  the Euler equations
\eqref{stp.m} and the Laplace-Young boundary condition \eqref{stp.bc}.
 \begin{prop}
	[Energy estimates for  the action of $\p^6_t$ in
        the wave-type equation \eqref{wv.J}]
\label{z.prop.z8}
	\begin{align*}
			 \Sup \norm{
                           \p^7_t\Jz(t)}_0^2&+\Sup\norm{\p^6_t\Jz(t)}_1^2
+\Sup\abs{\sqrt{\sigma}\p^6_t\vz\cdot
                           \nz(t)}_1^2
\le \int_{0}^{T}\RS .
	\end{align*}
\end{prop}
\begin{proof}
	Testing six time-derivatives of \eqref{wv.J} against $ \rho_0^{-2}\Jz ^{3}\p^7_t \fz^2$ in the
        $L^2(\Omega)$-inner product, and integrating by parts with
        respect to $\p_j$ in the  integral $-2\int_\Omega
\rho_0^{-1}       \az_i^j \p^6_t\bbset{ \az_i^k\fz^2\cp{k} }\cp{j}\p^7_t\fz^2$ yields
	\begin{align}
\label{z8.p}
\begin{aligned}[b]
\frac{1}{2}\frac{d}{dt}  \int_\Omega &\rho_0 ^{-2}\Jz ^{3}  \abs{\p^7_t\fz^2}^2 
+ 
2
  \underbrace{
\int_\Omega \rho_0^{-1} \p^6_t\bbset{ \az_i^k
      \fz^2\cp{k} }\az_i^j\p^7_t\fz^2 \cp{j}
}_{\I }
   -
2 \underbrace{\int_\Gamma  \rho_0^{-1}\p^6_t\bbset{ \az_i^k
      \fz^2\cp{k} }\az_i^j N^j\p^7_t\fz^2  }_{\II}
\\
&=\underbrace{\frac{1}{2}\int_{\Omega}(\rho_0 ^{-2}\Jz ^{3} )_t \abs{\p^7_t\fz^2}^2 }_{\RS}+
2  \underbrace{
\sum_{\counter=1}^6c_\counter\int_\Omega \fz\Az_i^j \bpset{ \p^\counter_t\Jz ^{-1}\p^{6-\counter}_t\bbset{ \az_i^k
      \fz^2\cp{k} }} \cp{j}\rho_0^{-2}\Jz ^{3}\p^7_t\fz^2 
}_{\RS}
\\
&+
2  \underbrace{
\sum_{\counter=1}^6c_\counter\int_\Omega \p^\counter_t\pset{\fz\Az_i^j }\p^{6-\counter}_t\bpset{  \Az_i^k
      \fz^2\cp{k} } \cp{j} \rho_0^{-2}\Jz ^{3}\p^7_t\fz^2 
}_{\RS}
+
\underbrace{
\int_\Omega \p^6_t\Fz\rho_0^{-2}\Jz^3\p^7_t\fz^2 
}_{\RS}
\\
&+
2\underbrace{\int_{\Omega} \rho_0^{-1}\az_i^j (\Jz ^{-1})\cp{j}
  \p^6_t\bset{\at_i^k\fz^2 \cp{k}}\Jz \p^7_t\fz^2
}_{\RS}
.
\end{aligned}
	\end{align}
We have
used the Cauchy-Schwarz inequality to analyze all of the terms in the
right-hand side of \eqref{z8.p}  except for the
highest-order terms of 
$\int_{\Omega} \p^4_t\pset{\fz\Az_i^j} \p^{2}_t\bpset{  \Az_i^k
      \fz^2\cp{k} } \cp{j} \rho_0^{-2}\Jz ^{3}\p^7_t\fz^2 $, where we
have used an $L^4$--$L^4$--$L^2$ H\"{o}lder inequality.

Writing $\p_t\fz^2 = -2 \rho_0^2 \Jz ^{-3}\Jz_t$, it follows that
\begin{align}
  \label{fz.Jz}
\p^7_t \fz ^2 =-2 \rho_0^2 \Jz ^{-3}\p^7_t\Jz-2\sum_{\counter=1}^{6}
c_\counter \p^\counter_t\pset{
\rho_0^2 \Jz ^{-3}}\p^{7-\counter}_t\Jz.
\end{align}
The identity \eqref{fz.Jz} provides that the equation
\eqref{z8.p} multiplied by $\frac{1}{2}$ is equivalent to
	\begin{align}
\label{z8}
\begin{aligned}[b]
\frac{d}{dt}  \int_\Omega\rho_0^2\Jz ^{-3}  \abs{\p^7_t\Jz}^2 
&+ 
\I 
   -
\II
=
\RS.
\end{aligned}
	\end{align}

\subsection*{Analysis of $\I$ in \eqref{z8}}
We equivalently write
\begin{align*}
\mathcal{I}= 
\underbrace{
\int_\Omega \rho_0^{-1}\az_i^k \p^6_t  \fz ^2\cp{k} \az_i^j\p^7_t\fz^2 \cp{j}
}_{\mathcal{I}_a }
+
\sum_{\counter=0}^5c_\counter\underbrace{\int_\Omega
\rho_0^{-1}
    \p^{6-\counter}_t\az_i^k \p^\counter_t\fz^2\cp{k} \az_i^j\p^7_t\fz^2 \cp{j}
}_{\mathcal{I}_{b,\counter }} .
\end{align*}
We have that
\begin{align*}
  \I_a = \frac{1}{2}\frac{d}{dt}\int_{\Omega} \rho_0^{-1}
  \abs{\az_\cdot^k \p^6_t\fz^2 \cp{k}}^2 
-
\int_{\Omega}\rho_0^{-1}
\az_i^k \p^6_t  \fz ^2\cp{k} \p_t\az_i^j\p^6_t\fz^2 \cp{j}
.
\end{align*}
Similar to \eqref{fz.Jz}, we have that $\p^6_t\fz^2$ is equal to
$-2 \fz^2\Jz ^{-1}\p^6_t\Jz$ plus lower-order terms. Thus,
\begin{align*}
  \I_a
=
2\frac{d}{dt}\int_{\Omega} 
\rho_0^{-1}  \fz^4\abs{\Az_\cdot^k \p^6_t\Jz \cp{k}}^2 +\RS.
\end{align*}
 For $\int_{0}^{T}\I_{b,\counter}$, we integrate by parts with respect to a
time-derivative of $\p^7_t\fz^2 \cp{j}$. For example,
\begin{align*}
  \int_{0}^{T} \I_{b,0} = \int_{\Omega} \rho_0^{-1}\p^6_t\az_i^k \fz ^2\cp{k} \az_i^j\p^6_t\fz^2 \cp{j} \big|_0^T  - \int_{0}^{T} \int_{\Omega} \rho_0^{-1}\p_t\bbset{ \p^6_t\az_i^k \fz ^2\cp{k} \az_i^j}\p^6_t\fz^2 \cp{j}  = \int_{0}^{T}\RS.
\end{align*}
Since $\p^6_t\az$ scales like $D\p^5_t\vz$,   the fundamental
theorem of calculus  ensures a good estimate for the term evaluated at time $t=T$.
The terms $\int_{0}^{T}\I_{b,\counter}$, $\counter=1,\dots,5$, are similarly analyzed:
\begin{align*}
  \sum_{\counter=0}^5 \I_{b,\counter} = \RS.
\end{align*}
This establishes that
\begin{align}
\label{z8.I}
  \mathcal{I}= 2 \frac{d}{dt}\int_{\Omega} \rho_0^{-1}\fz^4\abs{\Az_\cdot^k \p^6_t\Jz \cp{k}}^2 + \RS.
\end{align}

	\subsection*{Rewriting the boundary integral in
          \eqref{z8}} We use the trace of the action of $\p^6_t$ in the Euler
        equations \eqref{stp.m} to write $\II=
        \int_\Gamma\p^7_t\fz^2\p^7_t\vz^i\az_i^j N^j$, or equivalently,
        \begin{align*}
\II=
\int_\Gamma \p^7_t\fz^2\sqrt{\gz}\p^7_t\vz\cdot\nz .
        \end{align*}
 Using the Laplace-Young boundary
        condition \eqref{stp.bc},  we find
        that 
	\begin{align}
		\label{z7.i} 
		\begin{aligned}[b]
			{\II}
			&=\frac{1}{2}\frac{d}{dt}\int_\Gamma
                        \sqrt{\gz} \gz^{\alpha\beta} \,\sqrt{\sigma}
                        \p^6_t\vz\cp{\alpha}\cdot \nz\,\sqrt{\sigma}
                        \p^6_t\vz\cp{\beta}\cdot \nz  
 \\
			&\quad \underbrace{ -\sigma\int_\Gamma
                          \sqrt{\gz} \gz^{\alpha\beta}
                          \p^6_t\vz\cp{\alpha}\cdot \nz_t
                          \,\p^6_t\vz\cp{\beta}\cdot \nz 
}_{\j_1}- \frac{1}{2}
\underbrace{\int_\Gamma \pset{\sqrt{\gz}\gz^{\alpha\beta}}_t\,\sqrt{\sigma} \p^6_t\vz\cp{\alpha}\cdot \nz \,\sqrt{\sigma}\p^6_t\vz\cp{\beta}\cdot \nz }_{\RS} \\
			&\quad+ \underbrace{ \sigma\int_\Gamma
                          \p^6_t\vz\cp{\beta}\cdot\nz                            \sqrt{\gz}\gz^{\alpha\beta}
                          \p^7_t\vz \cdot \nz \cp{\alpha}
+
\sigma\int_\Gamma         \p^6_t\vz\cp{\beta}\cdot\nz \cp{\alpha}                            \sqrt{\gz}\gz^{\alpha\beta}
                          \p^7_t\vz \cdot \nz 
                       }_{ \j_{2} }\\
	&\quad+ \underbrace{ \sigma\int_\Gamma  \p^6_t\vz\cp{\beta}\cdot \nz(
          \sqrt{\gz}\gz^{\alpha\beta})\cp{\alpha} \p^7_t\vz \cdot \nz
-
\sigma \sum_{\counter=0}^{6}c_\counter\int_\Gamma \gz^{\alpha\beta} \p^\counter_t \ez ^j\cp{\alpha\beta} \p^{7-\counter}_t\nz^j\,\sqrt{\gz}\p^7_t\vz\cdot \nz }_{\j_{3}}\\
			&\quad- \underbrace{\sigma
                          \sum_{\counter=1}^7c_\counter\int_\Gamma\p^\counter_t{\gz^{\alpha\beta}}\p^{7-\counter}_t\bset{\ez
                            \cp{\alpha\beta}\cdot \nz
                          }\sqrt{\gz}\p^7_t\vz\cdot\nz }_{\j_{4}} 
+
\underbrace{
                        \int_{\Gamma} \p^7_t\beta_
                        \sigma(t) \sqrt{\gz}\p^7_t\vz\cdot \nz}_0. 
		\end{aligned}
	\end{align}

	\subsection*{Analysis of $\int_0^T \j_1$ in the time-integral
          of \eqref{z7.i}} The action of $\hd_ \alpha\p^5_t$ in the  tangential
        identity \eqref{z.tid} provides that 
        \begin{subequations}
            \label{z.tid.61} 
          \begin{align}\p^6_t\vz\cp{\alpha}\cdot \ez
            \cp{\gamma}=\rho_0^{-1}\Jz
            \bset{\hd_{\alpha\gamma}\p^5_t\big(\sigma\gz^{\mu\nu}\ez\cp{\mu\nu}\cdot\nz
              - \p^5_t\beta_ \sigma(t) \big) - \elz_{\gamma\alpha}},
          \end{align}
          where $\elz_{\gamma\alpha}$ is such that $\elz_{\gamma\alpha} \in H^{-0.5}(\Gamma)$, $\sqrt{\sigma}\elz_{\gamma\alpha} \in H^{0.5}(\Gamma)$ and is given by
          \begin{align}
            \label{z.tid.61.l} \elz_{\gamma\alpha}=
            \p^6_t\vz\cdot\pset{\ez\cp{\gamma}\rho_0
              \Jz^{-1}}\cp{\alpha}+\sum_{\counter=0}^4c_\counter\bset{\p^a_t\vz_t\cdot\p^{5-\counter}_t\pset{\ez\cp{\gamma}\rho_0
                \Jz^{-1}} }\cp{\alpha}.
          \end{align}
        \end{subequations}
	Setting
        $\ellz_{\gamma}^{\alpha\beta}=\rho_0^{-1}\Jz\sqrt{\gz}\gz^{\alpha\beta}\gz^{\gamma\delta}\vz\cp{\delta}\cdot\nz$
        we use the tangential identity \eqref{z.tid.61},
        together with the outward normal differentiation formula  \eqref{t derivative: outward normal vector}, to find that 
	\begin{align}
\label{z7.j1.e}
          \begin{aligned}[b]
            \j_1= \underbrace{ \sigma\int_\Gamma \ellz_{\gamma}^{\alpha\beta}
              \hd_{\gamma\alpha} \p^5_t\pset{\gz^{\mu\nu}\ez\cp{\mu\nu}\cdot\nz}\,
              \p^6_t\vz\cp{\beta}\cdot\nz}_{\j_1{}'} +
\underbrace{ \int_{\Gamma} \ellz_{\gamma}^{\alpha\beta}\sqrt{\sigma} \elz_{\gamma\alpha}\,\sqrt{\sigma}
              \p^6_t\vz\cp{\beta}\cdot\nz}_{\RS} 
+\RS.
          \end{aligned}
	\end{align}
We integrate by parts with respect to a time-derivative of $\p^6_t\vz\cp{\beta}\cdot\nz$ to write 
	\begin{align*}
&			\int_0^T\j_1{}'=
\\
&-
\underbrace{
\int_\Gamma
                          \hd_{\gamma}\p^5_t\bset{\gz^{\mu\nu}\ez\cp{\mu\nu}\cdot\nz}\,
                    \sigma^2      \hd_\alpha\bset{\ellz_{\gamma}^{\alpha\beta}
                            \p^5_t\vz\cp{\beta}\cdot\nz
                          }}_{\j_1{}'_{A}}\Big|_0^T
 +
 \int_0^T\underbrace{\sigma^2\int_\Gamma \hd_{\gamma}\p^5_t\bset{\gz^{\mu\nu}\ez\cp{\mu\nu}\cdot\nz}\, \hd_\alpha\bset{\ellz_{\gamma}^{\alpha\beta} \p^5_t\vz\cp{\beta}\cdot\nz_t}}_{\j_1{}'_{B}}\\
			& + \int_0^T \underbrace{\sigma^2\int_\Gamma \hd_{\gamma}\p^5_t\bset{\gz^{\mu\nu}\ez\cp{\mu\nu}\cdot\nz}\, \hd_\alpha\bset{\p_t\ellz_{\gamma}^{\alpha\beta} \p^5_t\vz\cp{\beta}\cdot\nz}}_{\j_1{}'_{C}} + \int_0^T \underbrace{ \sigma^2\int_\Gamma \hd_{\gamma}\p^6_t\bset{\gz^{\mu\nu}\ez\cp{\mu\nu}\cdot\nz}\, \hd_\alpha\bset{\ellz_{\gamma}^{\alpha\beta} \p^5_t\vz\cp{\beta}\cdot\nz}}_{\j_1{}'_{D}}.
	\end{align*} 
Since $\sigma\p^5_t\vz\cdot \nz$ is in $H^{2.5}(\Gamma)$, we may take
$\sigma$ sufficiently small so that
\begin{align*}
  \abs{\sigma \p^5_t\vz \cdot \nz(T)}_2^2 \le\sqrt{\sigma} C   \abs{ \sqrt{\sigma}
    \p^5_t\vz (T)}_{1.5}   \abs{\sigma \p^5_t\vz
    \cdot \nz(T)}_{2.5}\le \int_{0}^{T}\RS.
\end{align*}
Thus, the Cauchy-Schwarz inequality provides that
  \begin{align}
    \label{z8.j1Ai}    \j_1{}'_{A}\big|_0^T &=- \int_{\Gamma} \gz^{\mu\nu}
      \hd_{\gamma}\bset{\p^4_t\vz\cp{\mu\nu}\cdot\nz}\, \sigma^2
      \ellz_{\gamma}^{\alpha\beta} \hd_\alpha\bset{
        \p^5_t\vz\cp{\beta}\cdot\nz} \big|_0^T+
    \int_{0}^{T}\RS=    \int_{0}^{T}\RS.
  \end{align}
Since $\sigma\p^5_t\vz\cdot \nz$ is in $H^{3.5}(\Gamma)$, we find by use of an $H^{-0.5}(\Gamma)$-duality pairing that
\begin{align}
\label{z8.j1Aii}
\j_1{}'_{B}=
\underbrace{\sqrt{\sigma}  \int_\Gamma
  \hd_{\gamma}\bset{\gz^{\mu\nu}\sigma\p^4_t\vz\cp{\mu\nu}\cdot\nz}\,
\ellz_{\gamma}^{\alpha\beta}
\sqrt{\sigma}\p^5_t\vz\cp{\alpha\beta}\cdot\nz_t}_{\RS }+ \int_{0}^{T}\RS.
\end{align} 
We employ the Cauchy-Schwarz inequality to
conclude that
\begin{align} 
  \label{z8.j1Aiii}
\int_{0}^{T}      \j_1{}'_C  \le \int_{0}^{T}\RS.
\end{align}
We employ the Cauchy-Schwarz inequality or  an $H^{-0.5}(\Gamma)$-duality pairing to
conclude that
\begin{align}
  \label{z8.j1Aiv}
\int_{0}^{T}      \j_1{}'_{D}  \le \int_{0}^{T}\RS.
\end{align} 
The inequalities \eqref{z8.j1Ai}, \eqref{z8.j1Aii},
\eqref{z8.j1Aiii}, \eqref{z8.j1Aiv} establish
 that
\begin{align}
  \label{z8.j1}
\int_{0}^{T}\j_{1} \le \int_{0}^{T}\RS.
\end{align}	 
	\subsection*{Analysis of $\int_0^T\j_2$  in the time-integral
          of \eqref{z7.i}} We equivalently
        write $\j_2$ as 
 	\begin{align}
		\label{z8.p.j2} 
		\begin{aligned}[b]
			\j_2&= \underbrace{ \sigma\int_\Gamma
                          \p^6_t\vz\cp{\beta}\cdot\nz                            \sqrt{\gz}\gz^{\alpha\beta}
                          \p^7_t\vz \cdot \nz \cp{\alpha}
}_{\j_{2a}} 
+
\underbrace{\sigma
\int_\Gamma         \p^6_t\vz\cp{\beta}\cdot\nz \cp{\alpha}                            \sqrt{\gz}\gz^{\alpha\beta}
                         \p^7_t\vz \cdot \nz 
                       }_{ \j_{2b} }.
		\end{aligned} 
	\end{align}
The action of $\p^6_t$ in the tangential identity \eqref{z.tid} yields
        \begin{subequations}
            \label{z.tid.7}
          \begin{align}
\p^7_t\vz \cdot \ez
            \cp{\gamma}=\rho_0^{-1}\Jz
            \bset{\hd_{\gamma}\p^6_t\big(\sigma\gz^{\mu\nu}\ez\cp{\mu\nu}\cdot\nz
              - \p^6_t\beta_ \sigma(t) \big) - \elz_{\gamma}},
          \end{align}
          where $\elz_\gamma $ is such that $\sqrt{\sigma} \elz_
          \gamma$ is in  $H^{0.5}(\Gamma)$ and is given by
          \begin{align}
            \label{z.tid.7.l} \elz_\gamma =
            \sum_{\counter=0}^5c_\counter\p^\counter_t\vz_t\cdot\p^{6-\counter}_t\pset{\ez\cp{\gamma}\rho_0
              \Jz^{-1}} .
          \end{align}
        \end{subequations} 
		Letting
                $\ellz_\gamma^\beta=-\rho_0^{-1}\Jz\sqrt{\gz}\gz^{\alpha\beta}\gz^{\gamma\delta}\ez
                \cp{\delta\alpha}\cdot\nz$ and using the tangential
                identity \eqref{z.tid.7},
	we have that 
	\begin{align*}
		\j_{2a} &
		= \underbrace{- \sigma^2\int_\Gamma                   \bset{\p^6_t\vz\cp{\beta}\cdot\nz\ellz_\gamma^\beta }\cp{\gamma}
\p^6_t
                  \bset{\gz^{\mu\nu}\ez\cp{\mu\nu}\cdot\nz}}_{ \j_{2a}{}'}  +\RS.
	\end{align*}
 We integrate by parts with respect to a time-derivative of $\p^6_t\vz\cp{\beta}$ in order to write 
	\begin{align*}
		- \int_0^T &\j_{2a}{}' =-\sigma^2\int_\Gamma
                \bset{\ellz_\gamma^\beta\,\p^5_t\vz\cp{\beta}\cdot\nz}\cp{\gamma}\,
                \p^6_t\bset{\gz^{\mu\nu}\ez\cp{\mu\nu}\cdot\nz}
                \Big|_0^T + \int_0^T \sigma^2\int_\Gamma \bset{\ellz_\gamma^\beta\,
                \p^5_t\vz\cp{\beta}\cdot\nz_t} \cp{\gamma}\, \p^6_t\bset{\gz^{\mu\nu}\ez\cp{\mu\nu}\cdot\nz} \\
		&+ \int_0^T \sigma^2\int_\Gamma \p_t\bset{\ellz_\gamma^\beta\,\p^5_t\vz\cp{\beta}\cdot\nz}\cp{\gamma}\, \p^6_t\bset{\gz^{\mu\nu}\ez\cp{\mu\nu}\cdot\nz} + \int_0^T \sigma^2\int_\Gamma \bset{\ellz_\gamma^\beta\,\p^5_t\vz\cp{\beta}\cdot\nz}\cp{\gamma} \,\p^7_t\bset{\gz^{\mu\nu}\ez\cp{\mu\nu}\cdot\nz}\\
		=& \int_0^T\underbrace{ \sigma^2\int_\Gamma \bset{\ellz_\gamma^\beta\,\p^5_t\vz\cp{\beta}\cdot\nz}\cp{\gamma} \,\gz^{\mu\nu}\p^6_t\vz\cp{\mu\nu}\cdot\nz }_{\j_{2a}{}''}+ \int_0^T\RS. 
	\end{align*}
We write
\begin{align*}
\int_{0}^{T}\j_{2a}{}''=
  \int_0^T\underbrace{ \sigma^2\int_\Gamma
    \ellz_\gamma^\beta\,\p^5_t\vz\cp{\beta\gamma}\cdot\nz\,\gz^{\mu\nu}\p^6_t\vz\cp{\mu\nu}\cdot\nz
  }_{\j_{2a}{}'''}
+
  \int_0^T\underbrace{ \sigma^2\int_\Gamma \p^5_t\vz\cp{\beta}\cdot
    \bset{\nz\ellz_\gamma^\beta}\cp{\gamma}
    \,\gz^{\mu\nu}\p^6_t\vz\cp{\mu\nu}\cdot\nz }_{\j}.
\end{align*}
Using integration by parts with respect to a time-derivative of $\p^6_t\vz
\cp{\mu\nu}$, we conclude that $\int_{0}^{T}\j=\int_{0}^{T}\RS$.
	Letting $\ellz=\rho_0^{-1}\Jz\sqrt{\gz}\,\gz^{\gamma\delta}\ez\cp{\gamma\delta}\cdot\nz$ we utilize the symmetry of $\ellz_\gamma^\beta$ to exchange $\hd_\alpha$ and $\hd_\gamma$ via integration by parts for 
	\begin{align}
		\label{z.boundary: symmetry} 
		\begin{aligned}[b]
			\int_0^T&\j_{2a }{}'''= \frac{1}{2}\int_\Gamma \ellz\,\abs{\gz^{\mu\nu}\sigma\p^5_t\vz\cp{\mu\nu}\cdot\nz}^2\Big|_0^T \\
			& - \int_0^T \sigma^2\int_\Gamma \ellz\,\p^5_t\vz\cp{\alpha\beta}\cdot\nz\, \p^5_t\vz\cp{\mu\nu}\cdot\pset{ n\gz^{\mu\nu}}_t - \frac{1}{2}\int_0^T\int_\Gamma \ellz_t\,\abs{\gz^{\mu\nu}\sigma\p^5_t\vz\cp{\mu\nu}\cdot\nz}^2=\int_0^T\RS. 
		\end{aligned}
	\end{align}
	We have thus established that 
	\begin{align*}
		\int_0^T \j_{2a}=\int_0^T\RS. 
	\end{align*}
	
	The analysis of $\int_0^T\j_{2b}$ is similar. We set $\ellz^{\gamma\beta}=\sqrt{\gz}\gz^{\alpha\beta}\gz^{\gamma\delta}\ez \cp{\delta\alpha}\cdot\nz$ and write 
	\begin{align*}
		\int_0^T\j_{2b} = -\sigma^{\frac{3}{2}} \int_\Gamma
\p^6_t\vz\cdot\Bbset{\nz\cp{\alpha}\,\sqrt{\gz}\gz^{\alpha\beta\,}
                \sqrt{\sigma}\p^6_t\vz\cdot\nz} \cp{\beta}\Big|_0^T &- \int_0^T\sigma^2 \int_\Gamma\p^6_t\vz^j\cp{\beta}\bset{\nz^j\cp{\alpha}\sqrt{\gz}\gz^{\alpha\beta} \nz^i }_t\, \p^6_t\vz^i \\
		+ \int_0^T \underbrace{ \sigma^2\int_\Gamma \ellz^{\gamma\beta}\,\p^7_t\vz\cp{\beta}\cdot\ez \cp{\gamma}\,\p^6_t\vz\cdot\nz}_{\j_{2b} {}'} &= \int_0^T\RS+\int_0^T \j_{2b} {}'. 
	\end{align*}
	Regarding $\int_0^T \j_{2b} {}'$, we use the identity
	\begin{align}
		\label{z.tid.71} 
			\p^7_t\vz\cp{\beta}\cdot\ez \cp{\gamma} =
                        \rho_0^{-1}\Jz\bset{\hd_{\beta\gamma} \p^6_t
                          \big( \gz^{\mu\nu}\ez\cp{\mu\nu}\cdot\nz
                          -\beta_ \sigma(t) \big) -\elz'_{\beta\gamma}}, 
	\end{align}
	where
        $\elz'_{\beta\gamma}=\p^7_t\vz\cdot\ez\cp{\gamma\beta}\rho_0
        \Jz^{-1}+ \p^7_t\vz\cdot\ez\cp{\gamma}\pset{\rho_0
          \Jz^{-1}}\cp{\beta}+\hd_\beta \elz_\gamma$, with
        $\elz_\gamma$ given by \eqref{z.tid.7.l}. We integrate by parts with respect to $\hd_{\beta\gamma}$ in $\sigma^3\int_0^T \int_\Gamma\rho_0^{-1}\Jz
        \hd_{\beta\gamma} \p^6_t[
        \gz^{\mu\nu}\ez\cp{\mu\nu}\cdot\nz]\,
        \ellz^{\gamma\beta}\,\p^6_t\vz\cdot\nz$, where the
        highest-order term produced by $\hd_{\beta\gamma}$-integration
        by parts is (\ref{z.boundary: symmetry}). To estimate the
        integral where $\hd_\beta \elz_\gamma$ appears, we have the
        choice of integration by parts with respect to $\hd_\beta$ or
        an $H^{-0.5}(\Gamma)$-duality pairing. In the integral where
        $\p^7_t\vz\cdot\ez\cp{\gamma}\pset{\rho_0
          \Jz^{-1}}\cp{\beta}$ appears, we use the identity
\eqref{z.tid.7}.  Thus, we conclude that $
        \int_0^T\j_{2b} {}'= \int_0^T\RS$ and
	\begin{align}
		\label{z8.j2} \int_0^T \j_{2}\, \le \int_{0}^{T}\RS. 
	\end{align}

        \subsection*{Analysis of $\int_{0}^{T}\j_{3}$ in the time-integral
          of \eqref{z7.i}}
We equivalently write $\j_3$   as
        \begin{align*}
          \j_3= \underbrace{ \sigma\int_\Gamma  \p^6_t\vz\cp{\beta}\cdot \nz(
          \sqrt{\gz}\gz^{\alpha\beta})\cp{\alpha} \p^7_t\vz \cdot \nz }_{\j_{3a}}- \sum_{\counter=0}^{6}c_\counter\underbrace{\sigma\int_\Gamma \sqrt{\gz}\gz^{\alpha\beta} \p^\counter_t \ez ^j\cp{\alpha\beta} \p^{7-\counter}_t\nz^j\,\p^7_t\vz\cdot \nz }_{\j_{3b,\counter}}.
        \end{align*}
We infer from our analysis of \eqref{h.o. counter=0 term} that 
	\begin{align*}
			-\int_0^T\j_{3b,0} &= \int_0^T \underbrace{ \sigma\int_\Gamma \sqrt{\gz}\,\gz^{\alpha\beta} \gz^{\gamma\delta}\ez \cp{\alpha\beta}\cdot \ez \cp{\gamma} \p^6_t\vz\cp{\delta}\cdot\nz\,\p^7_t\vz\cdot\nz }_{-\j_{3a}} + \int_0^T\RS.
	\end{align*}
Hence,
	\begin{align*}
          \j_3= -\sum_{\counter=1}^{6}c_\counter\, \j_{3b,\counter}
+\RS. 
	\end{align*}
		The  terms $\int_{0}^{T}\j_{3b,\counter}$ for $\counter=1,\dots,5$,  are analyzed by integrating by
        parts with respect to a time-derivative of $\p^7_t\vz$ and
        then using elementary estimates.
Thus,
\begin{align*}
  \j_3= -\j_{3b,6}+\RS.
\end{align*}
Integration by parts with respect to a time-derivative of $\p^7_t\vz $
 yields
	\begin{align*}
\int_{0}^{T}\j_{3b,6} &=- \int_{0}^{T}\sigma\int_\Gamma
\sqrt{\gz}\gz^{\alpha\beta}
\p^6_t\vz\cp{\alpha\beta}\cdot\nz_t\,\p^6_t\vz\cdot\nz +
\int_{0}^{T}\RS \\
&= \int_{0}^{T}\underbrace{\sigma\int_\Gamma \sqrt{\gz}\gz^{\alpha\beta} \p^6_t\vz\cp{\beta}\cdot\nz_t\,\p^6_t\vz\cp{\alpha}\cdot\nz }_{\j_{3b,6}{}'} +\int_{0}^{T} \RS.
	\end{align*}
Letting
$\ellz^\gamma_{\alpha\beta}=\sqrt{\gz}\gz^{\alpha\beta}\gz^{\gamma\delta}\vz\cp{\delta}\cdot\nz$,
we once again integrate by parts with respect to time:
	\begin{align*}
&\int_{0}^{T}			\j_{3b,6}{}' =- \int_\Gamma
\p^6_t\vz\cp{\beta}\cdot\Bbset{
\ez \cp{\gamma}\ellz^\gamma_{\alpha\beta}
\sigma\p^5_t\vz\cp{\alpha}\cdot\nz} \cp{\beta}\Big|_0^T
\\
 &-
\int_0^T \sigma\int_\Gamma \p^6_t\vz ^j\cp{\beta}\p^5_t\vz^i \cp{\alpha}\cdot \pset{\nz^i\, \ez ^j\cp{\beta}\, \ellz^\gamma_{\alpha\beta} }_t 
			- \int_0^T \underbrace{\sigma \int_\Gamma \ellz^\gamma_{\alpha\beta} \,\p^7_t\vz\cp{\beta}\cdot \ez \cp{\gamma}\p^5_t\vz\cp{\alpha}\cdot\nz }_{\j} = - \int_0^T\j + \int_0^T\RS. 
	\end{align*}
We conclude 	via  the tangential identity \eqref{z.tid.71} that
        $\j=\RS$. 
	Hence, 
	\begin{align}
		\label{z8.j3} \int_0^T \j_{3}\, \le \int_{0}^{T}\RS. 
	\end{align}
	
	\subsection*{Analysis of $\int_0^T\j_4$  in the time-integral
          of \eqref{z7.i}} We integrate by parts with respect to a
        time-derivative of $\p^7_t\vz$  in the
        $\int_0^T\j_4$-terms and, if need be, spatially integrate by parts. For example, letting $\j_{4}=\sum_{\counter=1}^7 \j_{4,\counter}$, we find that after integration by parts with respect to time, 
	\begin{align*}
		\int_0^T\j_{4,1}&= \int_0^T\sigma\int_\Gamma
                \pset{\sqrt{\gz} \gz^{\alpha\beta}}_t \p^7_t\pset{\ez
                  \cp{\alpha\beta}\cdot\nz }  \,\p^6_t\vz \cdot\nz+\int_0^T\RS \\
		&= - \int_0^T \sigma\int_\Gamma \p^6_t\vz
                \cp{\alpha}\cdot\hd_\beta[\nz \,\pset{\sqrt{\gz}
                  \gz^{\alpha\beta}}_t \p^6_t\vz \cdot\nz]+\int_0^T\RS =\int_0^T\RS. 
	\end{align*}	
	Similarly, integration by parts with respect to time provides for the expression 
	\begin{align*}
		\sum_{\counter=2}^6\int_0^T\j_{4,\counter}=\int_0^T\RS. 
	\end{align*}
	Finally, using the differentiation formulas (\ref{formula: derivative: metric inverse}) and (\ref{formula: derivative: Jacobian determinant}), 
	\begin{align*}
		\int_0^T \j_{4,7}= \int_{0}^{T}\int_\Gamma \sqrt{\gz} (2 \gz^{\alpha\mu} \gz^{\nu\beta}-\gz^{\alpha\beta} \gz^{\mu\nu})\,\p^6_t\vz\cp{\mu}\cdot \ez \cp{\nu}\,\ez \cp{\alpha\beta}\cdot\nz\, \p^7_t\vz\cdot\nz + \int_0^T\RS = \int_0^T\RS, 
	\end{align*}
	where the second equality follows from our above analysis of $\int_0^T\j_{2b}$.
	
	Hence, 
	\begin{align}
		\label{z8.j4} \int_0^T \j_4\, \le \int_{0}^{T}\RS. 
	\end{align}
	
	\subsection*{Rewriting  equation \eqref{z8}} 
The inequalities \eqref{z8.j1}, \eqref{z8.j2},
        \eqref{z8.j3}, \eqref{z8.j4} provide that the boundary
        integral $\II$ expressed as \eqref{z7.i} satisfies
        \begin{align}
          \label{z8.i.0}
          \II = 	\frac{1}{2}\frac{d}{dt}\int_\Gamma
                        \sqrt{\gz} \gz^{\alpha\beta} \,\sqrt{\sigma}
                        \p^6_t\vz\cp{\alpha}\cdot \nz\,\sqrt{\sigma}
                        \p^6_t\vz\cp{\beta}\cdot \nz  
+\RS.
        \end{align}
Using
        the identities \eqref{z8.I} and \eqref{z8.i.0}, we have that
        \eqref{z8} is equivalently written as
        \begin{align}
          \label{z8.2}
\frac{d}{dt}  \int_\Omega \rho_0^2\Jz ^{-3}  \abs{\p^7_t\Jz}^2 
+
           2 \frac{d}{dt}\int_{\Omega} \rho_0^{-1}\fz^4\abs{\Az_\cdot^k \p^6_t\Jz
             \cp{k}}^2
 + 
\frac{1}{2}\frac{d}{dt}\int_\Gamma
                        \sqrt{\gz} \gz^{\alpha\beta} \,\sqrt{\sigma}
                        \p^6_t\vz\cp{\alpha}\cdot \nz\,\sqrt{\sigma}
                        \p^6_t\vz\cp{\beta}\cdot \nz  =\RS.
        \end{align}
The time-integral of \eqref{z8.2} completes the proof.
\end{proof}
\subsection*{Step 3: The  energy estimates for the
  action of $\hd^a\p^{8-2a}_t$, $a=1,2,3,4$} 
We define the vector $\az_\alpha^k$  as $ \ez\cp{\alpha}\cdot\az_\cdot^k$
and let the vector $\az_{\nz}^k$ be defined as
$\nz\cdot\az_\cdot^k$. Using these vector identities, we decompose the cofactor
matrix $\az$ as
\begin{align}
  \label{at.dcom}
  \az_r^k =\az_ \alpha^k  \gz^{\alpha\beta} \ez^r \cp{\beta}+
  \az_{\nz}^k\nz^r\ \ \ \text{on } \Gamma.
\end{align}
Since $ \az_\cdot^k N^k=\sqrt{\gz}\nz$ by the formula \eqref{formula: outward normal},
it follows that
\begin{align}
\label{anz}
  \az_{\nz}^\cdot = \frac{1}{\sqrt{\gz}} N^j\az_\ell^j\az_\ell^\cdot.
\end{align} 
According to  the identity \eqref{anz},
  the lower-bound \eqref{Taylor.t} is equivalently
stated as
\begin{align}
  \label{Taylor.2}
  0< \frac{\nu}{2}\le-  \az_{\nz} ^k (\rho_0^2 \Jz ^{-2})\cp{k}.
\end{align}

\begin{prop}
	[Energy estimates for the action of $ \hd\p^6_t$ in
        the Euler equations \eqref{stp.m}]\label{z.prop.z6}\begin{align*}
			\Sup \norm{\hd\p^6_t\vz (t)}_0^2
&+
\Sup
\abs{\sqrt{\sigma} \hd\p^5_t\vz\cdot \nz(t)}_1^2
+
 \Sup
\abs{ \hd\p^5_t\vz\cdot \nz(t)}_0^2\le \int_{0}^{T}\RS .
	\end{align*}
\end{prop}
 
\begin{proof}
Testing the action of $\hd\p^6_t$ in the Euler equations
        \eqref{stp.m} against $\hd\p^6_t\vz$ in the $L^2(\Omega)$-inner product, and integrating
        by parts  in the integral $\int_\Omega \az_i^k \hd\p^6_t(\rho_0
        ^2\Jz^{-2})\cp{k} \hd\p^6_t\vz^i$  yields
	\begin{align}
\label{z6}
\begin{aligned}[b]
  {\int_\Omega \hd\p^6_t\bset{\rho_0\vz^i_t } \hd\p^6_t\vz^i} +
  \underbrace{\int_\Omega \hd\p^6_t \az_i^k\bset{\rho_0^2\Jz^{-2}
    }\cp{k}\hd\p^6_t\vz^i}_{\mathscr{K}} & - \underbrace{\int_\Omega
    \hd\p^6_t\bset{\rho_0^2\Jz^{-2} }\az_i^k\hd\p^6_t\vz^i\cp{k}}_{\I }
  \\
  & + \underbrace{\int_\Gamma \hd\p^6_t\bset{\rho_0^2\Jz^{-2}
    }\az_i^k\hd\p^6_t\vz^iN^k}_{\II}=\RS.
\end{aligned}
	\end{align}
We used the Cauchy-Schwarz inequality or an 
$L^4$--$L^4$--$L^2$ H\"{o}lder inequality to analyze the lower-order terms in \eqref{z6}.
        \subsection*{Analysis of $\mathcal{I}$ in \eqref{z6}}
We have that 
\begin{align*}
-  \mathcal{I}= \frac{d}{dt} \underbrace{
\int_{\Omega}\rho_0^2 \Jz ^{-3}
\abs{ \hd\p^6_t\Jz}^2 }_{\int_{0}^{T}\RS}
+
\underbrace{
 \int_{\Omega}\hd\p^6_t(\rho_0^2 \Jz ^{-2})
\hd\p^6_t\az_i^k \vz^i \cp{k}}_{\mathcal{I}_1 }+
\underbrace{
 \int_{\Omega}\hd\p^6_t(\rho_0^2 \Jz ^{-2})
\p_t\az_i^k \hd\p^5_t\vz^i \cp{k}}_{\mathcal{I}_2 }+ \RS.
\end{align*}
Integration by parts with respect to a time-derivative of
$\p^6_t(\rho_0^2 \Jz ^{-2})$ yields
\begin{align}
\label{z6.I1.IBP}
  \int_{0}^{T}\mathcal{I}_1= - \underbrace{
\int_{\Omega}\hd\Bbset{\hd\p^5_t(\rho_0^2 \Jz ^{-2})
 \vz^i \cp{k} }\p^6_t\az_i^k\big|_0^T}_{\mathcal{I}_{1A}} - \int_{0}^{T}
\underbrace{ \int_{\Omega}\hd\p^5_t(\rho_0^2 \Jz ^{-2})
\p_t\bpset{\hd\p^6_t\az_i^k \vz^i \cp{k}}}_{\mathcal{I}_{1B}}.
\end{align}
Since $ \p^5_t\Jz$ is in $H^2(\Omega)$ and the highest-order term of
$ \p^6_t\az_i^k$ scales like $  D\p^5_t\vz$, the term $\mathcal{I}_{1A}$ is
bounded by $\int_{0}^{T}\RS $ thanks to the fundamental theorem of
calculus. Since the analysis of $\mathcal{I}_{1B}$ is similar, we have
established that 
$
\mathcal{I}_1=  
 \RS.
$
Similarly, $\mathcal{I}_2= \RS$. Thus,
   \begin{align}
\label{z6.I}  \mathcal{I}= 
 \RS.
\end{align}

\subsection*{Analysis of $\mathscr{K}$ in \eqref{z6}}

From the differentiation formula \eqref{formula: a_t} we infer that
\begin{align*}
\mathscr{K} = \int_{\Omega} \Jz ^{-1}\bpset{\az_r^s\az_i^k 
  -\az_i^s\az_r^k}\hd\p^5_t\vz^r \cp{s}(\rho_0^2 \Jz
^{-2})\cp{k}\hd \p^6_t\vz^i+ \RS.
\end{align*}
Since
$ \az_r^s\hd\p^5_t\vz^r\cp{s}$ is equal to
$\hd\p^6_t\Jz$ plus lower-order
terms, we write  
 \begin{align*}
\begin{aligned}[b]
  \mathscr{K} &=  -\int_{\Omega} \Jz ^{-1} \az_r^k (\rho_0^2 \Jz
  ^{-2})\cp{k}
  \hd\p^5_t\vz^r \cp{s} \az_i^s\hd\p^6_t\vz^i+ \RS\\
  &= \underbrace{ \int_{\Omega} \Jz ^{-1}\az_r^k (\rho_0^2 \Jz
    ^{-2})\cp{k} \hd\p^5_t\vz^r \az_i^s\hd\p^6_t\vz^i\cp{s}}_{\mathscr{K}'} -
   \underbrace{ \int_{\Gamma} \Jz ^{-1} \az_r^k (\rho_0^2 \Jz
    ^{-2})\cp{k} \hd\p^5_t\vz^r \az_i^s\hd\p^6_t\vz^iN^s
}_{\k}+ \RS.
\end{aligned}
\end{align*}
Integrating by parts with respect to a time-derivative of
$\hd\p^6_t\vz ^i\cp{s}$ yields
\begin{align*}
&  \int_{0}^{T}\mathscr{K}' = \underbrace{ \int_{\Omega} 
\Jz ^{-1}\az_r^k (\rho_0^2 \Jz
    ^{-2})\cp{k} \hd\p^5_t\vz^r
    \az_i^s\hd\p^5_t\vz^i\cp{s}\big|_0^T}_{\int_{0}^{T}\RS}
\\
&-\int_{0}^{T}\underbrace{ \int_{\Omega}\Jz ^{-1}\az_r^k (\rho_0^2 \Jz
    ^{-2})\cp{k} \hd\p^5_t\vz^r
    \az_i^s\hd\p^5_t\vz^i\cp{s}}_{\RS}
-
\int_{0}^{T}
\underbrace{ 
\int_{\Omega}
\Jz ^{-1}\az_r^k (\rho_0^2 \Jz
    ^{-2})\cp{k} \hd\p^5_t\vz^r (\az_i^s)_t\hd\p^5_t\vz^i\cp{s}
 }_{\mathscr{K}''},
\end{align*}
where we have again used that $\az_i^s\hd\p^5_t\vz^i \cp{s}$ is equal to $\hd\p^6_t\Jz$ plus lower-order
terms. 
Lemma~\ref{lem.tech} provides that
\begin{align*}
  \int_{0}^{T}\mathscr{K}'' \le
C
  \int_{0}^{T}\norm{\hd\p^5_t\vz}_{0.5}\norm{\hd D\p^5_t\vz}_{H^{0.5}(\Omega)'}\le
  C\int_{0}^{T}\norm{\p^5_t\vz}_{1.5}^2 \le \int_{0}^{T}\RS.
\end{align*}
The decomposition \eqref{at.dcom} 
 provides that
\begin{align*}
-  \k = -\underbrace{ 
\int_{\Gamma}\sqrt{\gz} (\rho_0^2 \Jz
    ^{-2})\cp{\alpha} \hd\p^5_t\vz\cdot \ez \cp{\beta}\gz ^{\alpha\beta} \hd\p^6_t\vz\cdot \nz}_{\k_1}
 -
\underbrace{ 
 \int_{\Gamma}\sqrt{\gz}\Jz ^{-1} \az_{\nz}^k (\rho_0^2 \Jz
    ^{-2})\cp{k} \hd\p^5_t\vz\cdot \nz\, \hd\p^6_t\vz\cdot\nz}_{\k_2}.
\end{align*}
We have used the identities $\Jz ^{-1}\az=\Az$ and $\ez
\cp{\alpha}\cdot \Az_\cdot ^k= \delta_ \alpha^k$, where $\delta_
\alpha^k$ is the Kronecker delta, in writing $\k_1$. Thanks to the
Laplace-Young 
boundary condition \eqref{stp.bc}, we have that
\begin{align*}
  \k_1 = \underbrace{
\int_{\Gamma} \sqrt{\gz} (\beta_\sigma(t))\cp{\alpha}  \hd \p^5_t\vz \cdot \ez
  \cp{\beta} \gz ^{\alpha\beta} \hd \p^6_t\vz \cdot \nz
-
\int_{\Gamma} \sqrt{\gz} (\gz
  ^{\mu\nu}\ez\cp{\mu\nu}\cdot \nz)\cp{\alpha} \sqrt{\sigma}\hd \p^5_t\vz \cdot \ez
  \cp{\beta} \gz ^{\alpha\beta}   \sqrt{\sigma} \hd \p^6_t\vz \cdot \nz
}_{\RS}.
\end{align*}
We have used the identity \eqref{beta.sigma} to deduce that  $\hd
\beta_\sigma(t)$ scales like $\sigma$ in $L^\infty (\Gamma)$ in the
first term on the right-hand side and we have used an $L^4$--$L^4$--$L^2$
H\"{o}lder inequality in the second term on the
right-hand side.
We also have that
\begin{align*}
  -\k_2 = &\frac{1}{2}\frac{d}{dt}\int_{\Gamma} \sqrt{\gz}\Jz
  ^{-1}[-\az_{\nz}^k (\rho_0^2 \Jz ^{-2}) \cp{k} ]\abs{ \hd\p^5_t\vz \cdot
    \nz}^2 
\\
+
&\frac{1}{2}\underbrace{\int_{\Gamma} \p_t\bpset{\sqrt{\gz} \Jz ^{-1}\az_{\nz}^k (\rho_0^2 \Jz ^{-2})
\cp{k} } \abs{\hd \p^5_t\vz\cdot \nz}^2
}_{\RS}
+
\underbrace{\int_{\Gamma} \sqrt{\gz} \Jz ^{-1}\az_{\nz}^k (\rho_0^2 \Jz ^{-2})
\cp{k} \hd \p^5_t\vz\cdot \nz \hd \p^5_t\vz\cdot \nz_t}_{\RS} .
\end{align*} 
Hence, 
 	\begin{align}
		\label{z6.K}
                \begin{aligned}[b]
\mathscr{K}=\frac{1}{2}\frac{d}{dt}\int_{\Gamma} \sqrt{\gz}\Jz
  ^{-1}[-\az_{\nz}^k (\rho_0^2 \Jz ^{-2}) \cp{k} ]\abs{\hd\p^5_t\vz \cdot
    \nz}^2 +\RS.
          \end{aligned}
	\end{align}  
	\subsection*{Analysis of $\II$ in
          \eqref{z6}} Using the Laplace-Young boundary
        condition \eqref{stp.bc},  we find
        that 
	\begin{align}
		\label{z6.i.0} 
		\begin{aligned}[b]
			&{\II}
			=\frac{1}{2}\frac{d}{dt}\int_\Gamma
                        \sqrt{\gz} \gz^{\alpha\beta} \,\sqrt{\sigma}
                        \hd\p^5_t\vz\cp{\alpha}\cdot \nz\,\sqrt{\sigma}
                        \hd\p^5_t\vz\cp{\beta}\cdot \nz  
-\underbrace{ \sigma\int_{\Gamma}\sqrt{\gz} \p^5_t\vz
\cp{\alpha\beta}\cdot \hd(\nz\gz ^{\alpha\beta}) \hd\p^6_t \vz \cdot
\nz
}_{\j_1}
 \\
			& \underbrace{ -\sigma\int_\Gamma
                          \sqrt{\gz} \gz^{\alpha\beta}
                          \hd\p^5_t\vz\cp{\alpha}\cdot \nz_t
                          \,\hd\p^5_t\vz\cp{\beta}\cdot \nz 
}_{\j_2}
 - \frac{1}{2}\underbrace{
 \int_\Gamma \pset{\sqrt{\gz}\gz^{\alpha\beta}}_t\,\sqrt{\sigma} \hd\p^5_t\vz\cp{\alpha}\cdot \nz \,\sqrt{\sigma}\hd\p^5_t\vz\cp{\beta}\cdot \nz }_{\RS} \\
			&+ \underbrace{\int_\Gamma  \sigma
                          \hd\p^5_t\vz\cp{\beta}\cdot\nz                            \sqrt{\gz}\gz^{\alpha\beta}
                          \hd\p^6_t\vz \cdot \nz \cp{\alpha}
}_{\RS}
+
\underbrace{\sigma\int_\Gamma         \hd\p^5_t\vz\cp{\beta}\cdot\nz \cp{\alpha}                            \sqrt{\gz}\gz^{\alpha\beta}
                          \hd\p^6_t\vz \cdot \nz 
                       }_{ \j_3 }
\\
	&+ \underbrace{ \sigma\int_\Gamma  \hd\p^5_t\vz\cp{\beta}\cdot \nz(
          \sqrt{\gz}\gz^{\alpha\beta})\cp{\alpha} \hd\p^6_t\vz \cdot
          \nz
}_{\RS}
+
\underbrace{\sum_{\counter=0}^{5}c_\counter\int_\Gamma\gz^{\alpha\beta}
  \p^\counter_t \ez ^j\cp{\alpha\beta}
  \p^{6-\counter}_t\nz^j\,\hd\bpset{\sigma\sqrt{\gz}\hd\p^6_t\vz\cdot
    \nz }
}_{\j_4}\\
			&+\underbrace{
                          \sum_{\counter=1}^6c_\counter\int_\Gamma\p^\counter_t{\gz^{\alpha\beta}}\p^{6-\counter}_t\bset{\ez
                            \cp{\alpha\beta}\cdot \nz
                          }
                          \hd\bpset{\sigma\sqrt{\gz}\hd\p^6_t\vz\cdot\nz
                          }
}_{\j_5}
+
\underbrace{
                        \int_{\Gamma} \hd\p^6_t\beta_
                        \sigma(t)\hd\p^6_t\vz\cdot \nz}_{\RS}. 
		\end{aligned}
	\end{align}
We integrate by parts with respect to a time-derivative of
$\hd\p^6_t\vz\cdot \nz$ to write
\begin{align*}
&  \int_{0}^{T}\j_1 =
\\
& - \int_{\Gamma}\sqrt{\gz} \p^5_t\vz
\cp{\alpha}\cdot\Bbset{ \hd(\nz\gz ^{\alpha\beta}) \sigma \hd\p^5_t \vz \cdot
\nz} \cp{\beta}\big|_0^T
-
 \int_{0}^{T} \int_{\Gamma}\sqrt{\gz} \sqrt{\sigma}\p^5_t\vz
\cp{\alpha\beta}\cdot \hd(\nz\gz ^{\alpha\beta}) \sqrt{\sigma}\hd\p^5_t \vz \cdot
\nz_t
\\
&
+
 \int_{0}^{T} \int_{\Gamma}\p^5_t\vz 
\cp{\alpha}\cdot\Bbset{ \bset{\hd(\nz\gz ^{\alpha\beta}) \sqrt{\gz} }_t \sigma\hd\p^5_t \vz \cdot
\nz } \cp{\beta}
+
 \int_{0}^{T} \int_{\Gamma}\sqrt{\gz} \p^6_t\vz
\cp{\alpha}\cdot \Bbset{\hd(\nz\gz ^{\alpha\beta}) \sigma\hd\p^5_t \vz \cdot
\nz} \cp{\beta}.
\end{align*}
Since $\sigma\p^5_t\vz \cdot \nz
\in H^{2.5}(\Gamma)$, it follows that
\begin{align*}
  \int_{0}^{T}\j_1= \int_{0}^{T}\RS.
\end{align*}
Similarly, integration by parts with respect to a time-derivative of
$\hd\p^5_t\vz\cp{\beta} \cdot \nz$ yields
\begin{align*}
  \int_{0}^{T}\j_2= \int_{0}^{T}\RS.
\end{align*}
Since $\sqrt{\sigma}\p^5_t\vz \cdot \nz
\in H^2(\Gamma)$ and $\sqrt{\sigma}\p^4_t\vz\in H^{2.5}(\Gamma)$, for the $\counter=5$ term of $\j_4$, we have that
\begin{align*}
&  \int_{0}^{T} \int_\Gamma\gz^{\alpha\beta}
\p^4_t\vz ^j\cp{\alpha\beta}
  \p_t\nz^j\,\hd\bpset{\sigma\sqrt{\gz}\hd\p^6_t\vz\cdot
    \nz }=
\\
&\int_\Gamma\gz^{\alpha\beta}
\sqrt{\sigma}\p^4_t\vz ^j\cp{\alpha\beta}
  \p_t\nz^j\,\hd\bpset{ \sqrt{\sigma}\sqrt{\gz}\hd\p^5_t\vz\cdot
    \nz }\big|_0^T
-
\int_{0}^{T}\int_\Gamma\gz^{\alpha\beta}
\sqrt{\sigma}\p^4_t\vz ^j\cp{\alpha\beta}
  \p_t\nz^j\,\hd\bpset{\sqrt{\sigma}\sqrt{\gz}\hd\p^5_t\vz\cdot
    \nz _t}
\\
&-
\int_{0}^{T}\int_\Gamma
\p^4_t\vz ^j\cp{\alpha\beta}
(  \p_t\nz^j\gz^{\alpha\beta})_t\,\hd\bpset{\sigma\sqrt{\gz}\hd\p^5_t\vz\cdot
    \nz }
-
\int_{0}^{T}\int_\Gamma\gz^{\alpha\beta}
\p^5_t\vz ^j\cp{\alpha\beta}
  \p_t\nz^j\,\hd\bpset{\sigma\sqrt{\gz}\hd\p^5_t\vz\cdot
    \nz }
\\
&=\int_{0}^{T}\RS-
\int_{0}^{T}\int_\Gamma\gz^{\alpha\beta}
\p^5_t\vz ^j\cp{\alpha\beta}
  \p_t\nz^j\,\hd\bpset{\sigma\sqrt{\gz}\hd\p^5_t\vz\cdot
    \nz }= \int_{0}^{T}\RS
.
\end{align*}
The last equality follows from the analysis of
$\int_{0}^{T}\j_2$. Since the analysis of $\j_5$ and the
$\counter=0,\dots,4$ terms of $\j_4$ are similarly established, we
infer that \eqref{z6.i.0} is equally written as
	\begin{align}
		\label{z6.i} 
		\begin{aligned}[b]
			&{\II}
			=\frac{1}{2}\frac{d}{dt}\int_\Gamma
                        \sqrt{\gz} \gz^{\alpha\beta} \,\sqrt{\sigma}
                        \hd\p^5_t\vz\cp{\alpha}\cdot \nz\,\sqrt{\sigma}
                        \hd\p^5_t\vz\cp{\beta}\cdot \nz  +\RS.
		\end{aligned}
	\end{align}

        \subsection*{The time-integral of \eqref{z6}}

        Using
        the identities \eqref{z6.I}, \eqref{z6.K} and \eqref{z6.i} in
        the time-integral of the equation \eqref{z6} completes the proof.
\end{proof}

\begin{prop}
	[Energy estimates for the action of $ \hd^2\p^4_t$ in
        the Euler equations \eqref{stp.m}]\label{z.prop.z4}\begin{align*}
			\Sup \norm{\hd^2\p^4_t\vz (t)}_0^2
+
\Sup
\abs{\sqrt{\sigma} \hd^2\vz_{ttt}\cdot \nz(t)}_1^2
+
 \Sup
\abs{ \hd^2\vz_{ttt}\cdot \nz(t)}_0^2\le \int_{0}^{T}\RS .
	\end{align*}
\end{prop}
        \begin{proof}
          Testing the action of $\hd^2\p^4_t$ in the Euler equations
        \eqref{stp.m} against $\hd^2\p^4_t\vz$ in the $L^2(\Omega)$-inner product, and integrating
        by parts  in the integral $\int_\Omega \az_i^k \hd^2\p^4_t(\rho_0
        ^2\Jz^{-2})\cp{k} \hd^2\p^4_t\vz^i$  yields
	\begin{align*}
\begin{aligned}[b]
  {\int_\Omega \hd^2\p^4_t\bset{\rho_0\vz^i_t } \hd^2\p^4_t\vz^i} +
  \underbrace{\int_\Omega \hd^2\p^4_t \az_i^k\bset{\rho_0^2\Jz^{-2}
    }\cp{k}\hd^2\p^4_t\vz^i}_{\mathscr{K}} & - \underbrace{\int_\Omega
    \hd^2\p^4_t\bset{\rho_0^2\Jz^{-2} }\az_i^k\hd^2\p^4_t\vz^i\cp{k}}_{\I }
  \\
&   + \underbrace{\int_\Gamma \hd^2\p^4_t\bset{\rho_0^2\Jz^{-2}
    }\az_i^k\hd^2\p^4_t\vz^iN^k}_{\II}=\RS.
\end{aligned}
	\end{align*}
The analysis of $\mathscr{K}$ and $\II$ follows from the proof of
Proposition~\ref{z.prop.z6}. Hence,
	\begin{align}
\label{z4}
\begin{aligned}[b]
\frac{1}{2}\frac{d}{dt} \int_{\Omega}\rho_0 \abs{\hd^2\p^4_t\vz}^2  &+\frac{1}{2}\frac{d}{dt}\int_{\Gamma} \sqrt{\gz}\Jz
  ^{-1}[-\az_{\nz}^k (\rho_0^2 \Jz ^{-2}) \cp{k} ]\abs{\hd^2\vz_{ttt} \cdot
    \nz}^2
\\
			&+\frac{1}{2}\frac{d}{dt}\int_\Gamma
                        \sqrt{\gz} \gz^{\alpha\beta} \,\sqrt{\sigma}
                        \hd^2\vz_{ttt}\cp{\alpha}\cdot \nz\,\sqrt{\sigma}
                        \hd^2\vz _{ttt} \cp{\beta}\cdot \nz   =\RS+\mathcal{I}.
\end{aligned}
	\end{align}
We have that 
\begin{align*}
-  \mathcal{I}= \frac{d}{dt} \underbrace{
\int_{\Omega}\rho_0^2 \Jz ^{-3}
\abs{ \hd^2\p^4_t\Jz}^2 }_{\int_{0}^{T}\RS}
+
 \int_{\Omega}\hd^2\p^4_t(\rho_0^2 \Jz ^{-2})
\hd^2\p^4_t\az_i^k \vz^i \cp{k} 
+
 \int_{\Omega}\hd^2\p^4_t(\rho_0^2 \Jz ^{-2})
\p_t\az_i^k \hd^2\p^3_t\vz^i \cp{k} + \RS.
\end{align*}
Since $\hd^2\p^4_t\Jz$ is in $H^{0.5}(\Omega)$, we use
Lemma~\ref{lem.tech} to conclude that
   \begin{align}
\label{z4.I}  \mathcal{I}= 
 \RS.
\end{align}
Using \eqref{z4.I} in the time-integral of \eqref{z4} completes the proof.        \end{proof}
 We infer the following two
        propositions from the proof of Proposition~\ref{z.prop.z4}:
\begin{prop}
	[Energy estimates for the action of $ \hd^3\p^2_t$ in
        the Euler equations \eqref{stp.m}]\label{z.prop.z2}\begin{align*}
			\Sup \norm{\hd^3\vz_{tt} (t)}_0^2
+
\Sup
\abs{\sqrt{\sigma} \hd^3\vz_{t}\cdot \nz(t)}_1^2
+
 \Sup
\abs{ \hd^3\vz_{t}\cdot \nz(t)}_0^2\le \int_{0}^{T}\RS .
	\end{align*}
\end{prop}
 \begin{prop}
	[Energy estimates for the action of $ \hd^4$ in
        the Euler equations \eqref{stp.m}]\label{z.prop.z0}\begin{align*}
			\Sup \norm{\hd^4\vz (t)}_0^2 
+
\Sup
\abs{\sqrt{\sigma} \hd^4\ez\cdot \nz(t)}_1^2
+
 \Sup
\abs{ \hd^4\ez\cdot \nz(t)}_0^2\le \int_{0}^{T}\RS .
	\end{align*}
\end{prop}

        \subsection*{Step~4: The $\sigma$-independent higher-order
          estimates  via Proposition~\ref{prop:Hodge}}
        
        \begin{lem}
          [The $\sigma$-independent lower-order    estimates for $\p^a_t\Jz$,
          $a=0,\dots, 7$]\label{z.lem.J}
          \begin{align*}
            \Sup \sum_{a=0}^7
            \norm{\p^a_t\Jz(t)}_{3.5-\frac{1}{2}a}^2\le \int_{0}^{T}\RS.
          \end{align*}
        \end{lem}
        \begin{proof} The $a=6$ and $a=7$ cases are provided by
          Proposition~\ref{z.prop.z8}. The higher-order estimates for
          $a=0,\dots,5$, are established by interpolation and the
          fundamental theorem of calculus. For example, using that
          $\p^6_t\Jz$ is in $H^1(\Omega)$
\begin{align*}
  \Sup \norm{\p^5_t\Jz(t)}_{1.5}^2 \le C_\delta   \Sup
  \norm{\p^5_t\Jz(t)}_{1}^2 + \delta  \Sup \norm{\p^5_t\Jz(t)}_{2}^2
  \le \int_{0}^{T}\RS.
\end{align*}
        \end{proof}
        \begin{lem}
          [The $\sigma$-independent normal trace-estimates for
          $\p^{a}_t\ez$, $a=0,\dots,7$]\label{z.lem.nt}
          \begin{align*}
            \Sup \sum_{a=0}^7
            \abs{\p^{a}_t\ez(t)\cdot N}_{4-\frac{1}{2}a}^2 \le \int_{0}^{T}\RS.
          \end{align*}
        \end{lem}
        \begin{proof}

          By Lemma~\ref{lem: normal trace},
  \begin{align*}
    \Sup \abs{\hd\p^6_t\vz(t)\cdot N}_{-0.5}^2 \le
    C\Sup\Bpset{\norm{\hd\p^6_t\vz (t)}_0^2 + \norm{ \Div  \p^6_t\vz (t)}_0^2}\le\int_{0}^{T}\RS,
  \end{align*}
thanks to the estimates stated in 
Proposition~\ref{z.prop.z6} and  Lemma~\ref{z.lem.J}. 
Using the same argument,   the $L^2(\Omega)$-estimates for $\hd^3\vz_{tt}$ and $\hd^4\vz$
respectively given in Propositions~\ref{z.prop.z2} and 
\ref{z.prop.z0}, and the divergence-estimates given by
Lemma~\ref{z.lem.J} provide    that
\begin{align*}
  \Sup \sum_{a=0}^3 \abs{\p^{2a}_t\vz(t)\cdot N}_{3.5-a}^2 \le
  \int_{0}^{T}\RS.
\end{align*}
Using the fundamental theorem of calculus, the normal trace-estimates
stated in Propositions~\ref{z.prop.z6}, \ref{z.prop.z4} and \ref{z.prop.z0} complete the proof.
        \end{proof}

Via Proposition~\ref{prop:Hodge}, the estimates stated in Lemmas~\ref{z.lem.curl},
\ref{z.lem.J} and \ref{z.lem.nt} establish
        \begin{prop}
          [The $\sigma$-independent estimates for $\p^a_t\ez$, $a=0, \dots,7$]\label{z.prop.ez}
          \begin{align*}
            \Sup \sum_{a=0}^7 \norm{\p^a_t\ez(t)}_{4.5- \frac{1}{2}a}^2 \le \int_{0}^{T}\RS.
          \end{align*}
        \end{prop}
                    Via Proposition~\ref{prop:Hodge}, the estimates stated in
          Lemma~\ref{z.lem.curl} and Proposition~\ref{z.prop.z6}
          establish  
        \begin{prop}
          [The $\sigma$-independent estimate for
          $\sqrt{\sigma}\p^5_t\vz$]\label{z.prop.s.v5}
          \begin{align*}
\norm{\sqrt{\sigma}\p^5_t\vz(t)}_{2}^2\le \int_{0}^{T}\RS.
          \end{align*}
        \end{prop}

        \subsection*{Step~5: The $\sigma$-independent higher-order
          estimates  via interpolation}

        \begin{prop}
          [The $\sigma$-independent estimates for $\sqrt{\sigma}\p^a_t\ez$, $a=0, \dots,3$]\label{z.prop.s.ez}
          \begin{align*}
            \Sup \sum_{a=0}^3\norm{\sqrt{\sigma}\p^a_t\ez(t)}_{5.5- \frac{1}{2}a}^2 \le \int_{0}^{T}\RS.
          \end{align*}
        \end{prop}
        \begin{proof}
We note that by use of interpolation and Young's inequality,
\begin{align*}
  \Sup \norm{\sqrt{\sigma}\ez(t)}_{5.5}^2 \le C_ \delta   \Sup
  \norm{\ez(t)}_{4.5}^2 + \delta  \Sup \norm{\sigma\ez(t)}_{5.5}^2.
\end{align*}
It follows that the   estimates given in Proposition~\ref{z.prop.ez} complete the proof.
        \end{proof}

        \subsection*{Step~6: The $\sigma$-independent  higher-order  
          estimates via the Euler equations \eqref{stp.m}}
        \begin{prop}
          [The $\sigma$-independent estimates for $\p^a_t\Jz$,
          $a=0,\dots, 5$]\label{z.prop.J}
          \begin{align*}
            \Sup\sum_{a=0}^5 \norm{\p^a_t\Jz(t)}_{4.5-\frac{1}{2}a}^2
\le \int_{0}^{T}\RS.
          \end{align*}
        \end{prop}
        \begin{proof}
We infer from           the $a$th time-derivative of the Euler equations
          \eqref{stp.m}   that
          \begin{align*}
            \Sup  \norm{\p^a_t\Jz(t)}_{4.5-\frac{1}{2}a}^2 \le
C\Sup           \norm{\p^a_t\Az(t)}_{3.5-\frac{1}{2}a}^2 + C \Sup
\norm{\p^a_t\vz_t(t)}_{3.5-\frac{1}{2}a}^2+ \int_{0}^{T}\RS.
          \end{align*}
The highest-order term of $\p^a_t\Az$ scales like
$\Az\Az D\p^{a}_t\ez$. Hence, the estimates stated in
Proposition~\ref{z.prop.ez} complete the proof.
        \end{proof}
 
Given
  Proposition~\ref{z.prop.J}, we now establish 
        \begin{prop}
          [The $\sigma$-independent estimates for
           $\sqrt{\sigma}\p^4_t\vz$                         and          $\sqrt{\sigma}\vz_{ttt}$   via Proposition~\ref{prop:Hodge}]\label{z.prop.s.v34}\begin{align*}
            \Sup\sum_{a=0}^1\norm{\sqrt{\sigma}\p^{3+a}_t\vz(t)}_{3.5-
            \frac{1}{2}a}^2\le \int_{0}^{T}\RS.
          \end{align*}
        \end{prop}
        \begin{proof}
                              Via Proposition~\ref{prop:Hodge}, the estimates stated in
          Lemma~\ref{z.lem.curl} and Propositions~\ref{z.prop.z4} and \ref{z.prop.J}
          establish  the estimate for $\sqrt{\sigma}\vz_{ttt}$. For
          $\sqrt{\sigma}\p^4_t\vz$, we
          have that
          \begin{align*}
            \abs{\sqrt{\sigma} \p^4_t\vz \cdot \nz(t)}_{2.5}^2\le C_
            \delta \abs{\p^4_t\vz \cdot \nz(t)}_{1.5}^2+
            \delta \abs{\sigma \p^4_t\vz \cdot \nz(t)}_{3.5}^2
            \le \int_{0}^{T}\RS,
          \end{align*}
thanks to the estimate for $\p^4_t\vz$ stated in
Proposition~\ref{z.prop.ez}. Hence, we infer via Proposition~\ref{prop:Hodge}  the estimate for
$\sqrt{\sigma}\p^4_t\vz$ 
from the estimates stated in Lemma~\ref{z.lem.curl} and Proposition~\ref{z.prop.J}.
        \end{proof}
        \subsection*{Step~7: The $\sigma$-independent higher-order
          estimates for $\sigma\p^a_t\ez$, $a=0,\dots,4$}
         \begin{lem}
          [The $\sigma$-independent normal trace-estimates for
          $\sigma\p^{a}_t\ez$, $a=0,\dots,4$]\label{z.lem.s.nt}
          \begin{align*}
            \Sup \sum_{a=0}^3
            \abs{\sigma \hd^2\p^{a}_t\ez\cdot
              \nz(t)}_{4-\frac{1}{2}a}^2 +\Sup
            \abs{\sigma\p^3_t\vz\cdot \nz(t)}_{3.5}^2\le \int_{0}^{T}\RS.
          \end{align*}
        \end{lem}
        \begin{proof}
The estimates for $\p^a_t\Jz$ given in
Proposition~\ref{z.prop.J} and the fundamental theorem of calculus
provide that the $a$th time-derivative of the Laplace-Young boundary
condition \eqref{stp.bc} yields the desired estimates.
       \end{proof}

        \begin{prop}
          [The $\sigma$-independent estimates for $ \sigma\p^a_t\ez$, $a=0, \dots,4$]\label{z.prop.ss.ez}
          \begin{align*}
            \Sup \sum_{a=0}^3 \norm{ \sigma\p^a_t\ez(t)}_{6.5-
              \frac{1}{2}a}^2
+
\Sup   \norm{ \sigma\p^{3}_t\vz(t)}_{4}^2\le \int_{0}^{T}\RS.
          \end{align*}
        \end{prop}
        \begin{proof}
Since $\sigma<\sqrt{\sigma}$,    
we infer from the proof of Proposition~\ref{z.prop.J} that  the
estimates for $\sqrt{\sigma}\p^a_t\vz$, $a=3,4,5$, stated in
Propositions~\ref{z.prop.s.v5} and \ref{z.prop.s.v34} establish the
divergence-estimates for $\sigma\vz _{t}$, $\sigma\vz _{tt}$ and
$\sigma\vz _{ttt}$. Similarly, the estimates for
$\sqrt{\sigma}\p^a_t\vz$, $a=1,2$, stated in
Proposition~\ref{z.prop.s.ez} establish the
divergence-estimates for $\sigma\vz $ and
$\sigma\ez $.
Via Proposition~\ref{prop:Hodge}, the estimates stated in
Lemmas~\ref{z.lem.curl} and \ref{z.lem.s.nt} therefore complete the proof.         \end{proof}
        \subsection*{Step~8: The $\sigma$-independent improved boundary-regularity
          estimates}
Considering  the action of $\p^a_t$, $a=4,5,6,$ in  the Laplace-Young
boundary condition \eqref{stp.bc}, we notice that the estimates stated
in       Propositions~\ref{z.prop.J}  and
       \ref{z.prop.ss.ez} establish
        \begin{prop}
          [The $\sigma$-independent estimates for
          $\sigma\p^a_t\vz\cdot\nz$, $a=3,4,5$]\label{prop.z.last} 
          \begin{align*}
            \Sup\sum_{a=0}^1 \abs{\sigma\p^{3+a}_t\vz \cdot\nz(t)}_{4-\frac{1}{2}a}^2
+
            \Sup \abs{\sigma\p^5_t\vz \cdot\nz(t)}_{2.5}^2 \le \int_{0}^{T}\RS.
          \end{align*}
        \end{prop}

\subsection*{Step~9: Concluding the proof of Lemma~\ref{lem.ES}} \label{sec:asymptotic-analysis.z}  The sum of
the 
estimates given in Propositions~\ref{z.prop.z8}--\ref{prop.z.last} completes the proof of  Lemma~\ref{lem.ES}. Taking $\delta$ sufficiently small in the inequality \eqref{est: ES}
yields a polynomial-type inequality of the form \eqref{poly-type
  ineq}. Hence, there
exists $T>0$ that is independent of $\sigma   >0$  and verifies
\begin{align}
	\label{sigma independent} \Sup\ES\le2\MS. 
\end{align}

\subsection{The proof of Theorem~\ref{thm.z.main}}
\label{sec:concl-proof-theor-1}

\subsubsection{Existence}
\label{sec:exist-solut-surf-1.z}
Assuming that the  initial data   $(\rho_0,u_0,
\Omega)$ is of $C^\infty$-class, as in Section~\ref{z.sec:assum-cinfty-class-1},
we have that  $M_0=P(E(0))<
\infty $ where the
higher-order energy function    $E(t)$ is defined by \eqref{s.energy function}. Since the difference of the compatibility conditions
\eqref{s.compatibility conditions} and \eqref{stp: compat cond} is in
$C^\infty(\Gamma)$, we repeat the proof of
Theorem~\ref{thm.s.main} to obtain the existence of
a solution $\vz$ to the $\sigma$-problem~\eqref{stp}.
The $\sigma$-independent estimate (\ref{sigma independent}) and
 standard compactness arguments establish  the
 strong convergence, as $\sigma$ tends to zero,
\begin{alignat*}{2}
  \ez\to\eta&\ &\text{in $L^2\pset{0,T;H^{3.5}(\Omega)}$},
\\
\vz_{t}\to v_t&&\text{in $L^2\pset{0,T;H^{2.5}(\Omega)}$}.
\end{alignat*}
Letting
$\phi\in L^2\pset{0,T;H^1(\Omega)}$, we have that the variational form
of \eqref{stp} is 
\begin{align*}
  \int_0^T\int_\Omega\rho_0\vz _t\cdot\phi
-
 \int_0^T\int_\Omega \rho_0^2\Jz^{-2}\az_i^k\phi^i\cp{k} 
&+
\int_0^T\int_\Gamma \beta_\sigma(t)
 \sqrt{\gz}\phi\cdot \nz
=
\sigma\int_0^T\int_\Gamma
\sqrt{\gz}\gz^{\mu\nu}\ez\cp{\mu\nu}\cdot\nz\,\phi\cdot \nz.
\end{align*}
The strong convergence of the sequences
$\pset{\ez,\vz_{t} }$ and the pointwise convergence $\beta_\sigma(t)\to \beta$
provide that the limit $\pset{\eta,v_t,\beta}$ satisfies
\begin{align*}
  \int_0^T\int_\Omega \rho_0 v_t\cdot\phi
-
\int_0^T\int_\Omega \rho_0^2J^{-2}a_i^k\phi^i\cp{k}
+
\int_0^T\int_\Gamma\beta \sqrt{g}\phi\cdot n=0. 
\end{align*}
Thus, $v$ is a solution of the zero surface tension limit of \eqref{Euler} on a nonempty
time-interval $[0,T]$. Standard arguments provide that
$v(0)=u_0$ and $\eta(0)=e$. Furthermore, 
letting $\sigma=0$ in the a priori estimates
in Section~\ref{sec:sigma independent estimates}, we conclude that the right-hand side of  inequality \eqref{sigma independent}
depends only on  $\MZ=P(\mathscr{E}(0))$. Hence, 
  for sufficiently small $T>0$ the higher-order energy function
  $\EZ$ defined in \eqref{z.energy function} satisfies
\begin{align}
\label{z.go.a}
  \sup_{t\in[0,T]}\EZ\le2\MZ.
\end{align}
 The
bounds \eqref{z.assum} remain valid by taking $T>0$ even smaller if necessary. By the arguments in Section~\ref{sec:limit-as-kappato0}, the boundedness of $J$ in assumption \eqref{z.assum.Jt} implies that 
\begin{align*}
	\rho(t) \ge \lambda \ \ \text{in } \overline{\Omega}(t). 
\end{align*}
Similarly,   the lower-bound  \eqref{Taylor.t} provides that
$\rho(t)=f \circ \eta ^{-1}(t)$ satisfies
   \begin{align*}
    0< \nu\le-\frac{\p \rho^2(t)}{\p n(t)} \ \ \text{on
      } \Gamma(t).
  \end{align*}

\subsubsection{Optimal regularity for the initial
  data} \label{sec:optim-regul-init.z}In order to obtain the $H^{4.5}(\Omega)$-regularity of our
existence theory, we  assumed that the given
initial data is of  $C^\infty$-class in 
Section~\ref{z.sec:assum-cinfty-class-1}. 
By virtue  of the estimate \eqref{z.go.a}, it in fact
  suffices for the regularity of the initial data to be such that $\mathscr{E}(0)<\infty$.

\subsubsection{Uniqueness} \label{sec:exitence-uniqueness.z} 
We infer the uniqueness of the solution to the zero surface tension
limit of \eqref{Euler} by repeating the arguments given in Section~\ref{sec:exitence-uniqueness}.

\

\subsection*{Acknowledgements.} DC was supported by the Centre for Analysis and Nonlinear PDEs funded by the UK
EPSRC grant EP/E03635X and the Scottish Funding Council. SS  and JH were supported by the National
Science Foundation under grant DMS-1001850.

\

\appendix

\section{Constructing $C^\infty$-class initial data}
\label{sec:appendix:init}

We   demonstrate how given  initial data $(\rho_0, u_0, \Omega)$ of
optimal regularity satisfying the conditions
\eqref{l.boundedness.r0}  and
\eqref{s.compatibility conditions} (or the conditions
\eqref{l.boundedness.r0}, \eqref{Taylor} and \eqref{z.compatibility conditions})
we are able to produce
asymptotically consistent $C^\infty$-class data $(\vre,\bv_0,
\Ore)$ which satisfy similar statements of the conditions
\eqref{l.boundedness.r0} and \eqref{s.compatibility conditions} (or the conditions
\eqref{l.boundedness.r0}, \eqref{Taylor} and \eqref{z.compatibility
  conditions}).  

\subsection{The $C^\infty$-class data for the surface tension problem \eqref{Euler}}
\label{sec:cinfty-class-data}
We suppose the  initial data $(\rho_0, u_0, \Omega)$ is of
the optimal regularity stated in Theorem~\ref{thm.s.main}  and satisfy the conditions
\eqref{l.boundedness.r0} and
\eqref{s.compatibility conditions}. Letting
\begin{align}
\label{a.st.Ja.Ha}
  \Jre_a = \p^a_t(J ^{-2})|_{t=0}\ \ \ \text{and}\ \ \  \Hre_a= \p^a_t
  H(\eta)|_{t=0},
\end{align}
we will construct $C^\infty$-class
data $(\vre,\bv_0,
\Ore)$ which satisfy 
\begin{subequations}
\label{a.st}
  \begin{alignat}{2}
    \label{a.st.H}
 \sigma\Hre_0 & \ge 4 \lmr^2- \beta&\ \ &\text{for } \sigma, \beta>0,\\
    \label{a.st.r}
    \vre&\ge 2\lmr>0 &\ \ &\text{in }\overline{ \Ore},\\
    \label{a.st.J}
    \vre^2 \Jre_a& = \p^a_t\beta + \sigma \Hre_a &\ \ \ &\text{on }
    \Gre\ \ \text{for } a=0,1,2,3.
  \end{alignat}
\end{subequations}
\subsubsection{Defining the $C^\infty$-class initial domain $\Ore$}
\label{sec:defin-smooth-init}
For
$\epr>0$, we let $0\le \varphi^\epr\in C^\infty_0(\mathbb{R}^3)$
denote the standard family of mollifiers with
$\supp(\varphi^\epr)\subset\overline{B(0,\epr)}$. Recalling the local
coordinates near $\Gamma $   defined in
Section~\ref{sec:local-coord-near}, we let $\Ore\subset
 \mathbb{R}^3$ denote the $C^\infty$-class
domain defined by the $C^\infty$-class charts 
\begin{align*}
  \tre=\varphi^{\epr} *\bset{\mathcal{E}_{\mathcal{B}_l
      ^{\epr}}(\theta_l )} \ \ \text{in } \mathcal{B}_l\ \
  \text{for } 1\le \local \le L,
\end{align*}
where for $X\subset \mathbb{R}^3$ the operator $\mathcal{E}_{X}$ denotes a Sobolev extension
operator mapping $H^s(X)$ to $H^s (\mathbb{R}^3)$ and where
$\supp (\xi_l\circ \theta_l)\subset\mathcal{B}_l ^{\epr}\subset
\mathcal{B}_l$. 
We assume that  $\epr>0$ is taken sufficiently small so that the collection of open
set $\set{U_l}_{l=1}^L$ introduced in Section \ref{sec:local-coord-near}
 is an
open covering of $\Ore$.
  Setting $\Gre=\p \Ore$ defines a $C^\infty$-class surface.
We 
let $\ere$ denote the 
$C^\infty(\Gre)$-vector 
describing the geometry of $\Gre$. In other words,
$\hd\ere$ spans the tangent space of $\Gre$
and, by letting $\Nre$ 
denote the 
outward-pointing unit normal vector to the surface $\Gre$, 
we have that 
$\Nre= \ere\cp{1}\times \ere\cp{2}/\abs{\ere\cp{1}\times \ere\cp{2}}$. We let $\gre $ denote the
surface metric induced by $\ere$ and let $\Hre_0
$ denote twice the mean curvature of
$\Gre$. Thus,
\begin{align*}
  \Hre_0 =- \gre ^{\alpha\beta} \ere
  \cp{\alpha\beta} \cdot \Nre \text{ 
    in } C^ \infty(\Gre).
\end{align*}
We assume that the given $H^5$-class domain $\Omega \subset \mathbb{R}^3$ is such that
\eqref{s.compatibility conditions} is satisfied. Thus, for 
$\epr>0$ taken sufficiently small, standard properties of convolution provide that
  \begin{align}
  \label{boundedness.H.p}
\beta+  \sigma\Hre_0 & \ge 2\lambda^2\ \ \text{for } \sigma, \beta>0.
  \end{align}
Hence, \eqref{a.st.H} is satisfied.
\subsubsection{Defining the $C^\infty$-class data $(\rre,\bV_0, \Ore)$
to satsify \eqref{a.st}  for $a=0$}
\label{sec:defin-tang-smooth} 
  Given $u_0 \in H^4(\Omega)$, we define $\mathrm{u}_0$   in the $C^\infty$-class domain $\Ore$
  via the equation
\begin{align}
\label{u0.Ore}
\mathrm{u}_0= \sum_{\local=1}^L [\xl u_0  \circ \thetal ] \circ
 \tre^{-1}\ \ \ \ \text{in } H^4 (\Ore).
\end{align} 
Using the operator $\LE$ defined in
Section~\ref{sec:horiz-conv-layers}, for $\epr>0$
we  define the vector field   $\bV_0  \in 
C^\infty(\Ore)$  as the solution of the following elliptic Dirichlet problem:
\begin{subequations}
  \label{a.pV0i}
  \begin{alignat}{2}
    \bV_0 - \epr \Delta\bV_0&=\varphi^{\epr}*
    \mathcal{E}_{\Ore}(\mathrm{u}_0)&\ \ \  \ &\text{in }
    \Ore,\\
\bV_0&=\sum_{\local=1}^K \Lambda_{\epr}\bbset{\xl
  \mathrm{u}_0\circ\tre}\circ\tre ^{-1}&& \text{on } \Gre.
  \end{alignat}
\end{subequations}
  Given $\rho_0 \in H^4(\Omega)$, we define $\varrho_0$   in the $C^\infty$-class domain $\Ore$
  via the equation
\begin{align}
\label{r0.Ore}
\varrho_0=  \sum_{\local=1}^L [\xl  \rho_0  \circ \thetal ] \circ
 \tre^{-1} \ \ \ \text{in } H^4 (\Ore).
\end{align}  Assuming  that
$\rho_0\ge 2\lambda >0$ in $\overline{\Omega}$, we let  $\lmr>0$ be
such that $2\lambda^2\ge 9\lmr^2$. It follows for
$\varrho_0$ defined by \eqref{r0.Ore} that
\begin{align} 
\label{rr.0.O}
\varrho_0 \ge3 \lmr\ \ \text{in } \overline{\Ore}.
\end{align}
For $\mur>0$ and $\epr>0$, we define $\rre \in 
 C^\infty (\Ore)$ to be the solution of 
 \begin{subequations}
\label{a.pr}
   \begin{alignat}{2}
\label{a.pr.int}
     \rre- \mur \Delta \rre&= \varphi
     ^{\mur}*\mathcal{E}_{\Ore}(\varrho_0)&\ \ \ \
     &\text{in
     } \Ore,\\
\label{a.pr.b}
     \rre&=\sqrt{\beta+ \sigma\Hre_0}&&\text{on
     } \Gre.
    \end{alignat}
 \end{subequations}
The boundary condition
\eqref{a.pr.b} implies for $a=0$  that  \eqref{a.st.J} is satisfied, since \eqref{a.pr.b} is equivalently stated as
\begin{align}
  \label{Jre0}
  \rre ^2\Jre_0 = \beta+ \sigma\Hre_0\ \ \ \text{on }\Gre.
\end{align}
Elliptic regularity provides  for $s\ge 2$ and a positive
constant $C$ independent of $\mur$ that
\begin{align*}
  \norm{\sqrt{\mur} 
&  \rre}_{H^{s}(\Ore)}^2 \le
C\Bpset{ \norm{\varphi ^{\mur}* \mathcal{E}_{\Ore}(\varrho_0)}_{H^{s-2}(\Ore)}^2
+
 \norm{\rre
}_{H^{s-\frac{1}{2}} (\Gre)}^2
}
.
\end{align*}
The boundary forcing function  used in defining $\rre$ is in
$C^\infty(\Gre)$. Since the function $\varrho_0$ appearing in the
right-hand side of \eqref{a.pr.int} is in $H^4(\Ore)$,
there exists a constant $C$ which is independent of $\mur$ verifying 
\begin{align*}
 \norm{ \sqrt{\mur}
  \rre}_{H^{6}(\Ore)}^2
\le C.  
\end{align*}
The   equation \eqref{a.pr.int} provides that $\rre = \mur \Delta \rre+\varphi^{\mur} *
  \mathcal{E}_{\Ore} (\varrho_0) $ in $\Ore$. Hence, by standard
properties of convolution and taking $\mur$ sufficiently
small, we conclude  that
\begin{align*}
 \norm{ \rre-\varrho_0 }_{L^\infty(\Ore )}&\le
\sqrt{ \mur} \norm{\sqrt{\mur}\Delta\rre}_{H^2(\Ore)}+ \norm{\varphi
   ^{\mur}* \mathcal{E}_{\Ore}
   (\varrho_0)- \varrho_0}_{L^\infty(\Ore)}
\le \lmr.
\end{align*} 
Recalling the lower-bound \eqref{rr.0.O},
we conclude that 
$ 
  \rre\ge 2\lmr>0 $ in $\Ore.  
$
Letting $\epr$ be sufficiently small, the  boundary condition
\eqref{a.pr.b} provides that 
\begin{align}
\label{a.pr.bound}
  \rre\ge 2\lmr>0 \ \ \text{in }\overline{ \Ore}.  
\end{align} 
By \eqref{a.pr.bound} and \eqref{boundedness.H.p}, the boundary
condition \eqref{a.pr.b} provides that
  \begin{align}
  \label{boundedness.H}
\beta+  \sigma\Hre_0 & \ge 4\lmr^2\ \ \text{for } \sigma, \beta>0.
  \end{align}
Hence,  the $C^\infty$-class data $(\rre,\bV_0,\Ore)$ satisfy
the conditions
\eqref{a.st} for $a=0$, as verified by \eqref{Jre0},
\eqref{a.pr.bound} and \eqref{boundedness.H}.

\subsubsection{Formal definitions}
We formally set $\vre$ equal to $\rre$ defined by \eqref{a.pr}. We define
\begin{align}
  \label{a.st.va}
  \bv_{a+1}=-2\p^a_t\bpset{A_\cdot^k(\vre J ^{-1})}|_{t=0}\ \ \
  \text{for } a=0,\dots,5,
\end{align}
and assume that \eqref{a.st.va} yields $\bv_0=\bV_0$ defined by \eqref{a.pV0i}. We make the following definitions:
\begin{align}
  \label{a.st.ja.ka}
  \jre_a = \Div \bv_a\ \ \ \text{and}\ \ \ \kre_a=
  \p^a_t(A_r^sv^r \cp{s})|_{t=0}\ \ \ \text{for }a=0,\dots,6.
\end{align}
Using the notation
\begin{align}
  \label{a.st.Aa}
 \bset{ \Are_a}_r^s = \p^a_tA_r^s|_{t=0}\ \ \ \text{for } a=0,\dots,5,
\end{align}
it follows that $\kre_a$ defined in \eqref{a.st.ja.ka} satisfies
\begin{align}
\label{a.st.k1.k2.k3}
  \kre_0= \jre_0,\ \ \
  \kre_1= \bset{\Are_1}_r^s \bv_0^r \cp{s}+\jre_1
\ \ \text{and}\ \ 
  \kre_2 = \bset{\Are_2}_r^s \bv_0^r \cp{s}+2 \bset{\Are_1}_r^s \bv_1^r \cp{s}+\jre_2
.
\end{align}
We also have that $\Jre_a$ defined by \eqref{a.st.Ja.Ha} is
equivalently given by
\begin{align}
  \label{a.st.J0.J1.J2.J3}
  \Jre_0=1, \ \ \ \Jre_1 = -2\kre_0,\ \ \ \Jre_2= 4\kre_0^2 -2\kre_1\
  \ \text{and}\ \ \Jre_3= -8 \kre_0^3+12 \kre_0\kre_1 -2\kre_2.
\end{align}
Using \eqref{a.st.k1.k2.k3} and  \eqref{a.st.J0.J1.J2.J3},  the
condition \eqref{a.st.J} that $\vre^2\Jre_a= \sigma \Hre_a$ is equivalently
stated as
\begin{align}
  \label{eq.kre}
\Div\bv_{a-1}=\jre_{a-1} =  \underbrace{-\frac{1}{2}\bbset{\sigma\Hre_a -\vre^2\Jre
    -2\kre_{a-1}} +\dre_{a-1}}_{\fre_{a-1}},
\ \ \ \text{for }a=1,2,3,
\end{align} 
where $\dre_{a-1}$ represents the lower-order terms defining $\kre_{a-1}$.
We notice that $\bv_a$, $a=1,2$, defined by \eqref{a.st.va} is equivalently
given by
$\bv_1=-2 D \vre$ and 
$    \bv_2=-2 \bset{\Are_1}_\cdot^k \vre \cp{k} -2D(\vre \jre_0) $. Thus,
\begin{align}
\label{j.f.rel}
  \Div\bv_0  =\fre_0,\ \ \
  \Div\bv_1  =\fre_1= - 2 \Delta\vre,
\ \ \text{and}\ \
  \Div\bv_2  =\fre_2\approx -2 \Delta\Div \bv_0.
\end{align}

 \subsubsection{Defining $\bv_0$ so that the $C^\infty$-class data $(\rre,\bv_0, \Ore)$
  satisfy \eqref{a.st}  for $a=0,1$}\label{sec:a.v1} Extending $\Nre$ and $\Tre=\ere
\cp{\alpha}/\abs{\ere\cp{\alpha}}$, for $\alpha=1,2,$
into $\Ore$,  we decompose any vector $\xi\in \mathbb{R}^3 $  into normal and
tangential components as
\begin{align*}
  \xi= \xi  ^ \alpha\Tre+ \xi ^3 \Nre \ \
  \text{in } \Ore. 
\end{align*}
 It follows for sufficiently regular $\xi$ that $  \Div \xi = \xi ^i \cp{i}$ is
equivalently written as
\begin{align}
\label{div.DD}
  \Div \xi  
&=\underbrace{
\xi ^ \alpha\cp{\alpha} +  \xi ^3\cp{3} + \xi ^ \alpha\Div \Tre+ \xi ^3
  \Div \Nre}_{j_0}\ \  \ \text{on }\Gre.
\end{align}
This provides  the following identity for $(\Nre\cdot D) \xi^3 = \xi^3
\cp{3}$:
\begin{align}
\label{DN}
\frac{ \p\xi ^3}{\p\Nre}
&=  \underbrace{
j_0  -\bbset{\xi ^ \alpha\cp{\alpha} + \xi ^ \alpha\Div \Tre+ \xi ^3
  \Div \Nre }}_{\varphi_0}\ \  \ \text{on } \Gre.
\end{align}
We  define   $\bv_0^3
\in C^\infty(\Ore)$ by 
\begin{subequations}
  \label{bv1.3}
  \begin{alignat}{2}
    \bv_0^3 +\epr \Delta^2 \bv_0^3&=\varphi^{\epr}*
    \mathcal{E}_{\Ore}(\mathrm{u}_0^ 3)&\ \ \  \ &\text{in }
    \Ore,
\\
\label{bv1.3.bc2}
\frac{\p \bv_0^3}{\p\Nre}&=\phrep_0
&& \text{on } \Gre ,\\
\label{bv1.3.bc1}
\bv_0^3&=\bV_0^3&& \text{on } \Gre,
  \end{alignat}
\end{subequations}
where the boundary forcing function $\phrep_0$ is defined as
\begin{align}
  \label{phrep}
  \phrep_0=\fre_0-[\bV_0^ \alpha\cp{\alpha} + \bV_0^ \alpha\Div \Tre+
  \bV_0^3 \Div \Nre],
\end{align}
with the function $\fre_0$ appearing in the right-hand side of
\eqref{phrep}   defined by \eqref{eq.kre}. 
With $\bV_0^ \alpha$ denoting the tangential component of the vector
defined by \eqref{a.pV0i}, we define the vector $\bv_0\in C^\infty(\Omega)$ as
\begin{align}
  \label{bv0}
  \bv_0=\bV_0^\alpha\Tre +\bv_0^3 \Nre\ \ \ \text{in }\Ore.
\end{align}
The boundary condition \eqref{bv1.3.bc1} and the definition \eqref{bv0}
imply that
\begin{align}
  \label{bv0.bV0}
  \bv_0=\bV_0\ \ \ \text{on } \Gre.
\end{align}
Thanks to \eqref{div.DD}, the boundary condition
\eqref{bv1.3.bc2} and the definition \eqref{phrep} yield
\begin{align}
  \label{p.Jre1}
  \Div\bv_0 = \fre_0.
\end{align}
Thanks to the definition  \eqref{eq.kre} of $\fre_0$,  we equivalently write
\eqref{p.Jre1} as
\begin{align}
  \label{Jre1}
\rre^2 \Jre_1=\sigma \Hre_1\ \ \ \text{on }\Gre.
\end{align}
 Hence, the $C^\infty$-class data $(\rre,\bv_0,\Ore)$ satisfy
the conditions
\eqref{a.st} for $a=0,1$, as verified by \eqref{boundedness.H}, \eqref{Jre0}, \eqref{a.pr.bound} and \eqref{Jre1}.

 \subsubsection{Defining the $C^\infty$-class data $(\vre,\bv_0, \Ore)$
to satisfy \eqref{a.st}  for $a=0,1,2$}\label{sec.poly.vre}
According to \eqref{j.f.rel}, the condition \eqref{a.st.J} for $a=2$
may be imposed by prescribing a Dirichlet boundary condition for
$\Delta\vre$. We define 
$\vre \in
 C^\infty (\Ore)$ to be the solution of the polyharmonic problem
 \begin{subequations}
\label{sp.r2}   \begin{alignat}{2}
\label{sp.r2.int}
     \vre- \mur \Delta^3 \vre&= \varphi
     ^{\mur}*\mathcal{E}_{\Ore}(\varrho_0)&\ \ \ \
     &\text{in
     } \Ore,\\
\label{sp.r2.bc3}
-      2 \Delta\vre&= \fre_1  &&\text{on }
\Gre,\\
\label{sp.r2.bc2}
  -2   \frac{\p\vre}{\p \Nre}&=
     \bV_1^3 &&\text{on }
     \Gre,\\
\label{sp.r2.bc1}
     \vre&=\rre &&\text{on
     } \Gre.
   \end{alignat}
 \end{subequations}
The boundary-forcing function $\fre_1$  appearing in the right-hand
side of \eqref{sp.r2.bc3} is defined as
\begin{align}
  \label{fre1}
  \fre_1= -\frac{\sigma \Hre_2}{2 \rre^2} + 2\kre_0^2 - \bset{\Are_1}_r^s \bv_0^r \cp{s}.
\end{align}
Using the boundary conditions of the polyharmonic problem
\eqref{bv1.3}, we may express the right-hand
side of \eqref{fre1} in terms of $\bV_0$, $\bV_1$ and $\phrep_0$,
where 
     \begin{align}
\label{bV1}
       \bV_1=-2D\rre \ \ \ \text{on } \Gre.
     \end{align}
The boundary-forcing function $\bV_1^3$  appearing in
\eqref{sp.r2.bc2} is   the normal component of $\bV_1$
and the function $\rre$ is the solution of the elliptic Dirichlet
problem \eqref{a.pr}.
The boundary conditions \eqref{sp.r2.bc2} and \eqref{sp.r2.bc1} imply
that $\bv_1=-2D\vre$ defined by \eqref{a.st.va} satisfies
\begin{align}
\label{bv1.bV1}
  \bv_1=\bV_1 \ \ \ \text{on }\Gre.
\end{align}
 Polyharmonic regularity provides  for $s\ge 0$ a positive
constant $C$ independent of $\mur$ that
\begin{align*}
  \norm{\sqrt{\mur}
&  \vre}_{H^{s+6}(\Ore)}^2 \\
&\le
C\Bpset{ \norm{\varphi ^{\mur}* \mathcal{E}_{\Ore}(\varrho_0)}_{H^{s}(\Ore)}^2
 +
 \norm{\Delta\vre
 }_{H^{s+3.5} (\Gre)} ^2
+
 \Bnorm{
\frac{\p \vre}{\p \Nre}}_{H^{s+4.5} (\Gre)}^2
+
 \norm{\vre
}_{H^{s+5.5} (\Gre)}^2
}
.
\end{align*}
Recalling Section~\ref{sec:defin-tang-smooth}, the arguments establishing
the lower-bound \eqref{a.pr.bound} provide that
\begin{align}
\label{a.pr.bound.2}
  \vre\ge 2\lmr>0 \ \ \text{in }\overline{ \Ore}.  
\end{align} 
 Similarly, we infer from the arguments given in Section~\ref{sec:a.v1}
that the boundary condition \eqref{sp.r2.bc3} and the identity \eqref{bv1.bV1} establish that
\begin{align}
  \label{Jre2}
  \vre^2 \Jre_2 = \sigma \Hre_2\ \ \ \text{on }\Gre.
\end{align}
Hence, 
    the $C^\infty$-class data $(\vre,\bv_0,\Ore)$, where $\vre$ is the
    solution of \eqref{sp.r2} and $\bv_0$ is given by
    \eqref{bv0}, satisfy
the conditions
\eqref{a.st} for $a=0,1,2$, as verified by \eqref{boundedness.H},
\eqref{Jre0}, \eqref{Jre1},  \eqref{a.pr.bound.2}  and \eqref{Jre2}.

\begin{rem}
According to \eqref{j.f.rel}, the condition \eqref{a.st.J} for $a=3$
may be imposed by prescribing a Dirichlet boundary condition for
$\frac{\p}{\p\Nre}\Delta\bv_0^3$  in a polyharmonic problem for
$\bv_0^3$ satisfying $    \bv_0^3 +\epr \Delta^4\bv_0^3=\varphi^{\epr}*
    \mathcal{E}_{\Ore}(\mathrm{u}_0^ 3)$.   
\end{rem}
 \subsection{The $C^\infty$-class data for the zero surface tension
   limit of \eqref{Euler}}\label{sec:cinfty-class-data.z}

We suppose the  initial data $(\rho_0, u_0, \Omega)$ is of the
optimal regularity stated in Theorem~\ref{thm.z.main}  and satisfy the conditions
\eqref{l.boundedness.r0}, \eqref{Taylor} and
\eqref{z.compatibility conditions}. Then setting $\sigma=0$ in the
construction of the
$C^\infty$-class  data of
Section~\ref{sec:cinfty-class-data},
$\Jre_a = \p^a_t (J^{-2})|_{t=0}$ satisfies 
\begin{align*}
  \Jre_0=\beta,\ \ \ 
  \Jre_1=0,\ \ \ 
  \Jre_2
=0 \ \ \text{and} \ \ 
  \Jre_3=0\ \ \ \text{on }\Gre.
\end{align*}
For $\mur>0$, and with $\varrho_0$ defined by \eqref{r0.Ore}, $\rre$
solving \eqref{a.pr},
  we  define $\vre \in
 C^\infty (\Ore)$ to be the solution of the polyharmonic problem
 \begin{subequations}
\label{z.re}   \begin{alignat}{2}
\label{z.re.int}
     \vre- \mur \Delta^5 \vre&= \varphi
     ^{\mur}*\mathcal{E}_{\Ore}(\varrho_0)&\ \ \ \
     &\text{in
     } \Ore,\\
\Delta^2\vre&=\fre_3
&& \text{on }\Gre,
\\
\label{z.r1}
\frac{\p\Delta\vre}{\p\Nre}&=\frac{\p\Delta\rre}{\p\Nre}
&& \text{on }\Gre,
\\
\label{z.re.bc.h1}
-      2 \Delta\vre&= \fre_1  &&\text{on }
\Gre,\\
\label{z.re.bc.N}
  -2   \frac{\p\vre}{\p \Nre}&=
     \bV_1^3 &&\text{on }
     \Gre,\\
\label{z.r2}
     \vre&=\sqrt{\beta} &&\text{on
     } \Gre.
   \end{alignat}
 \end{subequations}
where the boundary-forcing function $\fre_3$ is defined below in
\eqref{z.f3}.
By Section~\ref{sec.poly.vre}, 
\begin{align}
\label{z.boundedness.r} 
  \vre\ge \lmr>0 \ \ \text{in }\overline{ \Ore}.  
\end{align}

According to the definition  \eqref{a.st.va} of $\bv_a$,
\begin{align*}
  \bv_3=-2 \bset{\Are_2}_\cdot^k \vre \cp{k}-4 \bset{\Are_1}_\cdot^k \bpset{\vre
    \Kre_0}\cp{k} -2D\bpset{\vre \Kre_1}\ \ \ \text{in } \Ore,
\end{align*}
where
  $\Kre_a=\p^a_t\bpset{ J ^{-1}(J ^{-1}J_t)}|_{t=0}=\p^a_t \bpset{J^{-1} (A_r^sv^r\cp{s})}|_{t=0}$. 
We compute 
\begin{align*}
  \Jre_4&= 16\kre_0^4  -48 \kre_0^2\kre_1 +12  \kre_1^2+ 16
 \kre_0 \kre_2-2  \kre_3.
\end{align*} 
Setting $\Jre_4=0$ yields an equation for $\kre_3$. Since $\kre_3=\bset{\Are_2}_r^s \bv_0^r \cp{s}+3\bset{\Are_2}_r^s \bv_1^r \cp{s}+3\bset{\Are_1}_r^s \bv_2^r \cp{s}+\jre_3$,
\begin{align*}
\jre_3 = \frac{1}{2}\Bbset{16\kre_0^4  -48 \kre_0^2\kre_1 +12  \kre_1^2+ 16
 \kre_0 \kre_2}
-
\Bbset{ \bset{\Are_2}_r^s \bv_0^r \cp{s}+3\bset{\Are_2}_r^s \bv_1^r \cp{s}+3\bset{\Are_1}_r^s \bv_2^r \cp{s}}.
\end{align*}
On the other hand, the divergence of
$  \bv_3=-2 \bset{\Are_2}_\cdot^k \vre \cp{k}-4 \bset{\Are_1}_\cdot^k \bpset{\vre \Kre_0}\cp{k} -2D\bpset{\vre \Kre_1}$
yields
\begin{align*}
  \Div \bv_3 = \Div\bbset{-2 \bset{\Are_2}_\cdot^k \vre \cp{k}-4
    \bset{\Are_1}_\cdot^k \bpset{\vre \Kre_0}\cp{k}-2D\vre \Kre_1}
  -2\vre\cp{i} \Kre_1\cp{i} - 2 \vre \Delta\Kre_1.
\end{align*}
We solve for $\Delta\Kre_1$ to write
\begin{align*}
 \Delta\Kre_1=\underbrace{
 \frac{1}{2\vre}\Bbset{-\jre_3+   \Div\bbset{-2 \bset{\Are_2}_\cdot^k \vre \cp{k}-4
    \bset{\Are_1}_\cdot^k \bpset{\vre \Kre_0}\cp{k}-2D\vre \Kre_1}
  -2\vre\cp{i} \Kre_1\cp{i} }}_{\JRE_3}.
\end{align*}
Using that $\Kre_1= - \kre_0^2 + \kre_1$ and $\kre_1 = \bset{\Are_1}_r^s
\bv_0^r \cp{s}+\jre_1$, we find that
\begin{align}
\label{D.jre.1}
  \Delta\jre_1= \JRE_3+ \Delta\kre_0^2 - \Delta(\bset{\Are_1}_r^s\bv_0^r \cp{s}).
\end{align}
Since $\bv_1= -2D\vre$ implies that $\jre_1 =-2 \Delta\vre $ we
equivalently write \eqref{D.jre.1} as 
\begin{align}
\label{z.f3}  \Delta^2\vre=\underbrace{
-\frac{1}{2}\bbset{ \JRE_3+ \Delta\kre_0^2 -
  \Delta(\bset{\Are_1}_r^s\bv_0^r \cp{s})}
}_{\fre_3}\ \ \ \text{on }\Gre.  
\end{align}
The quantity $D\Kre_1$ may be expressed using
the boundary conditions \eqref{z.r1}--\eqref{z.r2} and  $D^a\bv_0$,
$a=0,1,2,3$ may be expressed using the boundary conditions of the polyharmonic problems
defining $\bv_0$.
It follows from \eqref{a.st.va} and  \eqref{z.f3}   that 
$\Jre_a = \p^a_t (J^{-2})|_{t=0}$ satisfies
\begin{align}
\label{z.collect}
\Jre_a=\p^a_t\beta\ \ \ \text{on }\Gre\ \text{for $a=0,1,2,3,4$}.
\end{align}

With $0<2\nu\le - \frac{\p \rho_0^2}{\p N}$, we may take $\epr$
sufficiently small so that $\rre$ defined in \eqref{a.pr} satisfies\begin{align*}
0< \nu  \le - \frac{\p \rre^2}{\p \Nre}\ \ \text{on }\Gre.
\end{align*}
We then have that \eqref{bV1}
defining   the vector field $\bV_1\in C^\infty(\Gre)$
   is
   \begin{align*}
     \bV_1=-\frac{\p \rre ^2}{\p \Nre
  } \Nre
 \ \ \text{on } 
 \Gre.
   \end{align*}
The boundary condition  \eqref{z.re.bc.N} of the polyharmonic
problem \eqref{z.re} defining $\vre$ provides that 
\begin{align*}
0< \nu  \le - \frac{\p \vre^2}{\p \Nre}\ \ \text{on }\Gre.
\end{align*}
The conditions  \eqref{z.re.bc.N},  \eqref{z.boundedness.r} and \eqref{z.collect}
 establish that the  $C^\infty$-class initial data  $(\vre,\bv_0,
\Ore)$   satisfy
\begin{subequations}
\label{B.cc}  \begin{alignat}{2}
    \label{boundedness.r.z}
    \vre&\ge \lmr>0 &\ \ &\text{in }\overline{ \Ore},\\
- \frac{\p \vre^2}{\p \Nre}&\ge   \nu>0  &\ \ &\text{on }\Gre,\\
    \label{cc.r.z}
    \vre^2 \Jre_a &= \p^a_t \beta&\ \ & \text{on } \Gre \ \ \text{for }
    a=0,1,2,3,4.
  \end{alignat}
\end{subequations}
The conditions \eqref{B.cc} are analogous statements of the conditions
\eqref{l.boundedness.r0}, \eqref{Taylor} and
\eqref{z.compatibility conditions}.
\begin{rem}According to \eqref{a.st.va}, 
  \begin{align*}
  \Div \bv_4 =\Div\bbset{-2 \bset{\Are_3}_\cdot^k \vre \cp{k}-6
    \bset{\Are_2}_\cdot^k  \bpset{\vre \Kre_0}\cp{k}-2 \bset{\Are_1}_\cdot^k
    \bpset{\vre \Kre_1}\cp{k} -2D\vre \Kre_2} 
- 2 \vre \cp{i} \Kre_2
  \cp{i} \\
- 2 \vre \Delta\Kre_2.
\end{align*} 
Since $\Delta\Kre_2$ is approximately $\Delta^2\Div\bv_0$, the
boundary condition $\Jre_5=0$
may be imposed by prescribing a Dirichlet boundary condition for
$\frac{\p}{\p\Nre}\Delta^2 \bv_0^3$ in a polyharmonic problem for
$\bv_0^3$ satisfying $    \bv_0^3 +\epr \Delta^6 \bv_0^3=\varphi^{\epr}*
    \mathcal{E}_{\Ore}(\mathrm{u}_0^ 3)$.  

The boundary condition $\Jre_6=0$ is obtained by defining a
Dirichlet boundary condition for $\Delta^3 \vre$ in a polyharmonic
problem for $\vre$ satisfying
$     \vre- \mur \Delta^7 \vre= \varphi
     ^{\mur}*\mathcal{E}_{\Ore}(\varrho_0)$.
\end{rem}
\section{Solutions  to the $\mu$-problem \eqref{shp}}
\label{sec:appendix:fp}

In this appendix, we construct a fixed-point solution $\vr$ to the $\mu$-problem
\eqref{shp}:

\begin{prop}
	[Solutions to the
$\mu$-problem] \label{prop.shp}
For        $C^\infty$-class initial data $(\rho_0, u_0,  \Omega )$  satisfying the conditions  \eqref{l.boundedness.r0} and
\eqref{s.compatibility conditions}, and for some
        $T=T_\kappa(\epsilon\mu)>0$, there exists a unique $\vr\in
        L^2(0,T;H^9(\Omega))$ solving the $\mu$-problem  \eqref{shp} on a time-interval
        $[0,T]$, with $\p^a_t \vr\in
L^2(0,T;H^{9-2a}(\Omega))$ for $a=1,2,3,4$, $\vr_{tttt}\in
L^\infty(0,T;L^2(\Omega))$ and $(\vr,\vr_t,\dots,\vr_{tttt} )|_{t=0} =(u_0, \mathbf{v}_1,\dots,\mathbf{v}_4)$. 
\end{prop}
\begin{rem}
  We recall that the initial data $\mathbf{v}_a$, $a=1,2,3,4,$ is
  defined in Section~\ref{sec:comp-cond-mu}.
\end{rem}
\subsection{The functional framework for the fixed-point scheme}
\label{sec:funct-fram-fixed}
To establish the existence and uniqueness of solutions to the
$\mu$-problem \eqref{shp}, we define for $T>0$ the Hilbert space\label{n:XT} 
\begin{align*}
  \begin{aligned}[b]
	\XT&= \set{v\in L^2\pset{0,T;H^{5}(\Omega)}\mid 
          \p^a_tv\in L^2(0,T;H^{5-2a}(\Omega)) \text{ for } a=1,2},
      \end{aligned}
\end{align*}
endowed with the natural Hilbert norm
\begin{align*}
	\norm{ v}_{\XT}^2&=\sum_{a=0}^2\norm{\p^a_t
          v}^2_{L^2(0,T;H^{5-2a}(\Omega))}  .
\end{align*} 
For $\MB>0$ (where the particular value of $\MB$ is
specified later), we define the  closed, bounded, convex subset  $\CT$
of  $\XT \label{n:CT}$ 
\begin{align}
v\in\CT\subset \XT 
\label{space: CT} 
\end{align}
to be all $v\in \XT$ that satisfy each of the following conditions:  
\begin{enumerate}[label=(X\arabic*)]
\item \label{CT.va} $(v,v_t,v_{tt})|_{t=0} =(u_0, \mathbf{v}_1,\mathbf{v}_2)$, and
\item \label{CT.M} $  \norm{ v}_{\XT}^2
\le \MB$.
\end{enumerate}

\begin{lem}
	[Solutions to the $\mu$-problem in $\XT$] \label{lem.shp}   For
        $C^\infty$-class initial data $(\rho_0, u_0,\Omega )$  satisfying  the conditions \eqref{l.boundedness.r0} and
\eqref{s.compatibility conditions}, and for some
        $T=T_\kappa(\epsilon\mu)>0$, there exists a unique $\vr\in
        \XT$ that solves the $\mu$-problem  \eqref{shp} on a time-interval
        $[0,T]$ and verifies $(\vr,\vr_t,\vr_{tt})|_{t=0} =(u_0, \mathbf{v}_1,\mathbf{v}_2)$. 
\end{lem}

\subsection{Linearizing   the $\mu$-problem  \eqref{shp}}
\label{sec:line-kappa-epsil}  
Letting $\vb\in \CT$, we set\label{n.vb}
$\bar{\eta}= e+\int_{0}^{t}\vb$ and
$\bar{\eta}_{\epsilon}  =e+\int_0^t{\vb}_{\epsilon} $ on $\Gamma$. We define
 $\bar{\zeta}_{\epsilon}$   to be the solution of the following time-dependent
elliptic Dirichlet problem:
\begin{subequations}
\label{elliptic.ebm}
  \begin{alignat}{2}
    \label{ebm.m}
  \Delta\bar{\zeta}_{\epsilon}&=\Delta\bar{\eta}&& \text{in } \Omega,
\\
\bar{\zeta}_{\epsilon}&=\bar{\eta}_{\epsilon}&\ \ \ \ &\text{on }\Gamma.
  \end{alignat}
\end{subequations}
We define the following  $\epsilon$-approximate Lagrangian variables:
\begin{align*} 
	\Abm  =\bset{D\bar{\zeta}_{\epsilon}  }^{-1},\ \ \Jbm  =\det
        D\bar{\zeta}_{\epsilon}  ,\ \ \abm =\Jbm \Abm  ,\ \gbmu_{\alpha\beta}=\bar{\zeta}_{\epsilon} \cp{\alpha}\cdot \bar{\zeta}_{\epsilon} \cp{\beta}, \ \text{and}\  \sqrt{\gbm } \nbm  =\bset{\abm }^TN.
\end{align*}
We assume that $T>0$ is given such that independently of the choice of
$\vb \in \CT$, the $\epsilon$-approximate Lagrangian map $\bar{\zeta}_{\epsilon} $ is injective for all $t\in [0,T]$, and that 
\begin{align}
 	\label{assumption: J} 
\frac{1}{2}\le \Jb(t) \le \frac{3}{2}\ \ \text{and}\ \ \frac{1}{2}\le \Jb_{\epsilon}(t) \le \frac{3}{2}\ \ \text{for all $t\in [0,T]$ and $x\in \overline\Omega$.} 
\end{align}
This is possible by the inequality \eqref{universal constant} as $\vb\in
\CT$ satisfies $\norm{\vb}_{\XT}^2\le \MB$.

  \begin{defn}[The  system  of linear heat-equations for
    $v$] \label{defn.lv}
For $\vb\in\CT$,  $\kappa>0$, $\epsilon>0  $ and $\mu>0$
given, we   define $v$ to be the
    solution of the 
system of linear   equations
\begin{subequations}
		\label{lv} 
		\begin{alignat}
			{4} \label{lv.momentum}     
v_t -\rb  \Abmu_r^j\bpset{\Abmu_r^k v
      \cp{k}}\cp{j}&=\Kb&\ &\text{in }
      \Omega\times(0,T_\kappa(\epsilon\mu)],\\
\label{lv.bc}
  \rb  N^j\Abmu_r^j
 \Abmu_r^s v \cp{s}
&=
\hbe+\crr
&\ \ &\text{on } \Gamma\times(0,T_\kappa(\epsilon\mu)],\\ 			\label{lv.ic} v|_{t=0}&=u_0&&on\ \Omega.
		\end{alignat}
	\end{subequations}    
The bounded, nonnegative function $\rb$ is defined as
\begin{align}
\label{lv.rb}
\rb   = \kappa\rho_0\Jbm. 
\end{align} The vector field $\Kb$ appearing in the right-hand
side of \eqref{lv.momentum} is given by
  \begin{align}
    \label{lv.kcal}
    \begin{aligned}[b]
      \Kb=\rb 
      \curl_{\ebm}\pset{\curl
          u_0+\varepsilon_{\cdot
            ji}\int_0^t\p_t\Abmu_j^s \vb^i\cp{s}} 
+ \Div_{\ebm} \vb \Abmu_\cdot^k\rb\cp{k}-2 \Abmu_\cdot^k(\rho_0\Jb
   ^{-1})\cp{k}.
    \end{aligned} 
  \end{align}
  The vector field $\hbe$  appearing in the right-hand side of
  \eqref{lv.bc} is given by 
  \begin{align}
    \label{lv.h}
    \begin{aligned}[b]
     \hbe=  \hb_{\curl}
+
 \hb_{\Div}
 +\sum_{\local=1}^K \sqrt{\xl} \Bpset{\LM\Bbset{\gbfk
    ^{\alpha\beta}\LM\bbset{\pset{\sqrt{\xl}\vb\cp{\alpha\beta}\cdot \nbm}\circ
  \thetal}}}\circ\thetal^{-1}\nbm
 +
\rb\Bbset{\berr^{\mu}_{\curl}+\berr^{\mu}_{\Div}},
   \end{aligned}
 \end{align}
 where 
 \begin{align}
   \label{lv.gbfk}
\gbfk ^{\alpha\beta}=  \kappa\frac{\sqrt{\gbm}
}{\rho_0 }\gbm^{\alpha\beta}\circ\thetal,
 \end{align}
and the vectors   $\hb_{\curl}$, $\hb_{\Div}$,
 $\berr^{\mu}_{\curl}$ and $\berr^{\mu}_{\Div}$ are defined as 
 \begin{subequations}
\label{hb.summands}
   \begin{align}
     \hb_{\curl}&= \rb \sqrt{\gbm}(\Jbm)^{-1} {\bpset{\curl
         u_0+\varepsilon_{\cdot ji}\int_0^t\p_t \Abmu_j^s
         \vb^i\cp{s}}\times \nbm},
     \\
     \hb_{\Div}&=  \frac{\sqrt{\gbm}
}{\rho_0}\Bbset{ \rho_0^2
        (\Jbm)^{-2} - \beta_{\epsilon}(t) +\sigma { \gbm
         ^{\alpha\beta}  {\ebm\cp{\alpha\beta}\cdot \nbm}}
     }\nbm,
     \\
     \berr^{\mu}_{\curl} &= \frac{\sqrt{\gbm}}{\Jbm} \Bbset{
 \sum_{\local=1}^K      \gbm^{\alpha\beta}\sqrt{\xl}\LM[ (\sqrt{\xl}\vb\cdot
       \nbm)\cp{\beta} \circ\thetal]\circ\thetal ^{-1}+ \vb^\gamma
       \gbm^{\gamma\delta}\gbm^{\alpha\beta}\ebm \cp{\delta\beta} \cdot
       \nbm}\ebm\cp{\alpha},\\
     \berr^{\mu}_{\Div}&= -\sum_{\local=1}^K \sqrt{\xl}\Bpset{\LM\Bbset{\sqrt{\xl}\frac{\sqrt{\gbm}}{\Jbm} \bbset{
       \gbm^{\alpha\beta} \vb^ \alpha\cp{\beta} + \vb ^{\alpha}
       (\sqrt{\gbm}\gbm^{\alpha\beta})\cp{\beta} + \vb^3H(\ebm)}
     \circ\thetal} }\circ\thetal ^{-1}\nbm.
   \end{align}
 \end{subequations}
  The vector field $\crr$  appearing in the right-hand side of
  \eqref{lv.bc} is given by \eqref{shp.c}.
  \end{defn}

\subsection{Implementation of  the fixed-point scheme to solve the
  $\mu$-problem \eqref{shp}}
\label{sec:implementation-fixed}

We will
allow constants to depend on $1/\delta\kappa\epsilon\mu$ in our fixed-point scheme.
\begin{defn}
	[Notational convention for constants depending on $1/
        \delta\kappa\epsilon\mu  >0$] \label{n:dk} 
Given $\vb\in \CT$,
we let $\PB\label{n:eB}$ denote a generic polynomial with constant and
coefficients depending on $1/\delta\kappa\epsilon\mu  >0$.
	We define the constant $\NB\label{n:bbN0}>0$ by 
	\begin{align*}
          \NB=\PB\pset{\norm{u_0}_{100}, \norm{\rho_0}_{100}}. 
	\end{align*}
		We let $\RB\label{n:bbR}$  denote generic lower-order terms satisfying 
	\begin{align*}
		\int_0^T\RB\le \NB + \delta
\norm{\vb}_{\XT}^2
+T\,\PB (\norm{\vb}_{\XT}^2 
)               . 
	\end{align*}
\end{defn}
\begin{lem} 
\label{lem.lv.1}Given $\vb\in \CT$, $\kappa>0$, $\epsilon>0$ and $\mu>0$,
there exists a unique $v\in \XT$ solving \eqref{lv} and satisfying {\rm
\ref{CT.va}}. Furthermore, 
  \begin{align}
\label{lv.est}
  \norm{v}_{\XT}^2 \le  \int_{0}^{T}\RB.
\end{align}
\end{lem}
\begin{proof}
  Standard parabolic theory provides for the existence and uniqueness of
  $v\in \XT$ solving \eqref{lv} and satisfying \ref{CT.va}.
For the purpose of establishing the estimate \eqref{lv.est}, it is useful to note   the scaling relations 
\begin{align*}
\Kb\sim D\vb + \int_{0}^{t}D^2 \vb \ \ \ \text{and}\ \ \
  \hbe \sim \vb + \int_{0}^{t}D\vb + \LM \hd^2 \vb + \hb _{\epsilon},
\end{align*}
where $\hb _{\epsilon}\in C^\infty(\Gamma)$ with $\abs{\hb
  _{\epsilon}}_s\le C _{\epsilon}\abs{\vb}_s$ thanks to the inequality \eqref{hc.est}.
We will establish the inequality \eqref{lv.est} in the following three steps:

\subsection*{Step~1: Parabolic estimates for $v_{tt}$}
\label{sec:step-1:-parabolic}
Testing two time-derivatives of the equations \eqref{lv.momentum}
against $v_{tt}$ in the $L^2(\Omega)$-inner  product and integrating
by parts with respect to $\p_j$ in the integral $- \int_{\Omega}  \p^2_t\bpset{\rb\Abmu_r^j\bpset{\Abmu_r^k v
      \cp{k}}\cp{j}}v_{tt}$ yields
  \begin{align}
    \label{lv.est.0}
    \frac{1}{2}\frac{d}{dt}\int_{\Omega} \abs{
 v_{tt}}^2 + \int_{\Omega}\rb\abs{\Abmu_\cdot^k v_{tt}
      \cp{k} }^2 = \underbrace{\int_{\Omega} \Kb_{tt}v_{tt}}_{\RB}+ \underbrace{\int_{\Gamma}
    \hbe_{tt}v_{tt}}_{\II}+ \RB.
  \end{align}
We   used the boundary condition \eqref{lv.bc} in writing the
boundary integral $\II$. Thanks to Lemma~\ref{lem.j}, we have that
$\abs{v_{tt}}_0^2=\RB$. Hence, $\II=\RB$. Taking the
time-integral of \eqref{lv.est.0} and using that $\rb\ge
\lambda_{\epsilon}>0$,
\begin{align}
  \label{lv.est.v2}
  \Sup \norm{v_{tt}(t)}_0^2 + \int_{0}^{T}\norm{v_{tt}}_1^2 \le \int_{0}^{T}\RB.
\end{align}

\subsection*{Step~2: Elliptic estimates for $v_t$}
\label{sec:step-2:-elliptic}
A time-derivative of the equations \eqref{lv} yields the following
linear 
Neumann-type elliptic problem for $v_t$:
\begin{subequations}
\label{lvt}
		\begin{alignat}
			{4} \label{lvt.momentum}     
-\rb  \Abmu_r^j\bpset{\Abmu_r^k v_t
      \cp{k}}\cp{j}&=\Gb_1
 &\ &\text{in } 
      \Omega ,\\
\label{lvt.bc}
  \rb  N^j\Abmu_r^j
 \Abmu_r^s v_t \cp{s}
&=\jb_1
& &\text{on } \Gamma ,
		\end{alignat}
	\end{subequations}
where the forcing functions $\Gb_1$ and $\jb_1$ are defined as
\begin{align*}
  \Gb_1&=\Kb_t-v_{tt}+\pset{\rb  \Abmu_r^j}_t\bpset{\Abmu_r^k v
      \cp{k}}\cp{j}
+ \rb  \Abmu_r^j\bpset{\p_t\Abmu_r^k v
      \cp{k}}\cp{j},
\\
\jb_1&=
\p_t\bbset{\hbe+\crr}- \bpset{ \rb  N^j\Abmu_r^j
 \Abmu_r^s}_t v \cp{s}.
\end{align*}
Since $\rb  \Abmu_r^j \p_j [\Abmu_r^k \p_k]
$ is a uniformly elliptic operator,  standard elliptic regularity
theory provides that for $s\ge 2$
\begin{align}
\label{lvt.est.ellip}
  \norm{v_t}_s^2 \le C\bpset{\norm{v_t}_0^2 + \norm{\Gb_1}_{s-2}^2 +
    \abs{\jb_1}_{s-\frac{3}{2}} ^2},
\end{align}
with the constant $C$ depending on the coefficients of the elliptic
and Neumann-type operators. The fundamental theorem of calculus
provides that
\begin{align*}
  \norm{v_t}_0^2 \le \NB + C\,T \Sup \norm{v_{tt}(t)}_0^2\le \int_{0}^{T}\RB.
\end{align*}
Thanks to the estimate \eqref{lv.est.v2}, we conclude that
\begin{align*}
  \Sup \norm{\Gb_1(t)}_0^2 + \int_{0}^{T}\norm{\Gb_1}_1^2 \le \int_{0}^{T}\RB. 
\end{align*}
Since $\jb_1$ scales like $Dv +D\vb+\vb_t$, the fundamental theorem of
calculus provides that 
\begin{align*}
  \Sup \abs{\jb_1(t)}_{0.5}^2 + \int_{0}^{T}\abs{\jb_1}_{1.5}^2 \le
  \int_{0}^{T}\RB+ \Sup \norm{\vb_t(t)}_1^2 .
\end{align*}
Since interpolation and Young's inequality provide that
$\norm{\vb_t}_1^2\le C_ \delta\norm{\vb_t}_0^2+
\delta\norm{\vb_t}_2^2$, we conclude by use of the fundamental theorem
of calculus that
\begin{align*}
  \Sup \abs{\jb_1(t)}_{0.5}^2 + \int_{0}^{T}\abs{\jb_1}_{1.5}^2 \le
  \int_{0}^{T}\RB.
\end{align*}
It therefore follows from the inequality \eqref{lvt.est.ellip} that 
\begin{align}
  \label{lv.est.v1}
  \Sup \norm{v_{t}(t)}_2^2 + \int_{0}^{T}\norm{v_{t}}_3^2 \le \int_{0}^{T}\RB.
\end{align}

\subsection*{Step~3: Concluding the proof of Lemma \ref{lem.lv.1}}
\label{sec:step-3:-concluding}
By repeating Step~2, we infer the following elliptic estimate for $v$:
\begin{align}
  \label{lv.est.v}
  \Sup \norm{v(t)}_4^2 + \int_{0}^{T}\norm{v}_5^2 \le \int_{0}^{T}\RB.
\end{align}
The sum of the inequalities \eqref{lv.est.v2}, \eqref{lv.est.v1} and
\eqref{lv.est.v} establishes the inequality \eqref{lv.est}. 
\end{proof}

\begin{lem}\label{lem.fp}
Let $v$ be the solution of \eqref{lv} determined by $\vb\in \CT$. The solutions map
  \begin{align*}
    \mathfrak{S}\colon \CT\to \CT\colon \vb\mapsto v
  \end{align*}
is a well-defined map for some $T=T_\kappa(\epsilon\mu)$ and
possesses a unique fixed-point.
\end{lem}
\begin{proof}
We recall that $v\in \XT$ is unique and satisfies \ref{CT.va} by
Lemma~\ref{lem.lv.1} .  Setting $\MB=\NB+1$, we
take $\delta$   so that 
  $\delta 
    \norm{\vb}_{\XT}^2<\frac{1}{2}$ and let
  $T=T_\kappa(\epsilon\mu)$ be sufficiently small so that
the inequality given in Lemma~\ref{lem.lv.1} reads
\begin{align*} 
  \norm{v}_{\XT}^2 \le \MB.
\end{align*}
Hence, for some
$T=T_{\kappa}(\epsilon\mu) >0$, the solutions 
map  $\mathfrak{S}$ is well-defined. 

Letting $\vb_l\in \CT$ for $l=1,2$, we set
$\bar{\eta}_l= e+\int_{0}^{t}\vb_l$ in $\Omega$ and
$\bar{\eta}_{\epsilon l}  =e+\int_0^t{\vb}_{\epsilon l} $ on
$\Gamma$. We define
$\bar{\zeta}_{\epsilon l}$   to be the solution of the following time-dependent
elliptic Dirichlet problem:
\begin{subequations}
\label{elliptic.eblm}
  \begin{alignat}{2}
    \label{eblm.m}
  \Delta\bar{\zeta}_{\epsilon l}&=\Delta\bar{\eta}_l&& \text{in } \Omega,
\\
\bar{\zeta}_{\epsilon l}&=\bar{\eta}_{\epsilon l}&\ \ \ \ &\text{on }\Gamma.
  \end{alignat}
\end{subequations}
We define the following  $\epsilon$-approximate Lagrangian variables:
\begin{align*} 
	\Abm{}_l  =\bset{D\ebm{}_l  }^{-1},\;\Jb_l  =\det
        D\ebm{}_l  ,\; \abm{}_l =\Jbm{}_l \Abm{}_l
        ,\;\bset{\gbm{}_l}_{\alpha\beta}=\ebm{}_l \cp{\alpha}\cdot
        \ebm{}_l \cp{\beta},\;  \sqrt{\gbm{}_l } \nbm{}_l  =\bset{\abm{}_l }^TN.
\end{align*}
For $l=1,2,$ we also take the quantities $\rb_l$, $\Kb_l$, $\hbe_l$ to be respectively formed via the definitions \eqref{lv.rb}, \eqref{lv.kcal},
\eqref{lv.h} with $\vb=\vb_l$.  Letting $v_l$ denote the solution of \eqref{lv} formed using $\vb=\vb_l$,
we have that the difference
\begin{align*}
w=v_1-v_2  
\end{align*}
 satisfies
\begin{subequations}
		\label{lw} 
		\begin{alignat}
			{4} \label{lw.momentum}     
w_t -\rb_1  \bset{\Abm{}_1}_r^j\bpset{\bset{\Abm{}_1}_r^k w
      \cp{k}}\cp{j}&=F&\ \ \  &\text{in }
      \Omega\times(0,T_\kappa(\epsilon\mu)],\\
\label{lw.bc}
\begin{aligned}[b]
  \rb_1  N^j\bset{\Abm{}_1}_r^j
 \bset{\Abm{}_1}_r^s w \cp{s}
\end{aligned}
&=
G
&\ \ &\text{on } \Gamma\times(0,T_\kappa(\epsilon\mu)],\\
			\label{lvw.ic} w|_{t=0}&=0&&\text{on }\Omega,
		\end{alignat}
	\end{subequations}    
where the vector field $F$ appearing in the right-hand side of
\eqref{lw.momentum} is given by
\begin{align}
  \label{lw.F}
  F=\Kb_1-\Kb_2 +\underbrace{
\rb_1  \bset{\Abm{}_1}_r^j\bpset{\bset{\Abm{}_1}_r^k v_2
      \cp{k}}\cp{j}
-\rb_2  \bset{\Abm{}_2}_r^j\bpset{\bset{\Abm{}_2}_r^k v_2
      \cp{k}}\cp{j} 
}_{\mathscr{J}},
\end{align}
and the vector field $G$ appearing in the right-hand side of
\eqref{lw.bc} is given by
\begin{align}
  \label{lw.G}
  G=\hbe_1-\hbe_2- \underbrace{
\bbset{
\rb_1 N^j  \bset{\Abm{}_1}_r^j\bset{\Abm{}_1}_r^k v_2
      \cp{k}
+\rb_2 N^j  \bset{\Abm{}_2}_r^j\bset{\Abm{}_2}_r^k v_2
      \cp{k}
}}_{\j }.  
\end{align}

Testing two time-derivatives of \eqref{lw.momentum} against $w_{tt}$
in the $L^2(\Omega)$-inner product and integrating by parts in the integrals $-\int_{\Omega}\p^2_t\bpset{\rb_1
  \bset{\Abm{}_1}_r^j\bpset{\bset{\Abm{}_1}_r^k w \cp{k}}\cp{j}}w_{tt}$ and
$\int_{\Omega} \mathscr{J}_{tt}w_{tt}$ yields
\begin{align}
  \label{tw.tt}
  \begin{aligned}[b]
&    \frac{1}{2}\frac{d}{dt}\int_{\Omega} \abs{w_{tt}}^2+
    \int_{\Omega}\p^2_t\bpset{\rb_1 \bset{\Abm{}_1}_r^j\bset{\Abm{}_1}_r^k w
      \cp{k}}w_{tt}\cp{j} + \int_{\Omega}\p^2_t\bpset{(\rb_1 \bset{\Abm{}_1}_r^j) \cp{j} \bset{\Abm{}_1}_r^k w
      \cp{k}}w_{tt}
\\
&\qquad\qquad \qquad\qquad\ \ \ - \int_{\Gamma} 
 ( \hbe_1-\hbe_2)_{tt}w_{tt}
\\
&= 
\int_{\Omega} \p^2_t\Bbset{\Kb_1-\Kb_2
-
\Bpset{(\rb_1  \bset{\Abm{}_1}_r^j) \cp{j} \bset{\Abm{}_1}_r^k 
-(\rb_2  \bset{\Abm{}_2}_r^j) \cp{j} \bset{\Abm{}_2}_r^k }v_2
      \cp{k} }w_{tt}
\\
&
\qquad\qquad\qquad\qquad\quad -\int_{\Omega} \p^2_t\bbset{ \bpset{\rb_1  \bset{\Abm{}_1}_r^j\bset{\Abm{}_1}_r^k 
-\rb_2  \bset{\Abm{}_2}_r^j\bset{\Abm{}_2}_r^k }v_2
      \cp{k} }w_{tt}\cp{j}.
  \end{aligned}
\end{align}
We have used that two time-derivatives of the boundary condition
\eqref{lw.bc} is equivalent to
\begin{align*}
\p^2_t\bbset{
\rb_1  N^j\bset{\Abm{}_1}_r^j
 \bset{\Abm{}_1}_r^s w \cp{s} +\j
}
= ( \hbe_1-\hbe_2)_{tt}.
\end{align*}
Since $\p^2_t(\rb_1   
-\rb_2 ) =\int_{0}^{t}\p^3_t(\rb_1-\rb_2)$
 scales like
$\int_{0}^{T}D(\vb_1-\vb_2)_{tt}$, we  infer that    the time-integral
of \eqref{tw.tt} yields 
\begin{align}
\label{tw.tt.v1v2}
\begin{aligned}[b]
\Sup \norm{w_{tt}(t)}_{0}^2+ \int_{0}^{T}\norm{w_{tt}}_1^2  \le \delta \int_{0}^{T}   \norm{\vb_1-\vb_2}_{\XT}^2
  + C_{\delta\kappa\mu}\, T  \norm{\vb_1-\vb_2}_{\XT}^2.
\end{aligned}
\end{align}
We use the estimate  \eqref{tw.tt.v1v2} and follow Steps~2 and 3 of the proof of Lemma~\ref{lem.lv.1} to
obtain   estimates for $w_t$ in $L^2(0,T; H^3(\Omega))$ and  $w$ in
$L^2(0,T;H^5(\Omega))$.  Thus,
\begin{align}
  \label{w.XT.v1v2}
\begin{aligned}[b]
 \norm{w}_{\XT}^2 \le \delta \int_{0}^{T}   \norm{\vb_1-\vb_2}_{\XT}^2
  + C_{\delta\kappa\mu}\, T  \norm{\vb_1-\vb_2}_{\XT}^2.
\end{aligned}
\end{align}
We recall that $w=v_1-v_2$. The inequality \eqref{w.XT.v1v2} therefore
provides that
$\mathfrak{S}$ is a contraction mapping for
sufficiently small $\delta>0$ and
$T=T_\kappa(\epsilon\mu)>0$. Hence, the solutions mapping
$\mathfrak{S}$ possess  a unique fixed-point $\vr\in \CT$.
\end{proof}

\subsection{The proof of Lemma~\ref{lem.shp}}

The proof of Lemma~\ref{lem.shp} is complete since
the
fixed-point $\vr\in \CT$ of  Lemma~\ref{lem.fp} verifies \ref{CT.va} and  is
the unique solution of the $\mu$-problem  \eqref{shp}.
\subsection{The proof of Proposition~\ref{prop.shp}}

We infer from the proof of Lemma~\ref{lem.shp} that   considering two more time-derivatives in the functional framework
of our fixed-point scheme
 establishes Proposition~\ref{prop.shp}.

\section{Equivalence of the $\kappa\epsilon$-problem \eqref{mkp} and
  the heat-type $\kappa\epsilon$-problem \eqref{hp}}
\label{sec:appendix:equiv}
 
In this appendix, we  prove the following
\begin{lem}\label{lem.equiv}
The $\kappa\epsilon$-problem
\eqref{mkp}  and the  heat-type $\kappa\epsilon$-problem  \eqref{hp} are equivalent.
\end{lem}
\begin{proof}
 Sections 
\ref{sec:rew1} and \ref{sec:rew2} provide that a solution of the $\kappa\epsilon$-problem
\eqref{mkp}  satisfies the heat-type $\kappa\epsilon$-problem  \eqref{hp}.  We now establish  the converse. Using the   identity
  \eqref{HH.Lagrangian} stating $  -\Amu_r^j[\Amu_r^k \vc \cp{k}]\cp{j}= \curl_{\ec}\curl_{\ec} \vc - \Amu_\cdot ^s
(  \Div_{\ec} \vc)\cp{s}
$, the nonlinear heat-type  equations \eqref{hp.momentum} are
  equivalently written as
  \begin{align*}
    \vc_t + \check{\varrho}\curl_{\ec}\curl_{\ec} \vc -\check{\varrho} \Amu_\cdot ^s (\Div_{\ec} \vc)\cp{s} = \Kc.
  \end{align*}
The identity
\begin{align*}
   -\check{\varrho} \Amu_\cdot ^s (\Div_{\ec} \vc)\cp{s}=
 - \Amu_\cdot ^k (\check{\varrho}\Div_{\ec} \vc)\cp{k}+ \Div_{\ec} \vc \Amu_\cdot ^k\check{\varrho}\cp{k},
\end{align*}
and the definition   \eqref{mkp.m.21} of $\Kc$ imply that the equations \eqref{hp.momentum} are
  equivalently written as
  \begin{align}
\label{pre.lve}
    \vc_t  - \Amu_\cdot ^k (\check{\varrho} \Div_{\ec} \vc -2 \rho_0\Jm^{-1})\cp{k}=-   \check{\varrho}
    \curl_{\ec}\bbset{\curl_{\ec} \vc -\pset{\curl u_0+\varepsilon_{\cdot ji}\int_0^t\p_t \Amu_j^s \vc^i\cp{s}} }.
  \end{align}
Using the identity  $\curl _{\ec} \vc= \curl u_0+
\int_{0}^{t}\p_t(\curl_{\ec}\vc)=\curl u_0+
\int_{0}^{t}\p_t(\varepsilon_{\cdot
  ji}\Amu_j^s\vc^i \cp{s})$,
\begin{align}
\label{i.c.vt}
\int_{0}^{t}\curl_{\ec} \vc_t=  \curl_{\ec} \vc -\pset{\curl u_0+\varepsilon_{\cdot ji}\int_0^t\p_t \Amu_j^s \vc^i\cp{s}}.
\end{align}
The identity \eqref{i.c.vt} implies that the equations \eqref{pre.lve} are equivalently written as
  \begin{align}
\label{pre.lve.2}
    \vc_t  - \Amu_\cdot ^k (\check{\varrho} \Div_{\ec} \vc -2 \rho_0\Jm^{-1})\cp{k}=-   \check{\varrho}
    \curl_{\ec}\bbset{\int_{0}^{t}\curl_{\ec} \vc_t}.
  \end{align}
Applying the $\epsilon$-approximate  Lagrangian curl operator $\curl_{\ec}$ to \eqref{pre.lve.2} yields
  \begin{align}
    \label{ch.in.0}
    \curl_{\ec}\vc_t + \curl_{\ec} \bpset{\check{\varrho} \curl_{\ec}
      \int_{0}^{t}\curl_{\ec} \vc_t}=0 \ \ \ \text{in }
    \Omega.
  \end{align}
We recall $c(t)$ defined in \eqref{hp.c}. Using the definition
\eqref{DN.h} of $h$, we equivalently have that
\begin{align*}
  c(t) &=    \sum_{a=0}^2\frac{t^a}{a!} \p^a_t\bset{ \re N^j\Amu_r^j
\Amu_r^s \vc \cp{s} - \check{\varrho} [\err_{\curl}+\err_{\Div}]}
|_{t=0}
\\
&\qquad - \sum_{a=0}^2\frac{t^a}{a!} \p^a_t\bset{ {\hc_{\curl}
    +\hc_{\Div}}+\gfk^{\alpha\beta}  \bset{\vc\cp{\alpha\beta}\cdot \nm}\nm
}
|_{t=0}.
\end{align*}
By the identities \eqref{DN.2} and \eqref{hp.bc.n.1},
 we have that
\begin{align*}
c(t)
&=    \sum_{a=0}^2\frac{t^a}{a!} \p^a_t\bset{ \underbrace{
\re \sqrt{\gm}
  \Jm^{-1}(\curl_{\ec} \vc)\times \nm 
+
 \re \sqrt{\gm}
  \Jm^{-1}(\Div_{\ec}\vc)\nm
}_{ \re N^j\Amu_r^j
\Amu_r^s \vc \cp{s} - \check{\varrho}\err}}
|_{t=0}
\\
&\qquad - \sum_{a=0}^2\frac{t^a}{a!} \p^a_t\bset{ {\hc_{\curl}+
    \underbrace{
\re \Jm^{-1}\sqrt{\gm}( \Div _{\ec} \vc )n
}_{\hc_{\Div}+\gfk^{\alpha\beta}  \bset{v\cp{\alpha\beta}\cdot \nm}\nm}
}
}|_{t=0}\\
&=    \sum_{a=0}^2\frac{t^a}{a!} \p^a_t\bset{ \re \sqrt{\gm}
  \Jm^{-1}(\curl_{\ec} \vc)\times \nm -\hc_{\curl}
}
|_{t=0}.
\end{align*}
Recalling that $  \hc_{\curl}= \check{\varrho}
  \sqrt{\gm}\Jm^{-1} \bpset{\curl u_0+\varepsilon_{\cdot ji}\int_0^t\p_t
    \Amu_j^s \vc^i\cp{s}}\times \nm$,
we conclude that
\begin{align*}
  c(t) =       \sum_{a=0}^2\frac{t^a}{a!} \p^a_t\bset{  \re \sqrt{\gm} \Jm
  ^{-1}(\curl_{\ec} \vc-\curl u_0-\varepsilon_{\cdot ji}\int_0^t\p_t
    \Amu_j^s \vc^i\cp{s})\times \nm 
}|_{t=0}.
\end{align*}
By the
  identity \eqref{i.c.vt} for $\int_{0}^{t}\curl_{\ec} \vc_t$, we have
  established that
\begin{align}\label{f.c}
  c(t) =       \sum_{a=0}^2\frac{t^a}{a!} \p^a_t\bset{  \re \sqrt{\gm} \Jm
  ^{-1}( \int_{0}^{t}\curl_{\ec} \vc_t)\times \nm 
}|_{t=0}.
\end{align}
We will now verify that $c(t)=0$. With $(\int_{0}^{t}\curl_{\ec} \vc_t)|_{t=0}=0$, we  equivalently write \eqref{f.c} as
\begin{align*}
  c(t) = t\kappa\rho_0\sqrt{\gm}
\curl \mathbf{v}_1\times N + \frac{t^2}{2}\kappa\rho_0\sqrt{\gm}[\curl \mathbf{v}_2+\varepsilon_{\cdot ji}\p_t \Amu_j^s(0) \mathbf{v}_1^i \cp{s} ]\times N . 
\end{align*}
 According to \eqref{defn.va.bf}, 
we have that
  $\mathbf{v}_1= D(\kappa\rho_0\Div u_0 - 2 \rho_0)$. Hence,  $\curl
  \mathbf{v}_1=0$. Using the identity $\p_t\Amu_i^k|_{t=0}=- (\Amu_i^s \Amu_r^k
  \vc^r\cp{s}\!)|_{t=0}=-u_0^k\cp{j}$, 
we conclude  that
\begin{align*}
  c(t) =   \frac{t^2}{2}\kappa\rho_0 \sqrt{\gm}[\curl \mathbf{v}_2-\varepsilon_{\cdot ji}u_0^s\cp{j} \mathbf{v}_1^i \cp{s} ]\times N . 
\end{align*}
 According to \eqref{defn.va.bf},  $\mathbf{v}_2= D\p_t(\kappa \rho_0 \JT - 2 \rho_0 \Jm^{-1}
  )|_{t=0} + \p_t\Amu_\cdot^k|_{t=0} \p_k(\kappa\rho_0\Div u_0 - 2
  \rho_0)$. Thus,
  \begin{align*}
\curl \mathbf{v}_2&=- \varepsilon_{\cdot ji} 
\bbset{ u_0^k \cp{i} (\kappa\rho_0\Div u_0 - 2
  \rho_0)\cp{k}}\cp{j}=- \varepsilon_{\cdot ji} 
u_0^k \cp{i} (\kappa\rho_0\Div u_0 - 2
  \rho_0)\cp{jk}
=- \varepsilon_{\cdot ji}u_0^k
  \cp{i}\mathbf{v}_1^j \cp{k}.    
  \end{align*}
That is,
\begin{align*}
  \curl \mathbf{v}_2 = \varepsilon_{\cdot ji}u_0^s
  \cp{j}\mathbf{v}_1^i \cp{s},
\end{align*} 
so that the identity $  c(t) =
\frac{t^2}{2}\kappa\rho_0 \sqrt{\gm}[\curl
\mathbf{v}_2-\varepsilon_{\cdot ji}u_0^s\cp{j} \mathrm{v}_1^i \cp{s}
]\times N $ implies that
  \begin{align*}
    c(t)=0.
  \end{align*}
The boundary condition \eqref{hp.bc} is therefore equivalently
  written as 
  \begin{align}
\label{hp.bc.c0}
\begin{aligned}[b]
  \re N^j\Amu_r^j \Amu_r^s \vc \cp{s} &=
\\
\hc_{\curl} +\check{\varrho}\err
&+
 \frac{\sqrt{\gm}}{\rho_0}\Bbset{ \rho_0^2 \Jm ^{-2} - \beta_{\epsilon}(t) +
    \sigma { \gm
          ^{\alpha\beta}  {\ec\cp{\alpha\beta}\cdot \nm}}
  + \gfk ^{\alpha\beta}
  {\vc\cp{\alpha\beta}\cdot \nm}}\nm.
\end{aligned}
  \end{align}
  Taking the scalar product of \eqref{hp.bc.c0} with $\re
  ^{-1}\hd\ec$ yields
  \begin{align*}
\underbrace{\sqrt{\gm} \Jm^{-1}[(\curl_{\ec} \vc)\times
    \nm]\cdot \hd\ec 
}_{    N ^j \Amu_r^j \Amu_r^k \vc \cp{k}\cdot \hd\ec}
=
\underbrace{ 
\sqrt{\gm}\Jm^{-1}[ \bpset{\curl
      u_0+\varepsilon_{\cdot ji}\int_0^t\p_t \Amu_j^s \vc^i\cp{s}}\times
    \nm]\cdot \hd\ec }_{ \check{\varrho} ^{-1} \hc_{\curl} \cdot \hd \ec}.
  \end{align*}
Using that $[(\curl_{\ec} \vc)\times \nm]\cdot \nm=0$ and $\hc_{\curl}\cdot
  \nm=0$, we have that
  \begin{align}
\label{ch.pre.bc}
    (\curl_{\ec} \vc)\times \nm&=\bpset{\curl u_0+\varepsilon_{\cdot
        ji}\int_0^t\p_t \Amu_j^s \vc^i\cp{s}}\times \nm\ \ \ \text{on } \Gamma.  
  \end{align}
We note that \eqref{ch.pre.bc} is
equivalently stated as $    (\int_{0}^{t}\curl_{\ec} \vc_t)\times \nm=0$.
Hence, by the equation \eqref{ch.in.0}, we record that $\int_{0}^{t}\curl_{\ec} \vc_t$ satisfies\begin{subequations}
\label{h.ch}
  \begin{alignat}{2}
\label{ch.in}
    \curl_{\ec}\vc_t + \curl_{\ec} \pset{\check{\varrho} \curl_{\ec}
      \int_{0}^{t}\curl_{\ec} \vc_t}&=0 &\ \ \ &\text{in }
    \Omega\times[0,T],\\
    \label{ch.bc}
    (\int_{0}^{t}\curl_{\ec} \vc_t)\times \nm&=0&\ \ \ &\text{on }
    \Gamma\times [0,T],\\
  \label{ch.ic}
  \pset{\int_{0}^{t}\curl_{\ec}\vc_t}|_{t=0}&=0&\ \ \ &\text{on } \Omega.
  \end{alignat} 
\end{subequations}
Testing \eqref{ch.in} against $\int_{0}^{t}\curl _{\ec} \vc_t$ in the
$L^2(\Omega)$-inner product and integrating by parts with respect to
the operator $\curl_{\ec}$ in 
$\int_{\Omega}\curl_{\ec} \bpset{\check{\varrho} \curl_{\ec}
  \int_{0}^{t}\curl_{\ec} \vc_t}\int_{0}^{t}\curl_{\ec} \vc_t$ yields
\begin{align}
\label{t.ch.in}
\begin{aligned}[b]
\frac{1}{2}\frac{d}{dt} \norm{\int_{0}^{t}\curl_{\ec} \vc_t}_0^2&+
\norm{\sqrt{\check{\varrho}} \curl_{\ec}
  \int_{0}^{t}\curl_{\ec} \vc_t}_0^2
\\
&+ \int_{\Gamma}N^s \Amu_j^s \varepsilon_{ijk}\bpset{\check{\varrho} \curl_{\ec}
  \int_{0}^{t}\curl_{\ec} \vc_t}^k(\int_{0}^{t}\curl_{\ec} \vc_t)^i=0.
\end{aligned}
\end{align}
With the boundary condition
\eqref{ch.bc}, $\int_{\Gamma}N^s \Amu_j^s \varepsilon_{ijk}\pset{\check{\varrho} \curl_{\ec}
  \int_{0}^{t}\curl_{\ec} \vc_t}^k(\int_{0}^{t}\curl_{\ec}
\vc_t)^i=\int_{\Gamma}\sqrt{\gm} \Jm^{-1}[(\int_{0}^{t}\curl_{\ec} \vc_t)\times
n ]\cdot \pset{\check{\varrho} \curl_{\ec}
  \int_{0}^{t}\curl_{\ec} \vc_t}=0$. Integrating the identity
\eqref{t.ch.in} in time from $0$ to $t\in(0,T]$ produces $
\norm{  \int_{0}^{t}\curl_{\ec} \vc_t}_0^2=0
$. So according to \eqref{ch.in},
$\norm{\curl_{\ec} \vc_t}_0=0$. The regularity of $\vc$ stated in
Proposition~\ref{prop.hp} then provides  that almost everywhere in $\Omega$,
\begin{align*}
\curl_{\ec}\vc_t=0.
\end{align*}
The  equations \eqref{pre.lve.2} are thus written as
\begin{subequations}
\label{hp.mkp}
  \begin{align}
    \label{hp.mkp.momentum}
    \vc_t - \Amu_\cdot ^k (\kappa\rho_0\Jm\Div_{\ec} \vc -2 \rho_0\Jm ^{-1})\cp{k}=0.
  \end{align}
  As the solution $\vc$ of the heat-type $\kappa\epsilon$-problem  \eqref{hp}
 verifies the $\epsilon$-approximate Lagrangian vorticity equation
  \eqref{mkp.curl.v}, we infer that the boundary condition \eqref{hp.bc.c0} of the
heat-type problem    is equivalently written
  as
  \begin{align}
    \label{hp.mkp.bc}
    \begin{aligned}[b]
    \kappa  \rho_0^2 \Div _{\ec} \vc= 
  \rho_0^2  \Jm ^{-2} - \beta_{\epsilon}(t) + \sigma { \gm
          ^{\alpha\beta} {\ec\cp{\alpha\beta}\cdot \nm}} 
      +\kappa \gm ^{\alpha\beta}
     \vc\cp{\alpha\beta}\cdot \nm.
    \end{aligned}
  \end{align}
\end{subequations}
The equations \eqref{hp.mkp} are equivalent to the momentum equations and boundary
condition of the $\kappa\epsilon$-problem \eqref{mkp}.
\end{proof}

\newpage

\section*{List of Notation} 
\begin{tabular}
	{clr} \notation{u}{Eulerian velocity}{n:u}\\
	\notation{p}{Eulerian pressure}{n:p}\\
	\notation{\rho}{Eulerian density}{n:rho}\\
	\notation{u_0}{Initial velocity}{n:initial data}\\
	\notation{\rho_0}{Initial density}{n:initial data}\\
	\notation{\sigma}{Surface tension}{n:sigma}\\
	\notation{\beta}{A parameter to the equation of state $p=\alpha\rho^\gamma-\beta$ for $\gamma>1$}{n:beta}\\
	\notation{H(\eta)}{Twice the mean curvature of the moving
          surface $\eta(\Gamma)$}{n:H}\\
	\notation{\eta}{The particle flow map}{n:eta}\\
	\notation{v}{Lagrangian velocity}{n:Lagrangian variables}\\
	\notation{f}{Lagrangian density}{n:Lagrangian variables}\\
	\notation{A}{Inverse of the deformation tensor $D\eta$}{n:Lagrangian variables}\\
	\notation{J}{Jacobian determinant of the deformation tensor $D\eta$}{n:Lagrangian variables}\\
	\notation{a}{Cofactor matrix of the deformation tensor $D\eta$}{n:Lagrangian variables}\\
	\notation{e}{The identity function $e(x)=x$}{n:e}\\
	\notation{D}{The three-dimensional gradient operator}{n:D}\\
	\notation{\hd}{The surface gradient operator}{n:D}\\
\notation{\Div _{ \eta},\curl _{ \eta}}{The Lagrangian divergence and curl operators}{n.div.curl}\\
	\notation{n}{The outward-pointing unit normal to the moving surface $\eta(\Gamma)$}{n:n}\\
	\notation{g}{The surface metric induced by the moving surface $\eta(\Gamma)$}{n:g}\\
	\notation{g_0}{The surface metric of the initial surface $\Gamma$}{n:g0}\\
	\notation{\sqrt{g}}{The square root of the determinant of the
          metric $g$}{n:sqrt g}\\
	\notation{\Delta_g}{The Laplace-Beltrami operator}{n:BL}\\ 
	\notation{\norm{\cdot}_s}{The norm of the Hilbert space $H^s(\Omega)$}{n:interior norm}\\
	\notation{\abs{\cdot}_s}{The  norm of the Hilbert space $H^s(\Gamma)$}{n:boundary-norm}\\
	\notation{N}{The outward-pointing unit normal to $\Gamma$}{n:N}\\
	\notation{\vt}{A solution of the  $\kappa$-problem \eqref{akp}}{n:EK}\\
	\notation{ \et}{The Lagrangian map of the solution $\vt$ to the $\kappa $-problem}{n:EK}\\
	\notation{\EK}{The higher-order energy function for solutions $\vt$ to the $\kappa $-problem}{n:EK}\\
		\notation{\PK,\NK,\RK }{The generic system of constants
used in the $\kappa$-independent estimates}{n:kP}\\
	\notation{\LE}{The horizontal-convolution operator for $\epsilon>0$}{n.ve}\\ 
	\notation{\vc}{A solution of the
          $\kappa\epsilon$-problem \eqref{mkp}}{EM.l}\\
	\notation{\EM}{The higher-order energy function for solutions $\vc$ to the $\kappa\epsilon $-problem}{EM.l}\\
		\notation{\PM,\MM,\RM }{The generic system of constants
used in the $\epsilon$-independent estimates}{EM.l}\\
	\notation{\vr}{A solution of the
          $\mu$-problem \eqref{shp}}{def.EE}\\
	\notation{\EE}{The higher-order energy function for solutions
          $\vr$ to the  $\mu$-problem \eqref{shp}}{def.EE}\\
	\notation{\PE,\ME,\RE }{The generic system of
          constants used in the $\mu$-independent estimates}{n:eP}\\
	\notation{\vz}{A solution of the
$\sigma$-problem \eqref{stp}}{defn:ES}\\
	\notation{\ES}{The higher-order energy function for solutions
          $\vz$ to the surface tension problem}{defn:ES}\\
		\notation{\PS,\MS,\RS }{The generic system of constants
used in the $\sigma$-independent estimates}{defn:ES}\\
	\notation{\XT} 
{The Hilbert space used in the fixed-point scheme} {n:XT}\\ 
	\notation{\CT}{A closed, bounded, convex subset of
          $\XT$}{n:CT}\\
\notation{\vb }{An arbitrary vector  in $\CT$}{n.vb}\\
 	\notation{\PB,\NB,\RB }{The generic system of
          constants used in the fixed-point scheme}{n:dk}\\
\end{tabular}


\newpage
\begin{thebibliography}
	{50}

\bibitem{A75} Adams, R.~A. \textit{Sobolev
          spaces}. Academic Press, 1975.

 \bibitem{AM05} D.~Ambrose and N.~Masmoudi.
The zero surface tension limit of two-dimensional water waves. 
\textit{Comm. Pure Appl. Math.} \textbf{58}(10):1287–-1315, 2005. 

 \bibitem{AM09} D.~Ambrose and N.~Masmoudi. The zero surface tension
   limit of three-dimensional water waves. \textit{Indiana
     Univ. Math. J.}, \textbf{58}:479--521, 2009.

 \bibitem{CW08}
G.-Q.~Chen and Y.-G.~Wang. Existence and stability of compressible current-vortex sheets in three-dimensional magnetohydrodynamics. \textit{Arch. Rational Mech. Anal.} \textbf{187}:369--408, 2008.


\bibitem{CCS07} C.-H.~A.~Cheng, D.~Coutand, and
          S.~Shkoller. Navier-Stokes equations interacting with
          nonlinear elastic biofluid shell. \textit{SIAM
            J. Math. Anal.}, \textbf{39}(3):742--800 (electronic), 2007.


        \bibitem{CSP08}
J.-F.~Coulombel and P.~Secchi. Nonlinear compressible vortex sheets
in two space dimensions. \textit{Ann. Sci. Ecole Norm. Sup.} \textbf{41}:85--139, 2008.
\bibitem{CSP09}
J.-F.~Coulombel and P.~Secchi. Uniqueness of 2-D compressible vortex
sheets. \textit{Commun. Pure Appl. Anal.} \textbf{8}:1439--1450, 2009.
 \bibitem{CF48} R.~Courant and K.~O.~Friedrichs. \textit{Supersonic
     flow and shock waves. Reprinting of the 1948 original. Applied
     Mathematical Sciences, Vol.~21.} Springer-Verlag, New
   York-Heidelberg, 1976. xvi+464 pp. 

\bibitem{CLS09} D.~Coutand, H.~Lindblad and S.~Shkoller. A priori estimates for the free-boundary 3-D compressible Euler equations in physical vacuum, 
\textit{Comm. Math. Phys.}, \textbf{296}(2):559--587, 2010. 
	
	\bibitem{CS06}D.~Coutand and S.~Shkoller. The interaction between Quasilinear Elastodynamics and the Navier-Stokes equations. \textit{Arch. Rational Mech. Anal.}, \textbf{179}:303--352, 2006.
	

	\bibitem{CS07} D.~Coutand and S.~Shkoller. Well-posedness of
          the free-surface incompressible Euler equations with or
          without surface tension. \textit{J. Amer. Math. Soc.},
          \textbf{20}(3):829-930 (electronic), 2007. 


          \bibitem{CS11} D.~Coutand and S.~Shkoller. Well-posedness in smooth function spaces for the moving-boundary 1-{D} compressible Euler equations in physical vacuum. \textit{Comm. Pure Appl. Math.}, \textbf{64}:328--366, 2011.
          
          
          	\bibitem{CS10}D.~Coutand and S.~Shkoller. Well-posedness in smooth function spaces for the moving-boundary 3-{D} compressible Euler equations in physical vacuum. 
	 \textit{Arch. Rational Mech. Anal.},  2012, DOI: 10.1007/s00205-012-0536-1.

		\bibitem{CS12}D.~Coutand and S.~Shkoller.  On the finite-time
          splash and splat singularities for the 3-D free-surface
          Euler equations. arXiv:1201.4919.


\bibitem{DS02} I.~Denisova and V.~Solonnikov. Classical solvability of
  the problem of the motion of an isolated mass of compressible
  fluid. (Russian) 
\textit{Algebra i Analiz} \textbf{14}(1):71--98, 2002. 
	

        \bibitem{FM} J.~Francheteau and G.~M\'{e}tivier. Existence de chocs faibles pour des systèmes quasilin\'{e}aires hyperboliques multidimensionnels. \textit{Ast\'{e}risque} \textbf{268}:1--198,
2000.

\bibitem{GM91}
J.~Glimm and A.~Majda. Multidimensional hyperbolic problems and
computations. \textit{The IMA Volumes in Mathematics and its Applications},
Vol. \textbf{29}. Springer, 1991.

\bibitem{GMWZ}
O.~Gu\'{e}s, G.~M\'{e}tivier, M.~Williams and K.~Zumbrun. Existence
and stability of multidimensional shock fronts in the vanishing
viscosity limit. \textit{Arch. Rational Mech. Anal.}
\textbf{175}:151--244, 2005.

	\bibitem{JM10} J.~Jang and N.~Masmoudi. Well-posedness for
          compressible Euler equations in a physical vacuum
          singularity. \textit{Comm. Pure Appl. Math.}
          \textbf{62}(10):1327-–1385, 2009.
	
\bibitem{JM11} J.~Jang and N.~Masmoudi. 
Vacuum in gas and fluid dynamics.\textit{ Nonlinear conservation laws and applications}, 315–329, 
IMA Vol. Math. Appl., 153, Springer, 2011. 

	\bibitem{D.L05} D.~Lannes. Well-posedness of the water-waves equations. \textit{J. Amer. Math. Soc.}, 18:605--654, 2005.
	
	\bibitem{H.L05} H.~Lindblad. Well-posedness for the motion of an incompressible liquid with free surface boundary. \textit{Annals of Math.}, \textbf{162}:109--194, 2005.
	
	\bibitem{H.L05.1} H.~Lindblad. Well posedness for the motion of a compressible liquid with free surface boundary. \textit{Commun. Math. Phys.} \textbf{260}:319--392, 2005.
	

        \bibitem{M84} A.~Majda. \textit{Compressible Fluid Flow and Systems of Conservation Laws in Several
Space Variables.} Springer, 1984.

\bibitem{M01}
G.~M\'{e}tivier. Stability of multidimensional shocks. Advances in the Theory of Shock
Waves. \textit{Progress in Nonlinear Differential Equations and Their
  Applications}, Vol. 47.
Birkh\"{a}user, 25–103, 2001.
	
 	\bibitem{Ra1878} L. Rayleigh, \emph{ On the instability of jets},  Proceedings of the London Mathematical Society,
\textbf{s1-10}(1):4--13, 1878.
	
	
	\bibitem{SZ08} J.~Shatah and C.~Zeng. Geometry and a priori estimates for free boundary problems of the Euler equation. \textit{Comm. Pure Appl. Math.}, \textbf{61}:698--744, 2008.
	
	\bibitem{ST91} V.~Solonnikov and A.~Tani. \textit{Free
            boundary problem for a viscous compressible flow with a
            surface tension. Constantin Carathéodory: an international
            tribute, Vol. I, II}, 1270--1303, World Sci. Publ.,
          Teaneck, NJ, 1991. 

\bibitem{TT03} N.~Tanaka and A.~Tani. Surface waves for a compressible
  viscous fluid. (English summary) \textit{J. Math. Fluid Mech.},
  \textbf{5}(4):303--363, 2003. 


\bibitem{Taylor1950} G. Taylor, \emph{The instability of liquid surfaces when accelerated in a direction perpendicular
to their planes  I.},  Proc. Roy. Soc. London. Ser. A.,  {\bf
201}:192--196, 1950.

\bibitem{Temam} R.~Temam. \textit{Navier-Stokes equations. Theory and numerical analysis}. North-Holland Publishing Co., Amsterdam, 1977. Studies in Mathematics and its Applications, Vol. 2.
	
\bibitem{T05}
Y.~Trakhinin. Existence of compressible current-vortex sheets: variable coefficients
linear analysis. \textit{Arch. Rational Mech. Anal.} \textbf{177} 331--366, 2005.
 
	
	\bibitem{T09} Y.~Trakhinin. Local existence for the free
          boundary problem for the non-relativistic and relativistic
          compressible Euler equations with a vacuum boundary
          condition. \textit{Comm. Pure Appl. Math}, \textbf{62}:1551--1594, 2009.

        \bibitem{T9}
Y.~Trakhinin. The existence of current-vortex sheets in ideal compressible magnetohydrodynamics. \textit{Arch. Rational Mech. Anal.} \textbf{191}:245--310, 2009.
	
	\bibitem{W97} S.~Wu. Well-posedness in Sobolev spaces of the full water wave problem in 2-D. \textit{ Invent. Math.}, \textbf{130}:39--72, 1997. 
	
	\bibitem{W99} S.~Wu. Well-posedness in Sobolev spaces of the full water wave problem in 3-D. \textit{ J. Amer. Math. Soc.}, \textbf{12}:445--495, 1999.
	
	\bibitem{XY05} C.-J.~Xu and T.~Yang. Local existence with physical vacuum boundary condition to Euler equations with damping. \textit{J. Differential Equations}, \textbf{210}:217--231, 2005.
	
	\bibitem{Z94} W.~M.~Zajaczkowski. On nonstationary motion of a compressible barotropic viscous capillary fluid bounded by a free surface. \textit{SIAM J.~Math.~Anal.}, \textbf{25}(1):1-84, 1994.
	
	\bibitem{ZZ08} P.~Zhang and Z.~Zhang. On the free boundary problem of three-dimensional incompressible Euler equations. \textit{Comm. Pure Appl. Math.}, \textbf{61}:877--940, 2008. 
\end{thebibliography}
\end{document}